\documentclass[reqno]{amsart}
\usepackage{graphicx}
\usepackage{amstext}
\usepackage{amssymb}
\usepackage{amsmath}
\usepackage{color}
\newtheorem{theorem}{Theorem}[section]
\newtheorem{lemma}[theorem]{Lemma}

\newtheorem{proposition}[theorem]{Proposition}
\theoremstyle{definition}

\theoremstyle{remark}
\newtheorem{remark}[theorem]{Remark}
\numberwithin{equation}{section}
\usepackage{hyperref}
\hypersetup{colorlinks,bookmarks=true,linkcolor=blue,citecolor=blue}
\begin{document}

\title{On Asymptotic Dynamics for $L^2$ Critical Generalized KdV equations with a Saturated perturbation}\def\rightmark{DYNAMICS FOR CRITICAL GKDV EQUATIONS}

\author{Yang LAN}

\address{Laboratoire de Mathematiques D'Orsay, Universit\'{e} Paris-Sud, Orsay, France}



\email{yang.lan@math.u-psud.fr}

\begin{abstract}
In this paper, we consider the $L^2$ critical gKdV equation with a saturated perturbation: $\partial_t u+(u_{xx}+u^5-\gamma u|u|^{q-1})_x=0$, where $q>5$ and $0<\gamma\ll1$. For any initial data $u_0\in H^1$, the corresponding solution is always global and bounded in $H^1$. This equation has a family of solitons, and our goal is to classify the dynamics near soliton. Together with a suitable decay assumption, there are only three possibilities: (i) the solution converges asymptotically to a solitary wave, whose $H^1$ norm is of size $\gamma^{-2/(q-1)}$, as $\gamma\rightarrow0$; (ii) the solution is always in a small neighborhood of the modulated family of solitary waves, but blows down at $+\infty$; (iii) the solution leaves any small neighborhood of the modulated family of the solitary waves.
	
This extends the classification of the rigidity dynamics near the ground state for the unperturbed $L^2$ critical gKdV (corresponding to $\gamma=0$) by Martel, Merle and Rapha\"{e}l. However, the blow-down behavior (ii) is completely new, and the dynamics of the saturated equation cannot be viewed as a perturbation of the $L^2$ critical dynamics of the unperturbed equation. This is the first example of classification of the dynamics near ground state for a saturated equation in this context. The cases of $L^2$ critical NLS and $L^2$ supercritical gKdV, where similar classification results are expected, are completely open.
\end{abstract}
\maketitle
\section{Introduction}
\subsection{Setting of the problem}
Let us consider the following Cauchy problem:
\begin{equation}\label{CCPG}\tag{${\rm gKdV}_{\gamma}$}\begin{cases}
\partial_t u +(u_{xx}+u^5-\gamma u|u|^{q-1})_x=0, \quad (t,x)\in[0,T)\times\mathbb{R},\\
u(0,x)=u_0(x)\in H^1(\mathbb{R}),
\end{cases}
\end{equation}
with $q>5$ and $0<\gamma\ll1$.

The equation has two conservation laws, i.e. the mass and the energy:
\begin{gather*}
M(u(t))=\int u(t)^2=M_0,\\
E(u(t))=\frac{1}{2}\int u_x(t)^2-\frac{1}{6}\int u(t)^6+\frac{\gamma}{q+1}\int |u(t)|^{q+1}=E_0.
\end{gather*}
 
We can see that the solution of \eqref{CCPG} is always global in time and bounded in $H^1$. First of all, \eqref{CCPG} is locally wellposed in $H^1$ due to \cite{Ka,KPV}, i.e. for any $u_0\in H^1$, there exists a unique strong solution in $C([0,T),H^1)$, with either $T=+\infty$ or $T<+\infty$ and $\lim_{t\rightarrow T}\|u_x(t)\|_{L^2}=+\infty.$ Since $\gamma>0$, $q>5$, the  mass and energy conservation laws ensure that for all $t\in[0,T)$,
$$\|u_x(t)\|^2_{L^2}\lesssim |E_0|+\gamma^{-\frac{4}{q-5}}M_0<+\infty,$$
so $T=+\infty$ and $u(t)$ is always bounded in $H^1$.

This equation does not have a standard scaling rule, but has the following pseudo-scaling rule: for all $\lambda_0>0$, if $u(t,x)$ is a solution to \eqref{CCPG}, then
\begin{equation}\label{C1}
u_{\lambda_0}(t,x)=\lambda_0^{-\frac{1}{2}}u(\lambda_0^{-3}t,\lambda_0^{-1}x),
\end{equation}
is a solution to
\begin{equation*}\begin{cases}
\partial_t v +(v_{xx}+v^5-\lambda_0^{-m}\gamma v|v|^{q-1})_x=0, \quad (t,x)\in[0,\lambda_0^{-3}T)\times\mathbb{R},\\
v(0,x)=\lambda_0^{-\frac{1}{2}}u_0(\lambda_0^{-1}x)\in H^1(\mathbb{R}),
\end{cases}
\end{equation*}
with 
\begin{equation}
m=\frac{q-5}{2}>0.
\end{equation}
The pseudo-scaling rule \eqref{C1} leaves the $L^2$ norm of the initial data invariant. 

There is a special class of solutions. We first introduce the ground state $\mathcal{Q}_{\omega}$ for $0\leq\omega<\omega^*\ll1$, which is the unique radial nonnegative solution with exponential decay to the following ODE%
\footnote{The existence of such $\mathcal{Q}_{\omega}$ was proved in \cite[Section 6]{BL}, but, in this paper we will give an alternative proof for the existence.}%
:
$$\mathcal{Q}_{\omega}''-\mathcal{Q}_{\omega}+\mathcal{Q}_{\omega}^5-\omega \mathcal{Q}_{\omega}|\mathcal{Q}_{\omega}|^{q-1}=0.$$
Then for all $\lambda_0>0$, $t_0\in\mathbb{R}$, $x_0\in\mathbb{R}$ with $\lambda_0^{-m}\gamma<\omega^*$, the following is a solution to \eqref{CCPG}: 
$$u(t,x)=\lambda_0^{-\frac{1}{2}}\mathcal{Q}_{\lambda_0^{-m}\gamma}\big(\lambda_0^{-1}(x-x_0)-\lambda_0^{-3}(t-t_0)\big).$$
A solution of this type is called a {\it solitary wave} solution.

\subsection{On the critical problem with saturated perturbation}
The saturated perturbation was first introduced for the nonlinear Schr\"{o}dinger (NLS):
\begin{equation}\label{CNLS}\tag{${\rm NLS}$}
i\partial_t u +\Delta u+g(|u|^2)u=0, \quad (t,x)\in[0,T)\times\mathbb{R}^d.
\end{equation}

In many applications, the leading order approximation of the nonlinearity, $g(s)$, is
the power nonlinearity, i.e. $g(s)=\pm s^{\sigma}$. For example, $g(s) = s$, leads to the focusing cubic
NLS equation, which appears in many contexts. 

But such approximation may lead to nonphysical predictions. For example, from \cite{Fibich,MR1,MRS2,SS}, for NLS with critical or supercritical focusing nonlinearities (i.e. $g(s)=s^{\sigma}$ with $\sigma d\geq2$), blow up may occur. However, this contradicts with the experiments in the optical settings \cite{JR}, which shows that there is no ``singularity'' and the solution always remains bounded.

One way to correct this model is to replace the power nonlinearities by saturated nonlinearities. A typical example%
\footnote{See \cite{GA,0MRS} for other kind of saturated perturbations.}
 is $g(s)=s^{\sigma}-\gamma s^{\sigma+\delta}$, with $\delta>0$, $\gamma>0$. Similar as \eqref{CCPG}, in this case any $H^1$ solution to \eqref{CNLS} is global in time and bounded in $H^1$.

On the other hand, the saturated perturbation is also related to the problem of continuation after blow up time. This kind of problems arising in physics is poorly understood even at a formal level. One approach is to consider the solution $u_{\varepsilon}(t)$ to the following critical NLS with saturated perturbation:
$$\begin{cases}i\partial_t u +\Delta u+|u|^{\frac{4}{d}}u-\varepsilon|u|^{q}u=0, \quad (t,x)\in[0,T)\times\mathbb{R}^d,\\
u(0,x)=u_0(x)\in H^1(\mathbb{R}^d),
\end{cases}$$
where $4/d<q<4/(d-2)$. Suppose the solution $u(t)$ to the unperturbed NLS (i.e. $\varepsilon=0$) with initial data $u_0$, blows up in finite time $T<+\infty$. Then, it is easy to see that for all $\varepsilon>0$, $u_{\varepsilon}(t)$ exists globally in time and for all $t<T$,
$$\lim_{\varepsilon\rightarrow0}u_{\varepsilon}(t)=u(t),\; \text{in }H^1.$$
Now, we may consider the following limit:
$$\lim_{\varepsilon\rightarrow0}u_{\varepsilon}(t),\quad t>T,$$ to see whether the limiting function exists and in what sense it satisfies the critical NLS. Such construction for blow up solutions using the Virial identity was given by Merle \cite{M2}. Alternative way to construct the approximate solution $u_{\varepsilon}(t)$ can also be found in \cite{M3,M6,MRS3}. But, this only remains for very special cases. General constructions of this type are mostly open. In all cases, the asymptotic behavior of the approximate solution $u_{\varepsilon}(t)$ is crucial in the analysis.

Therefore, the asymptotic dynamics of dispersive equations with a saturated perturbation becomes a natural question.

\subsection{Results for $L^2$ critical gKdV equations}
Let us recall some results for the following $L^2$ critical gKdV equations:
\begin{equation}\label{CCP}\tag{${\rm gKdV}$}\begin{cases}
\partial_t u +(u_{xx}+u^5)_x=0, \quad (t,x)\in[0,T)\times\mathbb{R},\\
u(0,x)=u_0(x)\in H^1(\mathbb{R}).
\end{cases}
\end{equation}

This equation is $L^2$ critical, since for all $\lambda>0$,
$$u_{\lambda}(t,x)=\lambda^{-\frac{1}{2}}u(\lambda^{-3}t,\lambda^{-1}x),$$
is still a solution to \eqref{CCP} and $\|u_{\lambda}\|_{L^2}=\|u\|_{L^2}.$

There is a special class of solutions i.e. the solitary waves, which is given by 
$$u(t,x)=\frac{1}{\lambda_0^{1/2}}Q\bigg(\frac{x-x_0-\lambda_0^{-2}(t-t_0)}{\lambda_0}\bigg),$$
with
$$Q(x)=\bigg(\frac{3}{\cosh^2(2x)}\bigg)^{\frac{1}{4}},\quad Q''-Q+Q^5=0.$$
The function $Q$ is called the ground state.

From variational arguments \cite{W1}, we know that if $\|u_0\|_{L^2}<\|Q\|_{L^2}$, then the solution to \eqref{CCP} is global in time and bounded in $H^1$, while for $\|u_0\|_{L^2}\geq\|Q\|_{L^2}$, blow up may occurs. The blow up dynamics for solution with slightly supercritical mass:
\begin{equation}
\|Q\|_{L^2}<\|u_0\|_{L^2}<\|Q\|_{L^2}+\alpha^*
\end{equation}
has been developed in a series paper of Martel and Merle \cite{MM4,MM5,MM3,M1}. In particular, they prove the existence of blow up solutions with negative energy, and give a specific description of the blow up dynamics and the formation of singularity.
 
In \cite{MMR1,MMR2,MMR3}, Martel, Merle and Rapha\"{e}l give a exclusive study of the asymptotic dynamics near the ground state $Q$.

More precisely, consider the following initial data set
$$\mathcal{A}_{\alpha_0}=\bigg\{u_0\in H^1\Big|u_0=Q+\varepsilon_0,\;\|\varepsilon_0\|_{H^1}<\alpha_0,
\;\int_{y>0}y^{10}\varepsilon_0^2<1\bigg\},$$
and the following $L^2$ tube around the solitary wave family
$$\mathcal{T}_{\alpha^*}=\bigg\{u_0\in H^1\Big|\inf_{\lambda_0>0,x_0\in\mathbb{R}}\bigg\|u_0-\frac{1}{\lambda_0^{\frac{1}{2}}}Q\bigg(\frac{x-x_0}{\lambda_0}\bigg)\bigg\|_{L^2}<\alpha^*\bigg\}.$$
Then we have:
\begin{theorem}\label{CT1}
	For $0<\alpha_0\ll\alpha^*\ll1$, and $u_0\in\mathcal{A}_{\alpha_0}$, let $u(t)$ be the corresponding solution to \eqref{CCP}, and $0<T\leq+\infty$ be the maximal lifetime. Then one of the following scenarios occurs:\\
	{\rm(Blow up)}: The solution $u(t)$ blows up in finite time $0<T<+\infty$ with
	$$\|u(t)\|_{H^1}=\frac{\ell(u_0)+o(1)}{T-t},\quad \ell(u_0)>0.$$
	In addition, for all $t<T$, $u(t)\in\mathcal{T}_{\alpha^*}$.\\
    {\rm(Soliton)}: The solution is global, and for all $t<T=+\infty$, $u(t)\in\mathcal{T}_{\alpha^*}$. In addition, there exist a constant $\lambda_{\infty}>0$ and a $C^1$ function $x(t)$ such that
    \begin{gather*}
    \lambda_{\infty}^{\frac{1}{2}}u(t,\lambda_{\infty}\cdot+x(t))\rightarrow Q\text{ in }H^{1}_{{\rm loc}},\text{ as }t\rightarrow+\infty,\\
    |\lambda_{\infty}-1|\lesssim\delta(\alpha_0),\quad x(t)\sim\frac{t}{\lambda_{\infty}^2},\text{ as }t\rightarrow+\infty.
    \end{gather*}
    {\rm(Exit)}: For some finite time $0<t^*<T$, $u(t^*)\notin\mathcal{T}_{\alpha^*}$.
    
	Morever, the scenarios (Blow up) and (Exit) are stable by small perturbation in $\mathcal{A}_{\alpha_0}$.
\end{theorem}

In \cite{MMNR}, Martel, Merle, Nakanishi and Rapha\"{e}l proved that the initial data in $\mathcal{A}_{\alpha_0}$ which corresponds to the (Soliton) regime is a codimension one threshold submanifold between (Blow up) and (Exit).

\begin{theorem}\label{CT2}
	Let 
	$$\mathcal{A}_{\alpha_0}^{\perp}=\bigg\{\varepsilon_0\in H^1\Big|\|\varepsilon_0\|_{H^1}< \alpha_0,\int_{y>0}y^{10}\varepsilon_0^2<1,(\varepsilon_0,Q)=0\bigg\}.$$
	Then there exist $\alpha_0>0$, $\beta_0>0$, and a $C^1$ function $A$:
	$$\mathcal{A}_{\alpha_0}^{\perp}\rightarrow (-\beta_0,\beta_0),$$
	such that for all $\gamma_0\in\mathcal{A}_{\alpha_0}^{\perp}$ and $a\in[-\beta_0,\beta_0]$, the solution of \eqref{CCP} corresponding to $u_0=(1+a)Q+\gamma_0$ satisfies:\\
	{\rm-(Soliton)} if $a=A(\gamma_0)$;\\
	{\rm-(Blow up)} if $a>A(\gamma_0)$;\\
	{\rm-(Exit)} if $a<A(\gamma_0)$.
	
	In particular, let
	$$\mathcal{Q}=\Big\{u_0\in H^1\big|\exists \lambda_0,x_0,\text{ such that }u_0=\lambda_0^{-\frac{1}{2}}Q\big(\lambda_0^{-1}(x-x_0)\big)\Big\}.$$
	then there exists a small neighborhood $\mathcal{O}$ of $\mathcal{Q}$ in $H^1\cap L^{2}(y_{+}^{10}dy)$ and a codimension one $C^1$ submanifold $\mathcal{M}$ of $\mathcal{O}$, such that $\mathcal{Q}\subset\mathcal{M}$ and for all $u_0\in\mathcal{O}$ the corresponding solution of \eqref{CCP} is in the (Soliton) regime if and only if $u_0\in\mathcal{M}$.
\end{theorem}

\subsection{Statement of the main result}
The aim of this paper is to classify the dynamics of \eqref{CCPG} near the ground state $Q$ for \eqref{CCP}, when $\gamma$ is small enough. The main idea is that the defocusing term $\gamma u|u|^{q-1}$ has weaker nonlinear effect than the focusing term $u^5$. So, we may expect that \eqref{CCPG} has similar separation behavior as \eqref{CCP}, when $\gamma$ is small.

More precisely, we fix a small universal constant $\omega^*>0$ (to ensure the existence of the ground state $\mathcal{Q}_{\omega}$), and then introduce the following $L^2$ tube around $\mathcal{Q}_{\gamma}$:
$$\mathcal{T}_{\alpha^*,\gamma}=\bigg\{u_0\in H^1\Big|\inf_{\lambda_0>0,\lambda_0^{-m}\gamma<\omega^*,x_0\in\mathbb{R}}\bigg\|u_0-\frac{1}{\lambda_0^{\frac{1}{2}}}\mathcal{Q}_{\lambda_0^{-m}\gamma}\bigg(\frac{x-x_0}{\lambda_0}\bigg)\bigg\|_{L^2}<\alpha^*\bigg\}.$$
Then we have:
\begin{theorem}[Dynamics in $\mathcal{A}_{\alpha_0}$]\label{CMT}
For all $q>5$, there exists a constant $0<\alpha^*(q)\ll1$, such that if $0<\gamma\ll\alpha_0\ll\alpha^*<\alpha^*(q)$, then for all $u_0\in\mathcal{A}_{\alpha_0}$, the corresponding solution $u(t)$ to \eqref{CCPG} has one and only one of the following behaviors:\\
	{\rm-(Soliton):} For all $t\in[0,+\infty)$, $u(t)\in\mathcal{T}_{\alpha^*,\gamma}$. Moreover,  there exist a constant $\lambda_{\infty}\in(0,+\infty)$ and a $C^1$ function $x(t)$ such that
	\begin{gather}
	\lambda_{\infty}^{\frac{1}{2}}u(t,\lambda_{\infty}\cdot+x(t))\rightarrow \mathcal{Q}_{\lambda_{\infty}^{-m}\gamma} \text{ in } H^1_{{\rm loc}}, \text{ as $t\rightarrow +\infty$};\label{C11}\\
	x(t)\sim\frac{t}{\lambda_{\infty}^2},\quad\text{as $t\rightarrow +\infty$}.
	\end{gather}
	{\rm-(Blow down):} For all $t\in[0,+\infty)$, $u(t)\in\mathcal{T}_{\alpha^*,\gamma}$. Moreover, there exist two $C^1$ functions $\lambda(t)$ and $x(t)$, such that
	\begin{gather}
	\lambda^{\frac{1}{2}}(t)u(t,\lambda(t)\cdot+x(t))\rightarrow Q\text{ in } H^1_{{\rm loc}}, \text{ as $t\rightarrow +\infty$};\\
	\lambda(t)\sim t^{\frac{2}{q+1}},\quad x(t)\sim t^{\frac{q-3}{q+1}},\quad\text{as $t\rightarrow +\infty$},
	\end{gather}
	{\rm-(Exit):} There exists a $0<t^*_{\gamma}<+\infty$ such that $u(t^*_{\gamma})\notin\mathcal{T}_{\alpha^*,\gamma}$.
	
	There exist solutions associated to each regime. Moreover, the regime (Soliton) and (Exit) are stable under small perturbation in $\mathcal{A}_{\alpha_0}.$ 
\end{theorem}

\noindent {\it Comments on Theorem \ref{CMT}:} 

{\it 1. Classification of the flow near ground state.} Theorem \ref{CMT} gives a detailed description of the flow near the ground state $\mathcal{Q}_{\gamma}$ of \eqref{CCPG}. This kind of problems has attracted considerable attention especially for the dispersive equations. For example, Nakanishi and Schlag \cite{NS1,NS3,NS2} for Klein-Gordon and mass supercritical nonlinear Schr\"{o}dinger equations; Merle-Rapha\"{e}l \cite{FMR,MR3,MR2,MR1,MR5,MR4} and Merle-Rapha\"{e}l-Szeftel \cite{MRS3} for mass critical nonlinear Schr\"{o}dinger equations; Martel-Merle-Rapha\"{e}l \cite{MMR1,MMR2} for $L^2$ critical gKdV equations; Kenig-Merle \cite{KM2} and Duyckaerts-Merle \cite{DM2} for energy-critical nonlinear Schr\"odinger equations; Kenig-Merle \cite{KM1}, Duyckaerts-Merle \cite{DM1} and Krieger-Nakanishi-Schlag \cite{KNS,KNS1} for energy-critical wave equations; Collot-Merle-Rapha\"{e}l \cite{CMR} for energy critical nonlinear heat equations. Note that the fact that the regime (blow down) near the ground state is a codimension one threshold submanifold of initial data in $\mathcal{A}_{\alpha_0}$ could be proved similarly as in \cite{MMNR}.

{\it 2. Asymptotic stability of solitons for \eqref{CCPG}.} Since the (Soliton) regime is open, Theorem \ref{CMT} also implies the asymptotic stability of the soliton $\mathcal{Q}_{\gamma}$ for \eqref{CCPG} under some suitable decay assumption. Recall that from \cite{MM2}, the soliton $Q$ for the unperturbed critical gKdV equation is not stable in $H^1$.

{\it 3. Blow down behaviors.} Theorem \ref{CMT} shows that a saturated perturbation may lead to some chaotic behaviors (i.e. the blow down behaviors), which does not seem to appear in the unperturbed case. Examples for solution with a blow down behavior was also found by Donninger, Krieger \cite{DK} for energy critical wave equations. While for mass critical NLS, the blow down behavior can be obtained as the pseudo-conformal transformation of the log-log regime%
\footnote{See \cite[(1.16)]{MRS3} for example.}%
. However, Theorem \ref{CMT} is the first time that this type of blow down behavior is obtained in the context of a saturated perturbation. Furthermore, in Theorem \ref{CMT}, the (blow down) regime is a codimension one threshold between two stable ones, which is in contrast with the mass critical nonlinear Schr\"{o}dinger case, where the blow down regime is stable.
\\

Now we consider the case when $\gamma\rightarrow 0$. As we mentioned before, the defocusing term $\gamma u|u|^{q-1}$ has weaker nonlinear effect than the focusing term $u^5$. So the results in Theorem \ref{CMT} are expected to be a perturbation of the one in Theorem \ref{CT1}.

More precisely, we have:

\begin{theorem}\label{CST}
Let us fix a nonlinearity $q>5$, and choose $0<\alpha_0\ll\alpha^*<\alpha^*(q)$ as in Theorem \ref{CMT}. For all $u_0\in\mathcal{A}_{\alpha_0}$, let $u(t)$ be the corresponding solution of \eqref{CCP}, and $u_{\gamma}(t)$ be the corresponding solution of \eqref{CCPG}. Then we have:
\begin{enumerate}
	\item If $u(t)$ is in the (Blow up) regime defined in Theorem \ref{CT1}, then there exists $0<\gamma(u_0,\alpha_0,\alpha^*,q)\ll\alpha_0$ such that if $0<\gamma<\gamma(u_0,\alpha_0,\alpha^*,q)$, then $u_{\gamma}(t)$ is in the (Soliton) regime defined in Theorem \ref{CMT}. Moreover, there exist constants $d_i=d_i(u_0,q)>0$, $i=1,2$, such that
	\begin{equation}
	\label{C12}
	d_1\gamma^{\frac{2}{q-1}}\leq\lambda_{\infty}\leq d_2\gamma^{\frac{2}{q-1}},
	\end{equation}
	where $\lambda_{\infty}$ is the constant defined in \eqref{C11}.
	\item If $u(t)$ is in the (Exit) regime defined in Theorem \ref{CT1}, then there exists $0<\gamma(u_0,\alpha_0,\alpha^*,q)\ll\alpha_0$ such that if $0<\gamma<\gamma(u_0,\alpha_0,\alpha^*,q)$, then $u_{\gamma}(t)$ is in the (Exit) regime defined in Theorem \ref{CMT}.
\end{enumerate} 
\end{theorem}

\begin{remark}
We can see from Theorem \ref{CST} that \eqref{CCPG} is a perturbation of \eqref{CCP} as $\gamma\rightarrow0$: the (Soliton) regime of \eqref{CCPG} ``converges'' to the (Blow up) regime of \eqref{CCP}, and the (Exit) regime ``converges'' to the (Exit) regime of \eqref{CCP}. 
\end{remark}

\begin{remark}
Theorem \ref{CST} is the first result of this type for nonlinear dispersive equations. One may also expect similar results for the critical NLS or the slightly supercritical gKdV cases. But they are still completely open. 

Indeed, for critical NLS, Malkin \cite{Mal} predicted a similar asymptotic behavior for the solution to the saturated problem of critical NLS in the log-log region. However, due to the different structure between NLS and gKdV, it seems hard to apply the strategy in this paper to the NLS case.

{While for slightly supercritical gKdV case, the stable self-similar blow-up dynamics is well studied in \cite{L1}. But, due to the fact that the self-similar profile constructed in \cite[Theorem 3]{K} is not in the energy space $H^1$, we have to choose a suitable cut-off as an approximation of this profile. As a consequence, this generates some error terms that are hard to control, which makes it impossible to consider the saturated  problem in this case. However, in \cite{Stru}, Strunk proved the local wellposedness result for supercritical gKdV in a space that contains the self-similar profile, which provides an alternative option for the saturated problems.} 
\end{remark}

\subsection{Notation}\label{S15}
For $0\leq\omega<\omega^*\ll1$, we let $\mathcal{Q}_{\omega}$ be the unique nonnegative radial solution with exponential decay to the following ODE:
\begin{equation}\label{CME0}
\mathcal{Q}_{\omega}''-\mathcal{Q}_{\omega}+\mathcal{Q}_{\omega}^5-\omega \mathcal{Q}_{\omega}|\mathcal{Q}_{\omega}|^{q-1}=0.
\end{equation}
For simplicity, we denote by $Q=\mathcal{Q}_0$. Recall that we have:
$$Q(x)=\bigg(\frac{3}{\cosh^2(2x)}\bigg)^{\frac{1}{4}}.$$

We also introduce the linearized operator at $\mathcal{Q}_{\omega}$:
$$L_{\omega}f=-f''+f-5\mathcal{Q}_{\omega}^4f+q\omega|\mathcal{Q}_{\omega}|^{q-1}f.$$
Similarly, we denote by $L=L_0$.

Next, we introduce the scaling operator:
$$\Lambda f=\frac{1}{2}f+yf'.$$

Then, for a given small constant $\alpha^*$, we denote by $\delta(\alpha^*)$ a generic small constant with
$$\lim_{\alpha^*\rightarrow0}\delta(\alpha^*)=0.$$

Finally, we denote the $L^2$ scalar product by
$$(f,g)=\int f(x)g(x)\,dx.$$

\subsection{Outline of the proof}
\subsubsection{Decomposition of the flow}We are searching for solutions of the following form:
$$u(t,x)\sim\frac{1}{\lambda(t)^{\frac{1}{2}}}Q_{b(t),\omega(t)}\bigg(\frac{x-x(t)}{\lambda(t)}\bigg),$$
$$\omega=\frac{\gamma}{\lambda^m},\;\frac{ds}{dt}=\frac{1}{\lambda^3},\;\frac{\lambda_s}{\lambda}=-b,\;\frac{x_s}{\lambda}=1,$$
which lead to the modified self-similar equation:
\begin{equation}\label{CME}
b\Lambda Q_{b,\omega}+\big(Q_{b,\omega}''-Q_{b,\omega}+Q_{b,\omega}^5-\omega Q_{b,\omega}|Q_{b,\omega}|^{q-1}\big)'=0.
\end{equation}
Formal computations show that $b$ and $\omega$ must satisfy the following condition:
$$b_s+2b^2+c_0\omega_s=0,$$
where $c_0=c_0(q)>0$ is a universal constant.

Combining all the above, we get the following formal finite dimensional system:
\begin{equation}\label{CFDS}
\begin{cases}
\frac{ds}{dt}=\frac{1}{\lambda^3},\;\frac{\lambda_s}{\lambda}=-b,\;\frac{x_s}{\lambda}=1,\\
b_s+2b^2+c_0\omega_s=0,\;\omega=\frac{\gamma}{\lambda^m}.
\end{cases}
\end{equation}
By standard computations, it is easy to see that \eqref{CFDS} has the following behavior. Let
$$L_0=\frac{b(0)}{\lambda^2(0)}+\frac{mc_0\gamma}{(m+2)\lambda^{m+2}(0)}.$$
We have:
\begin{enumerate}
	\item If $L_0>0$, then
	$$b(t)\rightarrow0,\quad\lambda(t)\rightarrow \bigg(\frac{m\gamma c_0}{(m+2)L_0}\bigg)^{\frac{1}{m+2}},\quad x(t)\sim \bigg(\frac{(m+2)L_0}{m\gamma c_0}\bigg)^{\frac{2}{m+2}}t,$$
	as $t\rightarrow +\infty$, which corresponds to the (Soliton) regime.
	\item If $L_0=0$, then
	$$b(t)\rightarrow0,\quad\lambda(t)\rightarrow +\infty,\quad x(t)\rightarrow +\infty,$$
	as $t\rightarrow +\infty$, which corresponds to the (Blow down) regime.
	\item If $L_0<0$, then
	$$b(t)\rightarrow-\infty,\quad\lambda(t)\rightarrow +\infty,$$
	as $t\rightarrow +\infty$, which corresponds to the (Exit) regime.
\end{enumerate}

\subsubsection{Modulation theory}
Our first step is to find a solution to \eqref{CME}. But for our analysis, it is enough to consider a suitable approximation%
\footnote{See Section 2.1 for more details.}%
:
$$Q_{b,\omega}(y)=\mathcal{Q}_{\omega}(y)+b\chi(|b|^{\beta}y)P_{\omega}(y).$$

As long as the solution remains in $\mathcal{T}_{\alpha^*,\gamma}$, we can introduce the following geometrical decomposition:
$$u(t)=\frac{1}{\lambda(t)^{\frac{1}{2}}}\Big[Q_{b(t),\omega(t)}+\varepsilon(t)\Big]\bigg(\frac{x-x(t)}{\lambda(t)}\bigg)$$
with $\omega(t)=\gamma/\lambda(t)^m$, and the error term satisfies some orthogonality conditions. Then the equation of the parameters are roughly speaking of the following form:
$$\frac{\lambda_s}{\lambda}+b=\frac{dJ_1}{ds}+O(\|\varepsilon\|^2_{H^1_{\rm loc}}),\quad b_s+2b^2+c_0\omega_s=\frac{dJ_2}{ds}+O(\|\varepsilon\|^2_{H^1_{\rm loc}}),$$
with 
$$|J_{i}|\lesssim \|\varepsilon\|_{H^1_{\rm loc}}+\int_{y>0}|\varepsilon|.$$

Therefore, a $L^1$ control on the right is needed, otherwise $J_i$ will perturb the formal system \eqref{CFDS}. 

\subsubsection{Monotonicity Formula}
Our next step is to derive a control for $\|\varepsilon\|_{H^1_{\rm loc}}$. Similar to {\cite[Proposition 3.1]{MMR1}}, we introduce the following nonlinear functional:
\begin{align*}
\mathcal{F}\sim\int\big(\psi\varepsilon_{y}^2+\varphi\varepsilon^2-5\varepsilon^2Q^4_{b,\omega}\psi+q\omega\varepsilon^2|Q_{b,\omega}|^{q-1}\psi\big),
\end{align*}
for some well-chosen weight functions $(\psi,\varphi)$, which decay exponentially to the left, and grow polynomially on the right. We will see from the choice of the orthogonality condition that the leading quadratic term of $\mathcal{F}$ is coercive:
$$\mathcal{F}\gtrsim \|\varepsilon\|^2_{H^1_{\rm loc}}.$$
Most importantly, we have the following monotonicity formula:
$$\frac{d}{ds}\bigg(\frac{\mathcal{F}}{\lambda^{2j}}\bigg)+\frac{\|\varepsilon\|^2_{H^1_{\rm loc}}}{\lambda^{2j}}\lesssim \frac{\omega^2b^2+b^4}{\lambda^{2j}},$$
for $j=0,1$. This formula is crucial in all three cases.

\subsubsection{Rigidity}
The selection of the dynamics depends on:
\begin{enumerate}
	\item For all $t$, 
	$$\bigg|b(t)+\frac{mc_0}{m+2}\omega(t)\bigg|\lesssim \|\varepsilon(t)\|^2_{H^1_{\rm loc}}+b^2(t)+\omega^2(t).$$
	\item For some $t_1^*<T=+\infty$, 
	$$b(t_1^*)+\frac{mc_0}{m+2}\omega(t_1^*)\gg \|\varepsilon(t)\|^2_{H^1_{\rm loc}}+b^2(t)+\omega^2(t).$$
	\item For some $t_1^*<T=+\infty$, 
	$$-b(t_1^*)-\frac{mc_0}{m+2}\omega(t_1^*)\gg \|\varepsilon(t)\|^2_{H^1_{\rm loc}}+b^2(t)+\omega^2(t).$$
\end{enumerate}
We will see that in the first case we have for all $t$,
$$|b(t)|\sim\omega(t)\gg \|\varepsilon(t)\|^2_{H^1_{\rm loc}},$$
And in the second case we have $$\omega(t)\gg|b(t)|\gg\|\varepsilon(t)\|^2_{H^1_{\rm loc}},$$
for $t>t_2^*\geq t_1^*$ as long as $u(t)$ remains in $\mathcal{T}_{\alpha^*,\gamma}$. 
While in the third case, we have
$$-b(t)\gg\omega(t)\gg\|\varepsilon(t)\|^2_{H^1_{\rm loc}},$$
for $t>t^*_1$ as long as $u(t)$ remains in $\mathcal{T}_{\alpha^*,\gamma}$. Then reintegrating the modulation equations, we will see that these three cases correspond to the (Blow down), (Soliton) and (Exit) regimes respectively.

Moreover, the condition on $b(t_1^*)$ and $\omega(t_1^*)$ which determines the (Soliton) and (Exit) regimes is an open condition to the initial data due to the continuity of the flow. On the other hand, it is easy to construct solutions, which belongs to the (Soliton) and (Exit) regime respectively. Since, the initial data set $\mathcal{A}_{\alpha_0}$ is connected, we can see that there exist solutions corresponding to the (Blow down) regime.

\subsubsection{Proof of Theorem \ref{CST}}
The proof of Theorem \ref{CST} is based on the fact that the separation condition for \eqref{CCPG} is close to the separation condition for \eqref{CCP}, when $\gamma\rightarrow0$. Then Theorem \ref{CST} follows immediately from a modified $H^1$ perturbation theory%
\footnote{See \cite[Theorem 3.1]{KKSV} for the standard $L^2$ perturbation theory.}%
. 

\subsection*{Acknowledgement}
The author would like to thank his supervisors F. Merle and T. Duyckaerts for having suggested
this problem to him and giving a lot of guidance.

\section{Nonlinear profile and decomposition of the flow}
In this section we will introduce the nonlinear profile and the geometrical decomposition similar to the one in \cite{MMR1}, which turns out to lead to the desired rigidity dynamics.
\subsection{Structure of the linearized operator $L_{\omega}$} 
Denote by $\mathcal{Y}$ the set of smooth function $f$ such that for all $k\in\mathbb{N}$, there exist $r_k>0$, $C_k>0$, with
{\begin{equation}
|\partial_{y}^kf(y)|\leq C_k(1+|y|)^{r_k}e^{-|y|}.
\end{equation}}

Let us first recall some results about the linearized operator $L$.

\begin{lemma}[Properties of $L$, \cite{MM2}, \cite{MMR1}, \cite{W2}]\label{CL0}
The self-adjoint operator $L$ ({recall that we use the notation $L=L_0$, which was introduced in Section \ref{S15}}) in $L^2$ has the following properties:
\begin{enumerate}
	\item Eigenfunction: $LQ^3=-8Q^3$, $LQ'=0$, $\ker{L}=\{aQ'|a\in\mathbb{R}\}$.
	\item Scaling: $L(\Lambda Q)=-2Q$.
	\item For any function $f\in L^2$ orthogonal to $Q'$, there exist a unique $g\in H^2$ such that $Lg=f$ with $(g,Q')=0$. Moreover, if $f$ is even (or respectively odd), then $g$ is even (or respectively odd).
	\item If $f\in L^2$ such that $Lf\in\mathcal{Y}$, then $f\in\mathcal{Y}$.
	\item Coercivity: For all $f\in H^1$, if $(f,Q^3)=(f,Q')=0$, then $(Lf,f)\geq (f,f)$.  Moreover, there exists a $\kappa_0>0$ such that for all $f\in H^1$,
	$$(Lf,f)\geq\kappa_0\|f\|_{H^1}^2-\frac{1}{\kappa_0}\Big[(f,Q)^2+(f,\Lambda Q)^2+(f,y\Lambda Q)^2\Big].$$ 
\end{enumerate}
\end{lemma}
\begin{proposition}[{Nonlocalized profiles, \cite[Proposition 2.2]{MMR1}}]\label{CP1}
There exist a unique function $P$ with $P'\in\mathcal{Y}$, such that:
\begin{gather}
(LP)'=\Lambda Q,\quad \lim_{y\rightarrow-\infty}P(y)=\frac{1}{2}\int Q,\quad |P(y)|\lesssim\ e^{-\frac{y}{2}}\text{ for }y>0,\\
(P,Q)=\frac{1}{16}\bigg(\int Q\bigg)^2,\quad(P,Q')=0.
\end{gather}
\end{proposition}

Now for the ground state $\mathcal{Q}_{\omega}$ and the linearized operator $L_{\omega}$, we have the following properties:
\begin{lemma}\label{CL1}
For $0<\omega<\omega^*\ll1$, we have:
\begin{enumerate}
	\item Null space: $\ker{L_{\omega}}=\{a\mathcal{Q}_{\omega}'|a\in\mathbb{R}\}$.
	\item Pseudo-scaling rule: $L_{\omega}(\Lambda \mathcal{Q}_{\omega})=-2\mathcal{Q}_{\omega}+\frac{q-5}{2}\omega \mathcal{Q}_{\omega}^q$.
	\item For any function $f\in L^2$ orthogonal to $\mathcal{Q}_{\omega}'$, there exist a unique $g\in H^2$ such that $L_{\omega}g=f$ with $(g,\mathcal{Q}_{\omega}')=0$. Moreover, if $f$ is even (or respectively odd), then $g$ is even (or respectively odd).\
	\item If $f\in L^2$ such that $L_{\omega}f\in\mathcal{Y}$, then $f\in\mathcal{Y}$.
	\item Let $Z_{\omega}=\frac{\partial \mathcal{Q}_{\omega}}{\partial \omega}$, then $Z_{\omega}\in\mathcal{Y}$, and $L_{\omega}Z_{\omega}=-\mathcal{Q}_{\omega}|\mathcal{Q}_{\omega}|^{q-1}$.
	\item Coercivity: There exists a $\kappa_0>0$ such that for all $f\in H^1$,
	$$(L_{\omega}f,f)\geq\kappa_0\|f\|_{H^1}^2-\frac{1}{\kappa_0}\Big[(f,\mathcal{Q}_{\omega})^2+(f,\Lambda \mathcal{Q}_{\omega})^2+(f,y\Lambda \mathcal{Q}_{\omega})^2\Big].$$
\end{enumerate}
\end{lemma}
\begin{proof}[Proof]
(1) follows from the same arguments {as the proof of \cite[Proposition 2.8]{W2} and \cite[Proposition 3.2]{W3}.} (2) follows from direct computation. (3) is a direct corollary of (1). While for (4), from standard elliptic theory, we know that $f$ is smooth and bounded. So we have $Lf\in\mathcal{Y}$, from Lemma \ref{CL0}, we have $f\in\mathcal{Y}$. 

Now we turn to the proof of (5). Differentiating the equation \eqref{CME0}, we obtain $L_{\omega}Z_{\omega}=-\mathcal{Q}_{\omega}|\mathcal{Q}_{\omega}|^{q-1}$. Since $\mathcal{Q}_{\omega}|\mathcal{Q}_{\omega}|^{q-1}\in\mathcal{Y}$, if we can show that $Z_{\omega}\in L^2$, then we have $Z_{\omega}\in\mathcal{Y}$. To do this, we introduce the following map:
$$F: H^2_e\times\mathbb{R}\mapsto L^2_e,\quad(u,\omega)\mapsto-u''+u-u^5+\omega u|u|^{q-1},$$
where $H^2_e$ (respectively $L^2_e$) is the Banach space consists of all $H^2$ (respectively $L^2$) functions which are even. Since $H^2(\mathbb{R})$ is continuously embedded into $L^{\infty}(\mathbb{R})$, the map $F$ is well-defined. 

We claim that there exists a small $\omega^*>0$, such that if $0\leq\omega<\omega^*$, then there exist a unique $u(\omega)\in H^2_e$, such that $F(u(\omega),\omega)=0$. Since we have $F(Q,0)=0$, from implicit function theory, we only remains to show that the Fr\'{e}chet derivative with respect to $u$, i.e. $\frac{\partial F}{\partial u}\big|_{(Q,0)}\in\mathcal{L}(H^2_e,L^2_e)$ is invertible and continuous. But it is easy to see that 
$$\frac{\partial F}{\partial u}\bigg|_{(Q,0)}=L,$$
which is invertible and continuous due to (3) of Lemma \ref{CL0}. Hence, we obtain the existence of such $u(\omega)$. Moreover, since $F$ is continuously differentiable with respect to both $u$ and $\omega$. So we have $u(\omega)$ is continuously differentiable with respect to $\omega$. In particular, we have $\frac{\partial u}{\partial \omega}\in H^2_e$. But from the uniqueness of $u(\omega)$, we must have $u(\omega)=\mathcal{Q}_{\omega}$. As a consequence, we have $Z_{\omega}=\frac{\partial \mathcal{Q}_{\omega}}{\partial \omega}=\frac{\partial u}{\partial \omega}\in H^2_e$, which concludes the proof of (5). 

Finally, (6) follows immediately from a perturbation argument for part (5) of Lemma \ref{CL0}. {More precisely, since $\mathcal{Q}_{\omega}$ is $C^1$ with respect to $\omega$, we have: for all $f\in H^1$,
$$(L_{\omega}f,f)=(Lf,f)+O(\omega)\|f\|^2_{H^1},$$
and 
$$(f,\mathcal{Q}_{\omega})^2+(f,\Lambda \mathcal{Q}_{\omega})^2+(f,y\Lambda \mathcal{Q}_{\omega})^2=(f,Q)^2+(f,\Lambda Q)^2+(f,y\Lambda Q)^2+O(\omega)\|f\|^2_{H^1}.$$
Together with part (5) of Lemma \ref{CL0}, we conclude the proof of part (6) of Lemma \ref{CL1}.}

We then finish the proof of Lemma \ref{CL1}.
\end{proof}

\begin{proposition}\label{CP2}
For $0<\omega<\omega^*\ll1$, there exist a smooth function $P_{\omega}$ with $P'_{\omega}\in\mathcal{Y}$, such that:
\begin{gather}
(L_{\omega}P_{\omega})'=\Lambda \mathcal{Q}_{\omega},\quad \lim_{y\rightarrow-\infty}P_{\omega}(y)=\frac{1}{2}\int \mathcal{Q}_{\omega},\label{C21}\\
(P_{\omega},\mathcal{Q}'_{\omega})=0,\quad  (P_{\omega},\mathcal{Q}_{\omega})=\frac{1}{16}\bigg(\int Q\bigg)^2+F(\omega),\label{C22}
\end{gather}
where $F$ is a $C^1$ function with $F(0)=0$. Moreover there exist constants $C_0,C_1,\ldots$, independent of $\omega$, such that
\begin{gather}
|P_{\omega}(y)|+\bigg|\frac{\partial P_{\omega}}{\partial\omega}(y)\bigg|\leq C_0e^{-\frac{y}{2}},\text{ \rm for all }y>0,\label{C2001}\\
|P_{\omega}(y)|+\bigg|\frac{\partial P_{\omega}}{\partial\omega}(y)\bigg|\leq C_0,\text{ \rm for all }y\in\mathbb{R},\label{C2002}\\
|\partial_{y}^kP_{\omega}(y)|\leq C_ke^{-\frac{|y|}{2}}\text{ \rm for all $k\in\mathbb{N}_+$, $y\in\mathbb{R}$}\label{C2003}.
\end{gather}
\end{proposition}
\begin{proof}[Proof]
	The proof of Proposition \ref{CP2} is almost parallel to Proposition \ref{CP1}. We look for solution of the form $P_{\omega}=\widetilde{P}_{\omega}-\int_{y}^{+\infty}\Lambda \mathcal{Q}_{\omega}$. The function $y\rightarrow \int_{y}^{+\infty}\Lambda \mathcal{Q}_{\omega}$ is bounded and decays exponentially as $y\rightarrow+\infty$. Then, $P_{\omega}$ solves \eqref{C21} if and only if $\widetilde{P}_{\omega}$ solves
	$$(L_{\omega}\widetilde{P}_{\omega})'=\Lambda \mathcal{Q}_{\omega}+\bigg(L_{\omega}\int_{y}^{+\infty}\Lambda \mathcal{Q}_{\omega}\bigg)'=R_{\omega}',$$
	where
	$$R_{\omega}=(\Lambda \mathcal{Q}_{\omega})'-5\mathcal{Q}_{\omega}^4\int_{y}^{+\infty}\Lambda \mathcal{Q}_{\omega}+q\omega |\mathcal{Q}_{\omega}|^{q-1}\int_{y}^{+\infty}\Lambda \mathcal{Q}_{\omega}.$$
	Note that $R_{\omega}\in\mathcal{Y}$. Since $(\Lambda \mathcal{Q}_{\omega},\mathcal{Q}_{\omega})=0$ and $L_{\omega}\mathcal{Q}_{\omega}'=0$, we have $(R_{\omega},\mathcal{Q}_{\omega}')=-(R_{\omega}',\mathcal{Q}_{\omega})=0$. Then from Lemma \ref{CL1}, there exists a unique $\widetilde{P}_{\omega}\in\mathcal{Y}$, orthogonal to $\mathcal{Q}_{\omega}'$, such that $L_{\omega}\widetilde{P}_{\omega}=R_{\omega}$. Then $P_{\omega}=\widetilde{P}_{\omega}-\int_{y}^{+\infty}\Lambda \mathcal{Q}_{\omega}$ satisfies \eqref{C21} with $(P_{\omega},\mathcal{Q}_{\omega}')=0$ and $\lim_{y\rightarrow-\infty}P_{\omega}(y)=\frac{1}{2}\int \mathcal{Q}_{\omega}$. Moreover, we have
	\begin{align*}
	2\int P_{\omega}\mathcal{Q}_{\omega}&=-\int(L_{\omega}P_{\omega})\Lambda \mathcal{Q}_{\omega}+O(\omega)=\int\Lambda \mathcal{Q}_{\omega}\int_{y}^{+\infty}\Lambda \mathcal{Q}_{\omega}+O(\omega)\\
	&=\frac{1}{2}\bigg(\int \Lambda \mathcal{Q}_{\omega}\bigg)^2+O(\omega)=\frac{1}{8}\bigg(\int Q\bigg)^2+O(\omega).
	\end{align*}
	Let
	$$F(\omega)=(P_{\omega},\mathcal{Q}_{\omega})-\frac{1}{16}\bigg(\int Q\bigg)^2,$$
	then $F(0)=0$. 
	
	Next we claim that $\frac{\partial \widetilde{P}_{\omega}}{\partial \omega}\in \mathcal{Y}$.  Let us differentiate the equation $L_{\omega}\widetilde{P}_{\omega}=R_{\omega}$ to get
	\begin{equation}
	\label{C2004}
	L_{\omega}\bigg(\frac{\partial \widetilde{P}_{\omega}}{\partial \omega}\bigg)=\frac{\partial R_{\omega}}{\partial \omega}-20Z_{\omega}\mathcal{Q}_{\omega}^3\widetilde{P}_{\omega}+q(q-1)\omega Z_{\omega}\mathcal{Q}_{\omega}|\mathcal{Q}_{\omega}|^{q-3}\widetilde{P}_{\omega}+q|\mathcal{Q}_{\omega}|^{q-1}\widetilde{P}_{\omega}.
	\end{equation}
	Since $Z_{\omega}\in\mathcal{Y}$, it is easy to check that $\frac{\partial R_{\omega}}{\partial \omega}\in \mathcal{Y}$. So Lemma \ref{CL1} implies that $\frac{\partial \widetilde{P}_{\omega}}{\partial \omega}\in \mathcal{Y}$. 
	
	Now it only remains to prove \eqref{C2001}--\eqref{C2003}. But from \cite[Section 6]{BL}, there exist constants $M_0, M_1,\ldots$, independent of $\omega$, such that for all $k\in\mathbb{N}$, $y\in\mathbb{R}$,
	$$|\partial_{y}^k\mathcal{Q}_{\omega}(y)|\leq M_ke^{-\frac{2|y|}{3}}.$$
	Together with \eqref{C2004} and the construction of $P_{\omega}$, we obtain \eqref{C2001}--\eqref{C2003}. It is easy to see that \eqref{C2001}--\eqref{C2003} also implies that $F\in C^1$. Then we conclude the proof of Lemma \ref{CL1}.
\end{proof}

Now, we proceed to a simple localization of the profile to avoid the nontrivial tail on the left. Let $\chi$ be a smooth function with $0\leq\chi\leq1$, $\chi'\geq0$, $\chi(y)=1$ if $y>-1$, $\chi(y)=0$ if $y<-2$. We fix a
\begin{equation}
\beta=\frac{3}{4}.
\end{equation} 
And define the localized profile:
\begin{equation}
\chi_{b}(y)=\chi(|b|^{\beta}y),\quad Q_{b,\omega}(y)=\mathcal{Q}_{\omega}+b\chi_b(y)P_{\omega}(y).
\end{equation}

\begin{lemma}[Localized Profiles]\label{CL2}
For $|b|<b^*\ll1$, $0<\omega<\omega^*\ll1$, there holds:
\begin{enumerate}
	\item Estimates on $Q_b$: For all $y\in\mathbb{R}$, $k\in\mathbb{N}$,
	\begin{align}
	&|Q_{b,\omega}(y)|\lesssim e^{-|y|}+|b|\big(\mathbf{1}_{[-2,0]}(|b|^{\beta}y)+e^{-\frac{|y|}{2}}\big),\\
	&|\partial_{y}^kQ_{b,\omega}(y)|\lesssim e^{-|y|}+|b|e^{-\frac{|y|}{2}}+|b|^{1+k\beta}\mathbf{1}_{[-2,-1]}(|b|^{\beta}y),
	\end{align}
	where $\mathbf{1}_{I}$ denotes the characteristic function of the interval $I$.
	\item Equation of $Q_{b,w}$: Let
	\begin{equation}\label{CAP}
	-\Psi_{b,\omega}=b\Lambda Q_{b,\omega}+\big(Q_{b,\omega}''-Q_{b,\omega}+Q_{b,\omega}^5-\omega Q_{b,\omega}|Q_{b,\omega}|^{q-1}\big)'.
	\end{equation}
	Then, for all $y\in\mathbb{R}$,
	\begin{align}
	-\Psi_{b,\omega}=&b^2\big((10\mathcal{Q}_{\omega}^3P_{\omega}^2)_{y}+\Lambda P_{\omega}\big)-\frac{1}{2}b^2(1-\chi_{b})P_{\omega}\nonumber\\
	&+O\Big(|b|^{1+\beta}\mathbf{1}_{[-2,-1]}(|b|^{\beta}y)+b^2(\omega+|b|)e^{-\frac{|y|}{2}}\Big).\label{C23}
	\end{align}
	Moreover, we have
	\begin{equation}\label{C24}
	|\partial_{y}\Psi_{b,\omega}(y)|\lesssim |b|^{1+2\beta}\mathbf{1}_{[-2,-1]}(|b|^{\beta}y)+b^2e^{-\frac{|y|}{2}}.
	\end{equation}
	\item Mass and energy properties of $Q_{b,\omega}$:
	\begin{gather}
	\Bigg|\int Q_{b,\omega}^2-\bigg(\int \mathcal{Q}_{\omega}^2+2b\int P_{\omega}\mathcal{Q}_{\omega}\bigg)\Bigg|\lesssim |b|^{2-\beta},\label{C25}\\
	|E(Q_{b,\omega})|\lesssim |b|+\omega.\label{C205}
	\end{gather}
\end{enumerate}
\end{lemma}
\begin{proof}[Proof]
The proof of (1) follows immediately from the definition of $Q_{b,\omega}$ and Proposition \ref{CP2}. For (2), let us expand $Q_{b,\omega}=\mathcal{Q}_{\omega}+b\chi_b P_{\omega}$ in the expression of $\Psi_{b,\omega}$, using the fact that
$$\mathcal{Q}_{\omega}''-\mathcal{Q}_{\omega}+\mathcal{Q}_{\omega}^5-\omega \mathcal{Q}_{\omega}|\mathcal{Q}_{\omega}|^{q-1}=0,\quad (L_{\omega}P_{\omega})'=\Lambda \mathcal{Q}_{\omega},$$
we have:
\begin{align*}
-\Psi_{b,\omega}=&b(1-\chi_{b})\Lambda \mathcal{Q}_{\omega}\\
&+b\big(\chi_b'''P_{\omega}+3\chi_b''P_{\omega}'+2\chi_{b}'P_{\omega}''-\chi_b'P_{\omega}+5\chi_b'\mathcal{Q}_{\omega}P_{\omega}-q\omega\chi_b'|\mathcal{Q}_{\omega}|^{q-1}P_{\omega}\big)\\
&+b^2\Big((10\mathcal{Q}_{\omega}^3\chi_b^2P_{\omega}^2)_{y}+P_{\omega}\Lambda \chi_b+\chi_byP_{\omega}'\Big)\\
&+b^3(10\mathcal{Q}_{\omega}^2\chi_b^3P_{\omega}^3)_{y}+b^4(5\mathcal{Q}_{\omega}\chi_b^4P_{\omega}^4)_{y}+b^5(\chi_b^5P_{\omega}^5)_{y}\\
&-\omega\Big((\mathcal{Q}_{\omega}+b\chi_bP_{\omega})|\mathcal{Q}_{\omega}+b\chi_bP_{\omega}|^{q-1}-\mathcal{Q}_{\omega}|\mathcal{Q}_{\omega}|^{q-1}-qb\chi_bP_{\omega}|\mathcal{Q}_{\omega}|^{q-1}\Big)_{y}.
\end{align*}
We keep track of all terms up to $b^2$. Then \eqref{C23} and \eqref{C24} follow from the construction of the profile $Q_{b,\omega}$.

Finally, for (3), we have
$$\int\chi_b^2P_{\omega}^2\lesssim |b|^{-\beta}.$$
Then \eqref{C25} follows from
$$\int Q_{b,\omega}^2=\int \mathcal{Q}_{\omega}^2+2b\int\chi_bP_{\omega}\mathcal{Q}_{\omega}+b^2\int\chi_b^2P_{\omega}^2.$$ 
While for \eqref{C205}, since $E(\mathcal{Q}_{\omega})=O(\omega)$, we have:
$$|E(Q_{b,\omega})|\lesssim |b|+|E(\mathcal{Q}_{\omega})|\lesssim |b|+\omega,$$
which conludes the proof of Lemma \ref{CL2}.
\end{proof}

\subsection{Geometrical decomposition and modulation estimates}
In this paper we consider $H^1$ solution to \eqref{CCPG} a priori in the modulates tube $\mathcal{T}_{\alpha^*,\gamma}$ of functions near the soliton manifold. More precisely, we have

\begin{lemma}\label{CL3}
Assume that there exist $(\lambda_1(t),x_1(t))\in\big((\gamma/\omega^*)^{1/m},+\infty\big)\times\mathbb{R}$ and $\varepsilon_1(t)$ such that for all $t\in[0,t_0)$, the solution $u(t)$ to \eqref{CCPG} satisfies
\begin{equation}
u(t,x)=\frac{1}{\lambda_1^{\frac{1}{2}}(t)}\big[\mathcal{Q}_{\omega_1(t)}+\varepsilon_1(t)\big]\bigg(\frac{x-x_1(t)}{\lambda_1(t)}\bigg),
\end{equation}
with, $\forall t\in[0,t_0)$,
\begin{equation}\label{C203}
\omega_1(t)+\|\varepsilon_1(t)\|_{L^2}\leq \kappa\ll1,
\end{equation} 
where 
$$\omega_{1}(t)=\frac{\gamma}{\lambda^m_1(t)}.$$
Then we have
\begin{enumerate}
	\item There exist continuous functions $(\lambda(t),x(t),b(t))\in(0,+\infty)\times\mathbb{R}^2$, such that for all $t\in[0,t_0)$,
\begin{equation}\label{CGD}
\varepsilon(t,y)=\lambda^{\frac{1}{2}}(t)u(t,\lambda(t)y+x(t))-Q_{b(t),\omega(t)}
\end{equation}
satisfies the orthogonality conditions:
\begin{equation}
\label{COC}
(\varepsilon(t),\mathcal{Q}_{\omega(t)})=(\varepsilon(t),\Lambda \mathcal{Q}_{\omega(t)})=(\varepsilon(t),y\Lambda \mathcal{Q}_{\omega(t)})=0,
\end{equation}
where
$$\omega(t)=\frac{\gamma}{\lambda^m(t)}.$$
Moreover,
\begin{gather}
\omega(t)+\|\varepsilon(t)\|_{L^2}+|b(t)|+\bigg|1-\frac{\lambda_1(t)}{\lambda(t)}\bigg|\lesssim\delta(\kappa),\\
\|\varepsilon(0)\|_{H^1}\lesssim \delta(\|\varepsilon_1(0)\|_{H^1}).
\end{gather}
\item The parameters and error term depend continuously on the initial data. Considering a family of solutions $u_n(t)$, with $u_{0,n}\in\mathcal{A}_{\alpha_0}$, and $u_{0,n}\rightarrow u_{0}$ in $H^1$, as $n\rightarrow+\infty$. Let $(\lambda_n(t),b_n(t),x_n(t),\varepsilon_n(t))$ be the corresponding geometrical parameters and error terms of $u_n(t)$. Suppose the geometrical decomposition of $u_n(t)$ and $u(t)$ hold on $[0,T_0]$ for some $T_0>0$. Then for all $t\in[0,T_0]$, we have:
	\begin{equation}\label{C206}
	\big(\lambda_n(t),b_n(t),x_n(t),\varepsilon_n(t)\big)\xrightarrow{\mathbb{R}^3\times H^1}\big(\lambda(t),b(t),x(t),\varepsilon(t)\big),
	\end{equation}
	as $n\rightarrow+\infty$.
\end{enumerate}
\end{lemma}
\begin{proof}
Lemma \ref{CL3} is a standard consequence of the implicit function theory. We leave the proof in Appendix \ref{APP1}.
\end{proof}
\begin{remark}
{Similar arguments have also been used in \cite[Lemma 1]{MM4}, \cite[Lemma 1]{MM3}, \cite[Lemma 2.5]{MMR1}, \cite[Lemma 2]{M1} $etc$.}
\end{remark}
\begin{remark}
The smallness of $\omega(t)$ ensures that $\mathcal{Q}_{\omega(t)}$ and $Q_{b(t),\omega(t)}$ are both well defined.
\end{remark}

\subsection{Modulation Equation}
In the frame work of Lemma \ref{CL3}, we introduce the rescaled variables $(s,y)$
\begin{equation}
y=\frac{x-x(t)}{\lambda(t)},\quad s=\int_0^t\frac{1}{\lambda^3(\tau)}\,d\tau.
\end{equation} 
Then, we have the following properties:
\begin{proposition}\label{CP3}
Assume for all $t\in[0,t_0)$,
\begin{equation}\label{C26}
\omega(t)+\|\varepsilon(t)\|_{L^2}+\int\varepsilon_{y}^2e^{-\frac{3|y|}{2(q-2)}}\,dy\leq\kappa
\end{equation}
for some small universal constant $\kappa>0$. Then the functions $(\lambda(s),x(s),b(s))$ are all $C^1$ and the following holds
\begin{enumerate}
	\item Equation of $\varepsilon$: For all $s\in[0,s_0)$,
	\begin{multline}\label{CMEQ}
	\varepsilon_s-(L_{\omega}\varepsilon)_{y}+b\Lambda\varepsilon=\bigg(\frac{\lambda_s}{\lambda}+b\bigg)(\Lambda Q_{b,\omega}+\Lambda\varepsilon)+\bigg(\frac{x_s}{\lambda}-1\bigg)(Q_{b,\omega}+\varepsilon)_{y}\\
	-b_s\frac{\partial Q_{b,\omega}}{\partial b}-\omega_s\frac{\partial Q_{b,\omega}}{\partial \omega}+\Psi_{b,\omega}-(R_b(\varepsilon))_{y}-(R_{\rm NL}(\varepsilon))_{y},
	\end{multline}
	where
	\begin{align}
	&\Psi_{b,\omega}=-b\Lambda Q_{b,\omega}-\big(Q_{b,\omega}''-Q_{b,\omega}+Q_{b,\omega}^5-\omega Q_{b,\omega}|Q_{b,\omega}|^{q-1}\big)',\\
	&R_{b}(\varepsilon)=5(Q_{b,\omega}^4-\mathcal{Q}_{\omega}^4)\varepsilon-q\omega(|Q_{b,\omega}|^{q-1}-|\mathcal{Q}_{\omega}|^{q-1})\varepsilon,\\
	&R_{\rm NL}(\varepsilon)=(\varepsilon+Q_{b,\omega})^5-5Q_{b,\omega}^{4}\varepsilon-Q_{b,\omega}^5-\omega\big[(\varepsilon+Q_{b,\omega})|\varepsilon+Q_{b,\omega}|^{q-1}\nonumber\\
	&\qquad\qquad\;\;-q\varepsilon |Q_{b,\omega}|^{q-1}-Q_{b,\omega}|Q_{b,\omega}|^{q-1}\big].
	\end{align}
	\item Estimate induced by the conservation laws: for $s\in[0,s_0)$, there holds:
	\begin{gather}
	\|\varepsilon\|_{L^2}\lesssim|b|^{\frac{1}{4}}+\omega^{\frac{1}{2}}+\bigg|\int u^2_0-\int Q^2\bigg|^{\frac{1}{2}},	\label{CMC}\\
	\frac{\|\varepsilon_{y}\|_{L^2}^2}{\lambda^2}\lesssim\frac{1}{\lambda^2}\bigg(\omega+|b|+\int\varepsilon^2e^{-\frac{|y|}{10}}\bigg)+\gamma\frac{\|\varepsilon_{y}\|_{L^2}^{m+2}}{\lambda^{m+2}}+|E_0|.\label{CEC}
	\end{gather}
	\item $H^1$ modulation equation: for all $s\in[0,s_0)$,
	\begin{align}
	&\bigg|\frac{\lambda_s}{\lambda}+b\bigg|+\bigg|\frac{x_s}{\lambda}-1\bigg|\lesssim \bigg(\int\varepsilon^2e^{-\frac{|y|}{10}}\bigg)^{\frac{1}{2}}+|b|(\omega+|b|),\label{CMS1} \\
	&|b_s|+|\omega_s|\lesssim(\omega+|b|)\Bigg[\bigg(\int\varepsilon^2e^{-\frac{|y|}{10}}\bigg)^{\frac{1}{2}}+|b|\Bigg]+\int\varepsilon^2e^{-\frac{|y|}{10}}.\label{CMS2}
	\end{align}
	\item $L^1$ control on the right: Assume the uniformly $L^1$ control on the right: $\forall t\in[0,t_0)$
	\begin{equation}\label{C202}
	\int_{y>0}|\varepsilon(t)|\lesssim \delta(\kappa).
	\end{equation} 
	then the quantities $J_1$ and $J_2$ below are well-defined. Moreover, we have
	\begin{enumerate}
		\item Law of $\lambda$: let
		\begin{equation}\label{C2005}
		\rho_1(y)=\frac{4}{(\int Q)^2}\int_{-\infty}^{y}\Lambda Q,\quad J_1(s)=(\varepsilon(s),\rho_1),
		\end{equation}
		where $Q$ is the ground state for \eqref{CCP}. Then we have:
		\begin{align}
		&\bigg|\frac{\lambda_{s}}{\lambda}+b-2\bigg((J_1)_s+\frac{1}{2}\frac{\lambda_s}{\lambda}J_1\bigg)\bigg|\nonumber\\
		&\lesssim (\omega+|b|)\Bigg[\bigg(\int\varepsilon^2e^{-\frac{|y|}{10}}\bigg)^{\frac{1}{2}}+|b|\Bigg]+\int\varepsilon^2e^{-\frac{|y|}{10}}.\label{CLOL}
		\end{align}
		\item Law of $b$: let
		\begin{equation}\begin{split}
		&\rho_2=\frac{16}{(\int Q)^2}\bigg(\frac{(\Lambda P,Q)}{\|\Lambda Q\|^2_{L^2}}\Lambda Q+P-\frac{1}{2}\int Q\bigg)-8\rho_1,\\
		&J_2(s)=(\varepsilon(s),\rho_2),
		\end{split}
		\end{equation}
		where $P$ was introduced in Proposition \ref{CP1}. Then we have
		\begin{align}
		&\bigg|b_s+2b^2+\omega_sG'(\omega)+b\bigg((J_2)_s+\frac{1}{2}\frac{\lambda_s}{\lambda}J_2\bigg)\bigg|\nonumber\\
		&\lesssim \int \varepsilon^2e^{-\frac{|y|}{10}}+(\omega+|b|)b^2,\label{CLOB}
		\end{align}
		where $G\in C^2$ with $G(0)=0$, $G'(0)=c_0>0$, for some universal constant $c_0$.
		\item Law of $\frac{b}{\lambda^2}$: let
		\begin{equation}\label{C204}
		\rho=4\rho_1+\rho_2\in\mathcal{Y}, \quad J(s)=(\varepsilon(s),\rho),
		\end{equation}
		then we have:
		\begin{align}
		&\bigg|\frac{d}{ds}\bigg(\frac{b}{\lambda^2}\bigg)+\frac{b}{\lambda^2}\bigg(J_s+\frac{1}{2}\frac{\lambda_s}{\lambda}J\bigg)+\frac{\omega_sG'(\omega)}{\lambda^2}\bigg|\nonumber\\
		&\lesssim \frac{1}{\lambda^2}\bigg(\int \varepsilon^2e^{-\frac{|y|}{10}}+(\omega+|b|)b^2\bigg).\label{CLOBOLS}
		\end{align}
	\end{enumerate}
\end{enumerate}
\end{proposition}
\begin{remark}
The proof of Proposition \ref{CP3} follows almost the same procedure as \cite[Lemma 2.7]{MMR1}. It is important that there is no a priori assumption on the upper bound of $\lambda(t)$. This fact ensures that Proposition \ref{CP3} can be used in all the 3 regimes%
\footnote{We will see in Section 4 that we can't expect any (finite) upper bound on the scaling parameter $\lambda(t)$ in both (Blow down) and (Exit) case.}%
.
\end{remark}
\begin{proof}[Proof]
{\it Proof of (1):} Equation \eqref{CMEQ} follows by direct computation from the equation of $u(t)$. 

{\it Proof of (2):} We write down the mass conservation law:
\begin{equation}
\label{C2006}
\int Q^2_{b,\omega}-\int Q^2+\int\varepsilon^2+2(\varepsilon,Q_{b,\omega})=\int u_0^2-\int Q^2.
\end{equation}
From \eqref{C25} and the orthogonality condition \eqref{COC}, we have
$$\int \varepsilon^2\lesssim |b|+\omega+|b|^{1-\beta}\|\varepsilon\|_{L^2}+\bigg|\int u_0^2-\int Q^2\bigg|.$$
then \eqref{CMC} follows from $\beta=\frac{3}{4}$.

Similarly, we use the energy conservation law and \eqref{C205} to obtain:
\begin{align*}
2\lambda^2 E_0=&2E(Q_{b,\omega})-2\int\varepsilon (Q_{b,\omega})_{yy}+\int\varepsilon^2_{y}-\frac{1}{3}\int\big[(Q_{b,\omega}+\varepsilon)^6-Q_{b,\omega}^6\big]\\
&+\frac{2\omega}{q+1}\int\big[|Q_{b,\omega}+\varepsilon|^{q+1}-|Q_{b,\omega}|^{q+1}\big]\\
=&O(|b|+\omega)+\int \varepsilon_{y}^2-2\int\varepsilon\big[(Q_{b,\omega}-\mathcal{Q}_{\omega})_{yy}+(Q_{b,\omega}^5-\mathcal{Q}_{\omega}^5)\\
&+\omega(Q_{b,\omega}|Q_{b,\omega}|^{q-1}-\mathcal{Q}_{\omega}|\mathcal{Q}_{\omega}|^{q-1})\big]-\frac{1}{3}\int\big[(Q_{b,\omega}+\varepsilon)^6-Q_{b,\omega}^6-6\varepsilon Q_{b,\omega}^5\big]\\
&+\frac{2\omega}{q+1}\int\big[|Q_{b,\omega}+\varepsilon|^{q+1}-|Q_{b,\omega}|^{q+1}-(q+1)\varepsilon Q_{b,\omega}|Q_{b,\omega}|^{q-1}\big]
\end{align*}
We estimate all terms in the above identity. By the definition of $Q_{b,\omega}$, we have:
\begin{align*}
&\bigg|\int\varepsilon\big[(Q_{b,\omega}-\mathcal{Q}_{\omega})_{yy}+(Q_{b,\omega}^5-\mathcal{Q}_{\omega}^5)+\omega(Q_{b,\omega}|Q_{b,\omega}|^{q-1}-\mathcal{Q}_{\omega}|\mathcal{Q}_{\omega}|^{q-1})\big]\bigg|\\
&\lesssim |b|\bigg(\int\varepsilon^2e^{-\frac{|y|}{10}}\bigg)^{\frac{1}{2}}+|b|^{1+2\beta}\int_{-2|b|^{-\beta}\leq y\leq0}|\varepsilon|\\
&\lesssim |b|+\int\varepsilon^2e^{-\frac{|y|}{10}}.
\end{align*}
For the nonlinear term, we use Gagliardo-Nirenberg's inequality to estimate:
\begin{align*}
&\bigg|\int\big[(Q_{b,\omega}+\varepsilon)^6-Q_{b,\omega}^6-6\varepsilon Q_{b,\omega}^5\big]\bigg|\lesssim \int \varepsilon^2\mathcal{Q}_{\omega}^4+\int\varepsilon^6+|b|\int\varepsilon^2\\
&\lesssim \int\varepsilon^2e^{-\frac{|y|}{10}}+|b|+\|\varepsilon\|_{L^2}^4\|\varepsilon_{y}\|_{L^2}^2,
\end{align*}
and
\begin{align*}
&\bigg|\omega\int\big[|Q_{b,\omega}+\varepsilon|^{q+1}-|Q_{b,\omega}|^{q+1}-(q+1)\varepsilon Q_{b,\omega}|Q_{b,\omega}|^{q-1}\big]\bigg|\\
&\lesssim \omega\bigg(|b|+\int\varepsilon^2e^{-\frac{|y|}{10}}+\int|\varepsilon|^{q+1}\bigg)\\
&\lesssim|b|+\int\varepsilon^2e^{-\frac{|y|}{10}}+\frac{\gamma}{\lambda^m}\|\varepsilon\|_{L^2}^{\frac{q+3}{2}}\|\varepsilon_{y}\|_{L^2}^{m+2}\\
&\lesssim |b|+\int\varepsilon^2e^{-\frac{|y|}{10}}+\gamma\frac{\|\varepsilon_{y}\|_{L^2}^{m+2}}{\lambda^m}.
\end{align*}
Collecting all the estimates above, we obtain \eqref{CEC}.

{\it Proof of (3):} Let us differentiate the orthogonality conditions
$$(\varepsilon(t),\Lambda \mathcal{Q}_{\omega(t)})=(\varepsilon(t),y\Lambda \mathcal{Q}_{\omega(t)})=0.$$
Note that
$$\frac{d}{ds}(\varepsilon,\Lambda \mathcal{Q}_{\omega})=(\varepsilon_s,\Lambda \mathcal{Q}_{\omega})+\omega_s(\varepsilon,\Lambda Z_{\omega}),$$
where $Z_{\omega}=\partial \mathcal{Q}_{\omega}/\partial \omega\in\mathcal{Y}$.
So we have:
\begin{align*}
&\bigg|\bigg(\frac{\lambda_s}{\lambda}+b\bigg)-\frac{(\varepsilon,L_{\omega}(\Lambda \mathcal{Q}_{\omega})')}{\|\Lambda \mathcal{Q}_{\omega}\|^2_{L^2}}\bigg|+\bigg|\bigg(\frac{x_s}{\lambda}-1\bigg)-\frac{(\varepsilon,L_{\omega}(y\Lambda \mathcal{Q}_{\omega})')}{\|\Lambda \mathcal{Q}_{\omega}\|^2_{L^2}}\bigg|\\
&\lesssim\bigg(\bigg|\frac{\lambda_s}{\lambda}+b\bigg|+\bigg|\frac{x_s}{\lambda}-1\bigg|+|b|\bigg)\times\Bigg(\omega+|b|+\bigg(\int\varepsilon^2e^{-\frac{|y|}{10}}\bigg)^{\frac{1}{2}}\Bigg)\\
&\quad+|b_s|+|\omega_s|+\int\varepsilon^2e^{-\frac{|y|}{10}}+\int\varepsilon^5e^{-\frac{9|y|}{10}}+\int|\varepsilon|^qe^{-\frac{9|y|}{10}}.
\end{align*}
For the nonlinear term, we use Sobolev embedding and the a priori smallness \eqref{C26}:
$$\|\varepsilon e^{-\frac{|y|}{4}}\|^2_{L^{\infty}}\leq\|\varepsilon e^{-\frac{3|y|}{4(q-2)}}\|^2_{L^{\infty}}\lesssim \int (\partial_{y}\varepsilon^2+\varepsilon^2)e^{-\frac{3|y|}{4(q-2)}}\ll1,$$
to estimate
\begin{equation}
\label{C27}
\int\varepsilon^5e^{-\frac{9|y|}{10}}+\int|\varepsilon|^qe^{-\frac{9|y|}{10}}\lesssim \Big(\|\varepsilon e^{-\frac{|y|}{4}}\|^3_{L^{\infty}}+\|\varepsilon e^{-\frac{3|y|}{4(q-2)}}\|^{q-2}_{L^{\infty}}\Big)\int\varepsilon^2e^{-\frac{|y|}{10}}.
\end{equation}
Here we use the basic fact that $q>5$.

For $\omega_s$, we have
\begin{equation}\label{C28}
\omega_s=-m\omega\frac{\lambda_s}{\lambda}=m\omega b-m\omega\bigg(\frac{\lambda_s}{\lambda}+b\bigg)
\end{equation}
The above estimates imply that
\begin{equation}
\label{C29}
\bigg|\frac{\lambda_s}{\lambda}+b\bigg|+\bigg|\frac{x_s}{\lambda}-1\bigg|\lesssim(\omega+|b|)|b|+|b_s|+\bigg(\int\varepsilon^2e^{-\frac{|y|}{10}}\bigg)^{\frac{1}{2}},
\end{equation}
and
\begin{align}
&\bigg|\bigg(\frac{\lambda_s}{\lambda}+b\bigg)-\frac{(\varepsilon,L_{\omega}(\Lambda \mathcal{Q}_{\omega})')}{\|\Lambda \mathcal{Q}_{\omega}\|^2_{L^2}}\bigg|+\bigg|\bigg(\frac{x_s}{\lambda}-1\bigg)-\frac{(\varepsilon,L_{\omega}(y\Lambda \mathcal{Q}_{\omega})')}{\|\Lambda \mathcal{Q}_{\omega}\|^2_{L^2}}\bigg|\nonumber\\
&\lesssim (\omega+|b|)\Bigg[\bigg(\int\varepsilon^2e^{-\frac{|y|}{10}}\bigg)^{\frac{1}{2}}+|b|\Bigg]+\int\varepsilon^2e^{-\frac{|y|}{10}}.\label{C210}
\end{align}
Next, let us differentiate the relation $(\varepsilon,\mathcal{Q}_{\omega})=0$ to obtain:
{
\begin{align}\label{ADD}
0=&(\varepsilon,\mathcal{Q}_{\omega})_s=(\varepsilon_s,\mathcal{Q}_{\omega})+\bigg(\varepsilon,\omega_s\frac{\partial \mathcal{Q}_{\omega}}{\partial\omega}\bigg)\nonumber\\
=&\omega_s\bigg(\varepsilon,\frac{\partial \mathcal{Q}_{\omega}}{\partial\omega}\bigg)-\big(\varepsilon,L_{\omega}(\mathcal{Q}_{\omega}')\big)-b(\Lambda\varepsilon,\mathcal{Q}_{\omega})\nonumber\\
&+\bigg(\frac{\lambda_s}{\lambda}+b\bigg)\Big[(\Lambda Q_{b,\omega},\mathcal{Q}_{\omega})+(\Lambda\varepsilon,\mathcal{Q}_{\omega})\Big]+\bigg(\frac{x_s}{\lambda}-1\bigg)\Big[(Q_{b,\omega}',\mathcal{Q}_{\omega})+(\varepsilon',\mathcal{Q}_{\omega})\Big]\nonumber\\
&-b_s\Big[(P_{\omega}\chi_b,\mathcal{Q}_{\omega})+(\beta y\chi_b',\mathcal{Q}_{\omega})\Big]-\omega_s\bigg(\mathcal{Q}_{\omega},\frac{\partial Q_{b,\omega}}{\partial \omega}\bigg)\nonumber\\
&+(\Psi_{b,\omega},\mathcal{Q}_{\omega})+\big(R_b(\varepsilon)+R_{\rm NL}(\varepsilon),\mathcal{Q}_{\omega}'\big).
\end{align}

Injecting the following facts:
\begin{align*}
&(P_{\omega}\chi_b,\mathcal{Q}_{\omega})+(\beta y\chi_b',\mathcal{Q}_{\omega})=(P_{\omega},\mathcal{Q}_{\omega})+O(b^{10})\sim 1,\\
&L_{\omega}\mathcal{Q}_{\omega}'=0,\quad(\mathcal{Q}_{\omega},\Lambda \mathcal{Q}_{\omega})=(\mathcal{Q}_{\omega},\mathcal{Q}_{\omega}')=(\varepsilon, \Lambda \mathcal{Q}_{\omega})=0,\\
&\big|\big(R_b(\varepsilon)+R_{\rm NL}(\varepsilon),\mathcal{Q}_{\omega}'\big)\big|\lesssim (\omega+|b|)\bigg(\int\varepsilon^2e^{-\frac{|y|}{10}}\bigg)^{1/2}+\int\varepsilon^2e^{-\frac{|y|}{10}}
\end{align*}
and \eqref{C23}, \eqref{C24}, \eqref{C27}, \eqref{C28} into \eqref{ADD}, we obtain:}
\begin{align}
|b_s|\lesssim&\bigg(\omega+|b|+\bigg(\int\varepsilon^2e^{-\frac{|y|}{10}}\bigg)^{\frac{1}{2}}\bigg)\times\bigg(\bigg|\frac{\lambda_s}{\lambda}+b\bigg|+\bigg|\frac{x_s}{\lambda}-1\bigg|\bigg)\nonumber\\
&+(\omega+|b|)\Bigg[\bigg(\int\varepsilon^2e^{-\frac{|y|}{10}}\bigg)^{\frac{1}{2}}+|b|\Bigg]+\int\varepsilon^2e^{-\frac{|y|}{10}}.\label{C211}
\end{align}
Combining \eqref{C28}, \eqref{C29} and \eqref{C211}, we get \eqref{CMS1} and \eqref{CMS2}.

{\it Proof of (4):} Firstly, we claim the following sharp equation:
\begin{align}
&b_s+2b^2+\omega_sG'(\omega)-\frac{16b}{(\int Q)^2}\bigg[\frac{(\Lambda P,Q)}{\|\Lambda Q\|^2_{L^2}}(\varepsilon,L(\Lambda Q)')+20(\varepsilon,PQ^3Q')\bigg]\nonumber\\
&=O\bigg(b^2(\omega+|b|)+\int\varepsilon^2e^{-\frac{|y|}{10}}\bigg).\label{C212}
\end{align} 
To prove this, we take the scalar product of \eqref{CMEQ} with $\mathcal{Q}_{\omega}$. We keep track of all terms up to $b^2$.

First, from \eqref{C23}, we have
\begin{align}
(\Psi_{b,\omega},\mathcal{Q}_{\omega})&=-b^2\big((10P_{\omega}^2\mathcal{Q}_{\omega}^3)_{y}+\Lambda P_{\omega},\mathcal{Q}_{\omega}\big)+O\big(b^2(|b|+\omega)\big)\nonumber\\
&=-b^2\big((10P^2Q^3)_{y}+\Lambda P,Q\big)+O\big(b^2(|b|+\omega)\big)\nonumber\\
&=-\frac{b^2}{8}\|Q\|_{L^1}^2+O\big(b^2(|b|+\omega)\big),\label{C213}
\end{align}
where for the last step use the following computation:
\begin{align*}
(\Lambda P,Q)&=-(P,\Lambda Q)=-(P,(LP)')=(P,(P''-P+5Q^4P)')\\
&=(P,P'''-P')+10\int Q^3Q'P^2,
\end{align*}
and from Proposition \ref{CP1}, we obtain:
$$\big((10P^2Q^3)_{y}+\Lambda P,Q\big)=\frac{1}{2}\lim_{y\rightarrow-\infty}P^2=\frac{1}{8}\|Q\|_{L^1}^2.$$
Next, from Proposition \ref{CP2}, we have:
\begin{align}
(b_s\frac{\partial Q_{b,\omega}}{\partial b},\mathcal{Q}_{\omega})&=b_s((\chi_b+\beta y\chi_b')P_{\omega},\mathcal{Q}_{\omega})=b_s(P_{\omega},\mathcal{Q}_{\omega})+O(b^{10})\nonumber\\
&=\frac{b_s}{16}\|Q\|_{L^1}^2+F(\omega)b_s+O(b^{10}),\label{C214}
\end{align}
where $F$ is  the $C^1$ function introduced in Proposition \ref{CP2}.
From Lemma \ref{CL1}, we have
\begin{align*}
(Z_{\omega},\mathcal{Q}_{\omega})&=-\frac{1}{2}(L_{\omega}Z_{\omega},\Lambda \mathcal{Q}_{\omega})+O(\omega)=\frac{1}{2}\int (\Lambda \mathcal{Q}_{\omega})\mathcal{Q}_{\omega}|\mathcal{Q}_{\omega}|^{q-1}+O(\omega)\\
&=\frac{q-1}{4(q+1)}\int|\mathcal{Q}_{\omega}|^{q+1}+O(\omega)>0.
\end{align*}
Then from \eqref{CMS2}, we have:
\begin{align}
(\omega_s\frac{\partial Q_{b,\omega}}{\partial \omega},\mathcal{Q}_{\omega})&=\omega_s\frac{1}{2}\frac{\partial\|\mathcal{Q}_{\omega}\|^2_{L^2}}{\partial\omega}+O(|b\omega_s|)\nonumber\\
&=\omega_s\widetilde{G}'(\omega)+O\bigg(b^2(\omega+|b|)+\int\varepsilon^2e^{-\frac{|y|}{10}}\bigg),\label{C215}
\end{align}
with $\widetilde{G}(\omega)=\frac{1}{2}(\|\mathcal{Q}_{\omega}\|^2_{L^2}-\|Q\|_{L^2}^2)$. It is easy to check $\widetilde{G}(0)=0$, $\widetilde{G}\in C^1$, and 
$$\widetilde{G}'(0)=(Z_{\omega},\mathcal{Q}_{\omega})\Big|_{\omega=0}=\frac{q-1}{4(q+1)}\int|Q|^{q+1}>0.$$
Next, from Proposition \ref{CP2} we have:
$$|(Q_{b,\omega}'+\varepsilon_{y},\mathcal{Q}_{\omega})|\lesssim\bigg(\int\varepsilon^2e^{-\frac{|y|}{10}}\bigg)^{\frac{1}{2}}+|(\mathcal{Q}_{\omega}',\mathcal{Q}_{\omega})|+|(P'_{\omega},\mathcal{Q}_{\omega})|+b^{10},$$
which together with \eqref{CMS1} implies that
\begin{equation}
\label{C216}
\bigg|\bigg(\frac{x_s}{\lambda}-1\bigg)(Q'_{b,\omega}+\varepsilon_{y},\mathcal{Q}_{\omega})\bigg|\lesssim b^2(\omega+|b|)+\int\varepsilon^2e^{-\frac{|y|}{10}}.
\end{equation}
For the small linear term, we have:
\begin{align}
\int R_{b}(\varepsilon)\mathcal{Q}_{\omega}'&=20b\int P_{\omega}\mathcal{Q}_{\omega}^3\mathcal{Q}_{\omega}'\varepsilon+|b|(\omega+|b|)O\bigg(\int\varepsilon^2e^{-\frac{|y|}{10}}\bigg)^{\frac{1}{2}}\nonumber\\
&=20b\int PQ^3Q'\varepsilon+|b|(\omega+|b|)O\bigg(\int\varepsilon^2e^{-\frac{|y|}{10}}\bigg)^{\frac{1}{2}}.\label{C217}
\end{align}
Since the nonlinear term can be estimated with the help of \eqref{C27}, we then have:
\begin{align*}
&b_s+\frac{2b^2+\omega_s\widetilde{G}'(\omega)}{1+H(\omega)}-\frac{16}{(1+H(\omega))(\int Q)^2}\bigg[(\Lambda  Q_{b,\omega},\mathcal{Q}_{\omega})\bigg(\frac{\lambda_s}{\lambda}+b\bigg)+20b(\varepsilon,PQ^3Q')\bigg]\nonumber\\
&=O\bigg(b^2(\omega+|b|)+\int\varepsilon^2e^{-\frac{|y|}{10}}\bigg),
\end{align*}
where
$$H(\omega)=\frac{16}{(\int Q)^2}F(\omega).$$
From \eqref{C210} we have
\begin{equation*}
\bigg|\bigg(\frac{\lambda_s}{\lambda}+b\bigg)-\frac{(\varepsilon,L(\Lambda Q)')}{\|\Lambda Q\|^2_{L^2}}\bigg|\lesssim (\omega+|b|)\Bigg[\bigg(\int\varepsilon^2e^{-\frac{|y|}{10}}\bigg)^{\frac{1}{2}}+|b|\Bigg]+\int\varepsilon^2e^{-\frac{|y|}{10}}.
\end{equation*}
Moreover, we have
\begin{equation*}
\big|(\Lambda Q_{b,\omega},\mathcal{Q}_{\omega})-b(\Lambda P,Q)\big|\lesssim b^{10}+|b(\Lambda P,Q)-b(\Lambda P_{\omega},\mathcal{Q}_{\omega})|\lesssim |b|(\omega+|b|).
\end{equation*}
We then conclude that
\begin{align}
&b_s+\frac{2b^2+\omega_s\widetilde{G}'(\omega)}{1+H(\omega)}-\frac{16b}{(1+H(\omega))(\int Q)^2}\bigg[\frac{(\Lambda P,Q)}{\|\Lambda Q\|^2_{L^2}}(\varepsilon,L(\Lambda Q)')+20(\varepsilon,PQ^3Q')\bigg]\nonumber\\
&=O\bigg(b^2(\omega+|b|)+\int\varepsilon^2e^{-\frac{|y|}{10}}\bigg)\label{C218}.
\end{align}

Finally, since $H\in C^1$, $H(0)=0$, it is to check that the following function
$$G(\omega)=\int_0^{\omega}\frac{\widetilde{G}'(x)}{1+H(x)}\,dx$$
satisfies $G\in C^2$, $G(0)=0$, $G'(0)=c_0>0$. Then, \eqref{C218} implies \eqref{C212} immediately.

Now, we turn to the proof of \eqref{CLOL}, \eqref{CLOB} and \eqref{CLOBOLS}. For all $f\in\mathcal{Y}$, independent of $s$, $(\varepsilon,\int_{-\infty}^{y}f)$ is well defined due to \eqref{C202}. Moreover, we have:
\begin{align*}
\frac{d}{ds}\bigg(\varepsilon,\int_{-\infty}^yf\bigg)=&-(\varepsilon,L_{\omega}f)+\bigg(\frac{\lambda_s}{\lambda}+b\bigg)\bigg(\Lambda Q_{b,\omega},\int_{-\infty}^{y}f\bigg)+\frac{\lambda_s}{\lambda}\bigg(\Lambda\varepsilon,\int_{-\infty}^{y}f\bigg)\\
&-\bigg(\frac{x_s}{\lambda}-1\bigg)(Q_{b,\omega}+\varepsilon,f)-\bigg(b_s\frac{\partial Q_{b,\omega}}{\partial b}+\omega_s\frac{\partial Q_{b,\omega}}{\partial \omega},\int_{-\infty}^{y}f\bigg)\\
&+\bigg(\Psi_{b,\omega},\int_{-\infty}^{y}\bigg)+\big(R_{b}(\varepsilon)+R_{\rm NL}(\varepsilon),f\big).
\end{align*}
Using \eqref{CMS1}, \eqref{CMS2}, \eqref{C27} and Proposition \ref{CP2}, we have:
\begin{align}
\frac{d}{ds}\bigg(\varepsilon,\int_{-\infty}^yf\bigg)=&-(\varepsilon,Lf)+\bigg(\frac{\lambda_s}{\lambda}+b\bigg)\bigg(\Lambda Q,\int_{-\infty}^{y}f\bigg)+\bigg(\frac{x_s}{\lambda}-1\bigg)(f,Q)\nonumber\\
&-\frac{1}{2}\frac{\lambda_s}{\lambda}\bigg(\varepsilon,\int_{-\infty}^{y}f\bigg)+O\Bigg((|b|+\omega)\bigg[\bigg(\int\varepsilon^2e^{-\frac{|y|}{10}}\bigg)^{\frac{1}{2}}\bigg]\Bigg)\nonumber\\
&+O\big((|b|+\omega)|b|\big)+O\bigg(\int\varepsilon^2e^{-\frac{|y|}{10}}\bigg).\label{C221}
\end{align}

-Proof of \eqref{CLOL}: We apply \eqref{C221} to $f=\Lambda Q$, using the following fact
$$L\Lambda Q=-2Q,\;\bigg(\Lambda Q,\int_{-\infty}^{y}\Lambda Q\bigg)=\frac{1}{8}\bigg(\int Q\bigg)^{2},\;\bigg(Q',\int_{-\infty}^{y}\Lambda Q\bigg)=0,$$
to obtain:
\begin{align*}
2(J_1)_s=&\frac{16(\varepsilon,Q)}{(\int Q)^2}+\bigg(\frac{\lambda_s}{\lambda}+b\bigg)-\frac{\lambda_s}{\lambda}J_1\\
&+O\Bigg((|b|+\omega)\bigg[\bigg(\int\varepsilon^2e^{-\frac{|y|}{10}}\bigg)^{\frac{1}{2}}+|b|\bigg]+\int\varepsilon^2e^{-\frac{|y|}{10}}\Bigg).
\end{align*}
Then \eqref{CLOL} follows immediately from the orthogonality condition \eqref{COC}.

-Proof of \eqref{CLOB}: We apply \eqref{C221} to $f=\rho_2'$.  Then from Lemma \ref{CL0} and Proposition \ref{CP1}, we have:
\begin{align*}
(\Lambda Q,\rho_2)&=\frac{16}{(\int Q)^2}\bigg(\frac{(\Lambda P,Q)}{\|\Lambda Q\|_{L^2}}\Lambda Q+P-\frac{1}{2}\int Q,\Lambda Q\bigg)-\frac{32}{(\int Q)^2}\bigg(\Lambda Q,\int_{-\infty}^{y}\Lambda Q\bigg)\\
&=\frac{16}{(\int Q)^2}\big[(\Lambda P,Q)+(\Lambda Q,P)\big]+\frac{4\|Q\|_{L^1}^2}{(\int Q)^2}-\frac{16}{(\int Q)^2}\bigg(\int \Lambda Q\bigg)^2=0,\\
(\rho',Q)&=\frac{16}{(\int Q)^2}\bigg(\frac{(\Lambda P,Q)}{\|\Lambda Q\|_{L^2}}(\Lambda Q)'+P',Q\bigg)-8(\rho'_1,Q).\\
\end{align*}
Next, from
$$L(P')=(LP)'+20Q'Q^3P=\Lambda Q+20Q'Q^3P,$$
and the orthogonality condition $(\varepsilon,\Lambda \mathcal{Q}_{\omega})=0$, we have
\begin{align*}
(\varepsilon,L\rho_2')&=\frac{16}{(\int Q)^2}\bigg(\varepsilon,L\bigg[\frac{(\Lambda P,Q)}{\|\Lambda Q\|_{L^2}}(\Lambda Q)'+P'\bigg]\bigg)-8(\varepsilon,L\rho_1')\\
&=\frac{16}{(\int Q)^2}\bigg[\frac{(\Lambda P,Q)}{\|\Lambda Q\|^2_{L^2}}(\varepsilon,L(\Lambda Q)')+20(\varepsilon,PQ^3Q')\bigg]+O(\omega)\bigg(\int\varepsilon^2e^{-\frac{|y|}{10}}\bigg)^{\frac{1}{2}},
\end{align*}
Injecting all the above estimates into \eqref{C221} with $f=\rho_2'$, we obtain:
\begin{align}\label{C222}
(J_2)_s=&-\frac{16}{(\int Q)^2}\bigg[\frac{(\Lambda P,Q)}{\|\Lambda Q\|^2_{L^2}}(\varepsilon,L(\Lambda Q)')+20(\varepsilon,PQ^3Q')\bigg]-\frac{1}{2}\frac{\lambda_s}{\lambda}J_2\nonumber\\
&+O\Bigg((|b|+\omega)\bigg[\bigg(\int\varepsilon^2e^{-\frac{|y|}{10}}\bigg)^{\frac{1}{2}}+|b|\bigg]+\int\varepsilon^2e^{-\frac{|y|}{10}}\Bigg).
\end{align}
Then \eqref{CLOB} follows from \eqref{C212} and \eqref{C222}.

-Proof of \eqref{CLOBOLS}: From \eqref{CLOL} and \eqref{CLOBOLS},
\begin{align*}
\frac{d}{ds}\bigg(\frac{b}{\lambda^2}\bigg)=&\frac{b_s+2b^2}{\lambda^2}-\frac{2b}{\lambda^2}\bigg(\frac{\lambda_s}{\lambda}+b\bigg)\\
=&-\frac{b}{\lambda^2}\bigg[(J_2)_s+\frac{1}{2}\frac{\lambda_s}{\lambda}J_2\bigg]-\frac{2b}{\lambda^2}\bigg[2(J_1)_s+\frac{\lambda_s}{\lambda}J_1\bigg]-\frac{\omega_sG'(\omega)}{\lambda^2}\\
&+O\Bigg(\frac{1}{\lambda^2}\bigg(\int \varepsilon^2e^{-\frac{|y|}{10}}+(\omega+|b|)b^2\bigg)\Bigg)\\
=&-\frac{b}{\lambda^2}\bigg[J_s+\frac{1}{2}\frac{\lambda_s}{\lambda}J\bigg]-\frac{\omega_sG'(\omega)}{\lambda^2}+O\Bigg(\frac{1}{\lambda^2}\bigg(\int \varepsilon^2e^{-\frac{|y|}{10}}+(\omega+|b|)b^2\bigg)\Bigg)
\end{align*}
which is exactly \eqref{CLOBOLS}.

Finally, it is easy to check that $\lim_{|y|\rightarrow+\infty}\rho(y)=0$, which implies that $\rho\in\mathcal{Y}$.
\end{proof}

\section{Monotonicity formula}
In this section, we will introduce the monotonicity tools developed in \cite{MM3} and \cite{MMR1}. This is the key technical argument of the analysis for solution near the soliton. 
\subsection{Pointwise monotonicity} Let $(\varphi_i)_{i=1,2},\psi\in C^{\infty}(\mathbb{R})$ be such that:
\begin{gather}
\varphi_i(y)=
\begin{cases}
e^y,\text{ for }y<-1,\\
1+y,\text{ for }-\frac{1}{2}<y<\frac{1}{2},\\
y^i,\text{ for }y>2,
\end{cases}
\quad \varphi'(y)>0,\text{ for all }y\in\mathbb{R},\\
\psi(y)=
\begin{cases}
e^{2y},\text{ for }y<-1,\\
1,\text{ for }y>-\frac{1}{2},
\end{cases}
\quad \psi'(y)\geq0,\text{ for all }y\in\mathbb{R}.
\end{gather}  
Let $B>100$ be a large universal constant to be chosen later. We then define the following weight function:
\begin{equation}
\psi_B(y)=\psi\bigg(\frac{y}{B}\bigg),\quad \varphi_{i,B}(y)=\varphi\bigg(\frac{y}{B}\bigg),
\end{equation} 
and the following weighted Sobolev norm of $\varepsilon$:
\begin{gather}
\mathcal{N}_i(s)=\int\bigg(\varepsilon_{y}^2(s,y)\psi_B(y)+\varepsilon^2(s,y)\varphi_{i,B}(y)\bigg)\,dy,\quad i=1,2,\\
\mathcal{N}_{i,\rm loc}(s)=\int \varepsilon^2(s,y)\varphi'_{i,B}(y)\,dy,\quad i=1,2.
\end{gather}

Then we have the following monotonicity:

\begin{proposition}[Monotonicity formula]\label{CP4}
There exist universal constants $\mu>0$, $B=B(q)>100$ and $0<\kappa\ll1$, such that the following holds. Let $u(t)$ be a solution of \eqref{CCPG} satisfying \eqref{C203} on $[0,t_0]$, and hence the geometrical decomposition \eqref{CGD} holds on $[0,t_0]$. Let $s_0=s(t_0)$, and assume the following a priori bounds hold for all $s\in[0,s_0]$:\\
(H1) Scaling invariant bounds:
\begin{equation}
\label{C31}
\omega(s)+|b(s)|+\mathcal{N}_2(s)+\|\varepsilon(s)\|_{L^2}+\omega(s)\|\varepsilon_{y}(s)\|_{L^2}^m\leq\kappa;
\end{equation}
(H2) Bounds related to $H^1$ scaling:
\begin{equation}
\label{C32}
\frac{\omega(s)+|b(s)|+\mathcal{N}_2(s)}{\lambda^2(s)}\leq\kappa;
\end{equation}
(H3) $L^2$ weighted bound on the right:
\begin{equation}
\label{C33}
\int_{y>0}y^{10}\varepsilon^2(s,y)\,dy\leq 50\bigg(1+\frac{1}{\lambda^{10}(s)}\bigg).
\end{equation}
We define the Lyapounov functionals for $(i,j)\in\{1,2\}^2$ as following:
\begin{align}
\label{C34}
\mathcal{F}_{i,j}=\int&\bigg(\varepsilon_{y}^2\psi_B+(1+\mathcal{J}_{i,j})\varepsilon^2\varphi_{i,B}-\frac{1}{3}\psi_B\big[(Q_{b,\omega}+\varepsilon)^6-Q_{b,\omega}^6-6\varepsilon Q_{b,\omega}^5\big]\nonumber\\
&+\frac{2\omega}{q+1}\big[|Q_{b,\omega}+\varepsilon|^{q+1}-|Q_{b,\omega}|^{q+1}-(q+1)\varepsilon Q_{b,\omega}|Q_{b,\omega}|^{q-1}\big]\psi_B\bigg),
\end{align}
with%
\footnote{Recall that $J_1$ was defined in \eqref{C2005}.}%
\begin{equation}
\label{C35}
\mathcal{J}_{i,j}=(1-J_1)^{-4(j-1)-2i}-1.
\end{equation}
Then the following estimates hold on $[0,s_0]$:
\begin{enumerate}
	\item Scaling invariant Lyapounov control: for $i=1,2$,
	\begin{equation}
	\label{CMF1}
	\frac{d\mathcal{F}_{i,1}}{ds}+\mu\int(\varepsilon_{y}^2+\varepsilon^2)\varphi'_{i,B}\lesssim_B b^2(\omega^2+b^2).
	\end{equation}
	\item $H^1$ scaling Lyapounov control: for $i=1,2$,
	\begin{equation}
	\label{CMF2}
	\frac{d}{ds}\bigg(\frac{\mathcal{F}_{i,2}}{\lambda^2}\bigg)+\frac{\mu}{\lambda^2}\int(\varepsilon_{y}^2+\varepsilon^2)\varphi'_{i,B}\lesssim_B \frac{b^2(\omega^2+b^2)}{\lambda^2}.
	\end{equation}
	\item Coercivity and pointwise bounds: there holds for all $(i,j)\in\{1,2\}^2$,
	\begin{gather}
	\mathcal{N}_i\lesssim\mathcal{F}_{i,j}\lesssim \mathcal{N}_i,\label{CCOER}\\
	|J_i|+|\mathcal{J}_{i,j}|\lesssim\mathcal{N}_2^{\frac{1}{2}}\label{C36}.
	\end{gather}
\end{enumerate}
\end{proposition}
\begin{remark}
The proof of Proposition \ref{CP4} is almost the same to \cite[Proposition 3.1]{MMR1}. The only difference here is the additional terms involving $\omega$.
\end{remark}
\begin{remark}
Similar as Proposition \ref{CP3}, we do not assume any a priori control on the upper bound of $\lambda(s)$ so that the monotonicity formula can be used in all the 3 cases.
\end{remark}
\begin{remark}
As mentioned in \cite[Proposition 3.1]{MMR1}, the weight function $\psi$ decays faster than $\varphi_i$ on the left. As a result, $\mathcal{N}_2$ and $\mathcal{F}_{i,j}$ do not control $\int\varepsilon^2_{y}\varphi_{i,B}'$ (See \cite[Remark 3.5]{MMR1} for more details).
\end{remark}

\begin{proof}[Proof of Proposition \ref{CP4}]The proof of \eqref{CCOER} and \eqref{C36} is exactly the same as \cite[Proposition 3.1]{MMR1}. We only need to prove \eqref{CMF1} and \eqref{CMF2}. To do this, we compute directly to obtain that for all $(i,j)\in\{1,2\}^2$,
\begin{equation}
\label{C310}
\lambda^{2(j-1)}\bigg(\frac{\mathcal{F}_{i,j}}{\lambda^{2(j-1)}}\bigg)_s=f_1+f_2+f_3+f_4+f_5,
\end{equation}
where
\begin{align*}
&f_1=2\int\bigg(\varepsilon_s-\frac{\lambda_s}{\lambda}\Lambda\varepsilon\bigg)\big(-(\psi_B\varepsilon)_{y}+\varepsilon\varphi_{i,B}-\psi_B\Delta_{b,\omega}(\varepsilon)\big),\\
&f_2=2\int\bigg(\varepsilon_s-\frac{\lambda_s}{\lambda}\Lambda\varepsilon\bigg)\varepsilon\mathcal{J}_{i,j}\varphi_{i,B},\\
&f_3=2\frac{\lambda_s}{\lambda}\int\Lambda\varepsilon\big(-(\psi_B\varepsilon_{y})_{y}+(1+\mathcal{J}_{i,j})\varepsilon\varphi_{i,B}-\psi_B\Delta_{b,\omega}(\varepsilon)\big)\\
&\qquad\quad+(\mathcal{J}_{i,j})_s\int\varphi_{i,B}\varepsilon^2-2(j-1)\frac{\lambda_s}{\lambda}\mathcal{F}_{i,j},\\
&f_4=-2\int\psi_B(Q_{b,\omega})_s\big[\Delta_{b,\omega}-5\varepsilon Q_{b,\omega}^4+q\omega\varepsilon|Q_{b,\omega}|^{q-1}\big],\\
&f_5=\frac{2\omega_s}{q+1}\int\big[|Q_{b,\omega}+\varepsilon|^{q+1}-|Q_{b,\omega}|^{q+1}-(q+1)\varepsilon Q_{b,\omega}|Q_{b,\omega}|^{q-1}\big]\psi_B,\\
&\Delta_{b,\omega}(\varepsilon)=(Q_{b,\omega}+\varepsilon)^5-Q_{b,\omega}^5-\omega(Q_{b,\omega}+\varepsilon)|Q_{b,\omega}+\varepsilon|^{q-1}+\omega Q_{b,\omega}|Q_{b,\omega}|^{q-1}.
\end{align*}

Our goal is to show that for some $\mu_0>0$,
\begin{gather}
\frac{d}{ds}f_1\leq-\mu_0\int\bigg((\varepsilon^2+\varepsilon^2)\varphi_{i,B}'+\varepsilon_{yy}^2\psi_B'\bigg)+Cb^2(\omega^2+b^2),\label{C311}\\
\bigg|\frac{d}{ds}f_k\bigg|\leq \frac{\mu_0}{10}\int\bigg((\varepsilon^2+\varepsilon^2)\varphi_{i,B}'+\varepsilon_{yy}^2\psi_B'\bigg)+Cb^2(\omega^2+b^2),\text{ for }k=2,3,4,5.\label{C312}
\end{gather}

The following properties will be used several times in this paper%
\footnote{See \cite[Section 3]{MMR1} for more details.}%
:\begin{gather}
|\varphi'''_i(y)|+|\varphi''_i(y)|+|\psi'''(y)|+|y\psi'(y)|+|\psi(y)|\lesssim \varphi'_i\lesssim \varphi_i,\text{ for all }y\in\mathbb{R},\\
e^{|y|}(\psi(y)+|\psi'(y)|)\lesssim\varphi'_i\sim\varphi_i,\text{ for all }y\in(-\infty,\frac{1}{2}],\\
\mathcal{N}_{1,\rm loc}\lesssim\mathcal{N}_{2,\rm loc}\lesssim\mathcal{N}_1\lesssim\mathcal{N}_2,\quad \int\varepsilon^2\varphi_{1,B}\,dy\lesssim\mathcal{N}_{2,\rm loc},\label{C38}\\
\int_{y>0}y^2\varepsilon^2(s)\lesssim\bigg(1+\frac{1}{\lambda^{\frac{10}{9}}(s)}\bigg)\mathcal{N}_{2,\rm loc}^{\frac{8}{9}}(s).\label{C39}
\end{gather}

{\bf Control of $f_1$.} First, we rewrite $f_1$ using the equation of $\varepsilon$ in the following form:
\begin{align}
\varepsilon_s-\frac{\lambda_s}{\lambda}\Lambda\varepsilon=&\big(-\varepsilon_{yy}+\varepsilon-\Delta_{b,\omega}(\varepsilon)\big)_{y}+\bigg(\frac{\lambda_s}{\lambda}+b\bigg)\Lambda Q_{b,\omega}\nonumber\\
&+\bigg(\frac{x_s}{\lambda}-1\bigg)(Q_{b,\omega}+\varepsilon)_{y}-b_s\frac{\partial Q_{b,\omega}}{\partial b}-\omega_s\frac{\partial Q_{b,\omega}}{\partial \omega}+\Psi_{b,\omega},\label{C320}
\end{align}
where $-\Psi_{b,\omega}=b\Lambda Q_{b,\omega}+(Q_{b,\omega}''-Q_{b,\omega}+Q_{b,\omega}^5-\omega Q_{b,\omega}|Q_{b,\omega}|^{q-1})_{y}$. This yields:
$$f_1=f_{1,1}+f_{1,2}+f_{1,3}+f_{1,4}+f_{1,5},$$
with
\begin{align*}
f_{1,1}=&2\int\big(-\varepsilon_{yy}+\varepsilon-\Delta_{b,\omega}(\varepsilon)\big)_{y}\big(-(\psi_B\varepsilon_{y})_{y}+\varepsilon\varphi_{i,B}-\psi_B\Delta_{b,\omega}(\varepsilon)\big),\\
f_{1,2}=&2\bigg(\frac{\lambda_s}{\lambda}+b\bigg)\int\Lambda Q_{b,\omega}\big(-(\psi_B\varepsilon_{y})_{y}+\varepsilon\varphi_{i,B}-\psi_B\Delta_{b,\omega}(\varepsilon)\big),\\
f_{1,3}=&2\bigg(\frac{x_s}{\lambda}-1\bigg)\int (Q_{b,\omega}+\varepsilon)_{y}\big(-(\psi_B\varepsilon_{y})_{y}+\varepsilon\varphi_{i,B}-\psi_B\Delta_{b,\omega}(\varepsilon)\big),\\
f_{1,4}=&-2\int\bigg(b_s\frac{\partial Q_{b,\omega}}{\partial b}+\omega_s\frac{\partial Q_{b,\omega}}{\partial \omega}\bigg)\big(-(\psi_B\varepsilon_{y})_{y}+\varepsilon\varphi_{i,B}-\psi_B\Delta_{b,\omega}(\varepsilon)\big),\\
f_{1,5}=&2\int\Psi_{b,\omega}\big(-(\psi_B\varepsilon_{y})_{y}+\varepsilon\varphi_{i,B}-\psi_B\Delta_{b,\omega}(\varepsilon)\big).
\end{align*}

For the term $f_{1,1}$, we integrate by parts to obtain a more manageable formula:
\begin{align*}
f_{1,1}=&2\int(-\varepsilon_{yy}+\varepsilon-\Delta_{b,\omega}(\varepsilon))_{y}(-\varepsilon_{yy}+\varepsilon-\Delta_{b,\omega}(\varepsilon))\psi_B\\
&\quad+2\int(-\varepsilon_{yy}+\varepsilon-\Delta_{b,\omega}(\varepsilon))_{y}(-\psi_B'\varepsilon_{y}+\varepsilon(\varphi_B-\psi_B)).
\end{align*}
We compute these terms separately. First, we have
\begin{align*}
&2\int(-\varepsilon_{yy}+\varepsilon-\Delta_{b,\omega}(\varepsilon))_{y}(-\varepsilon_{yy}+\varepsilon-\Delta_{b,\omega}(\varepsilon))\psi_B\\
&=-\int\psi_B'\big(-\varepsilon_{yy}+\varepsilon-\Delta_{b,\omega}(\varepsilon)\big)^2\\
&=-\int\psi_B'\big([-\varepsilon_{yy}+\varepsilon-\Delta_{b,\omega}(\varepsilon)]^2-(-\varepsilon_{yy}+\varepsilon)^2\big)-\int\psi_B'(-\varepsilon_{yy}+\varepsilon)^2\\
&=-\bigg[\int\psi_B'(\varepsilon_{yy}^2+2\varepsilon_{y}^2)+\varepsilon^2(\psi_B'-\psi_B''')\bigg]\\
&\quad\,-\int\psi_B'\big([-\varepsilon_{yy}+\varepsilon-\Delta_{b,\omega}(\varepsilon)]^2-(-\varepsilon_{yy}+\varepsilon)^2\big).
\end{align*}
Next, we integrate by parts to obtain
\begin{align*}
&-2\int(\Delta_{b,\omega}(\varepsilon))_{y}(\varphi_{i,B}-\psi_B)\varepsilon\\
&=-\frac{1}{3}\int(\varphi_{i,B}-\psi_B)'\big([(Q_{b,\omega}+\varepsilon)^6-Q_{b,\omega}^6-6\varepsilon Q_{b,\omega}^5]-6\varepsilon[(Q_{b,\omega}+\varepsilon)^5-Q_{b,\omega}^5]\big)\\
&\quad-2\int(\varphi_{i,B}-\psi_B)(Q_{b,\omega})_{y}\big[(Q_{b,\omega}+\varepsilon)^5-Q_{b,\omega}^5-5\varepsilon Q_{b,\omega}^4\big]\\
&\quad+\frac{2\omega}{q+1}\int(\varphi_{i,B}-\psi_B)'\Big(\big[|Q_{b,\omega}+\varepsilon|^{q+1}-|Q_{b,\omega}|^{q+1}-(q+1)\varepsilon Q_{b,\omega}|Q_{b,\omega}|^{q-1}\big]\\
&\qquad\qquad\qquad-(q+1)\varepsilon\big[(Q_{b,\omega}+\varepsilon)|Q_{b,\omega}+\varepsilon|^{q-1}-Q_{b,\omega}|Q_{b,\omega}|^{q-1}\big]\Big)\\
&\quad+2\omega\int(\varphi_{i,B}-\psi_B)(Q_{b,\omega})_{y}\Big[(Q_{b,\omega}+\varepsilon)|Q_{b,\omega}+\varepsilon|^{q-1}\\
&\qquad\qquad\qquad-Q_{b,\omega}|Q_{b,\omega}|^{q-1}-q\varepsilon |Q_{b,\omega}|^{q-1}\Big],
\end{align*}
and
\begin{align*}
&2\int(-\varepsilon_{yy}+\varepsilon)_{y}(-\psi_B'\varepsilon_{y}+\varepsilon(\varphi_{i,B}-\psi_B))\\
&=-2\Bigg[\int\psi_B'\varepsilon^2_{yy}+\int\varepsilon^2_{y}\bigg(\frac{3}{2}\varphi'_{i,B}-\frac{1}{2}\psi_B'-\frac{1}{2}\psi_B'''\bigg)\\
&\quad\,+\int\varepsilon^2\bigg(\frac{1}{2}(\varphi_B-\psi_B)'-\frac{1}{2}(\varphi_B-\psi_B)'''\bigg)\Bigg].
\end{align*}

Finally, by direct expansion, we have
\begin{align*}
&\int(\Delta_{b,\omega}(\varepsilon))_{y}\psi_B'\varepsilon_{y}\\
&=5\int\psi_B'\varepsilon_{y}\Big((Q_{b,\omega})_{y}\big[(Q_{b,\omega}+\varepsilon)^4-Q_{b,\omega}^4\big]+\varepsilon_{y}(Q_{b,\omega}+\varepsilon)^4\Big)\\
&\quad-q\omega\int\psi_B'\varepsilon_{y}\Big((Q_{b,\omega})_{y}\big[|Q_{b,\omega}+\varepsilon|^{q-1}-|Q_{b,\omega}|^{q-1}\big]+\varepsilon_{y}|Q_{b,\omega}+\varepsilon|^{q-1}\Big).
\end{align*}
Collecting all the estimates above, we have
$$f_{1,1}=I+II,$$
where
\begin{align*}
I=&-\int\big[3\psi_B'\varepsilon_{yy}^2+(3\varphi'_{i,B}+\psi_B'-\psi_{B}''')\varepsilon^2_{y}+(\varphi'_{i,B}-\varphi_{i,B}''')\varepsilon^2\big]\\
&-2\int\bigg[\frac{(Q_{b,\omega}+\varepsilon)^6-Q_{b,\omega}^6-6\varepsilon Q_{b,\omega}^5}{6}-\varepsilon\big[(Q_{b,\omega}+\varepsilon)^5-Q_{b,\omega}^5\big]\bigg](\varphi'_{i,B}-\psi_B')\\
&+2\int[(Q_{b,\omega}+\varepsilon)^5-Q_{b,\omega}^5-5\varepsilon Q_{b,\omega}^4](Q_{b,\omega})_{y}(\psi_B-\varphi_{i,B})\\
&+10\int\psi_B'\varepsilon_{y}\Big((Q_{b,\omega})_{y}\big[(Q_{b,\omega}+\varepsilon)^4-Q_{b,\omega}^4\big]+\varepsilon_{y}(Q_{b,\omega}+\varepsilon)^4\Big)\\
&-\int\psi_B'\big([-\varepsilon_{yy}+\varepsilon-\Delta_{b,\omega}(\varepsilon)]^2-(-\varepsilon_{yy}+\varepsilon)^2\big)\\
=&I_1+I_2+I_3+I_4+I_5,
\end{align*}
and
\begin{align*}
II=&2\omega\int\bigg[\frac{|Q_{b,\omega}+\varepsilon|^{q+1}-|Q_{b,\omega}|^{q+1}-(q+1)\varepsilon Q_{b,\omega}|Q_{b,\omega}|^{q-1}}{q+1}\\
&\qquad\quad-\varepsilon\big[(Q_{b,\omega}+\varepsilon)|Q_{b,\omega}+\varepsilon|^{q-1}-Q_{b,\omega}|Q_{b,\omega}|^{q-1}\big]\bigg](\varphi'_{i,B}-\psi_B')\\
&-2\omega\int\big[(Q_{b,\omega}+\varepsilon)|Q_{b,\omega}+\varepsilon|^{q-1}-Q_{b,\omega}|Q_{b,\omega}|^{q-1}\\
&\qquad\quad-q\varepsilon |Q_{b,\omega}|^{q-1}\big](Q_{b,\omega})_{y}(\psi_B-\varphi_{i,B})\\
&-2q\omega\int\psi_B'\varepsilon_{y}\Big((Q_{b,\omega})_{y}\big[|Q_{b,\omega}+\varepsilon|^{q-1}-|Q_{b,\omega}|^{q-1}\big]+\varepsilon_{y}|Q_{b,\omega}+\varepsilon|^{q-1}\Big).
\end{align*}

For $I_k$, $k=1,2,3,4$, we can use the same strategy as in \cite[Proposition 3.1]{MMR1} to obtain:
\begin{equation}
\label{C313}
\sum_{k=1}^4I_k\leq-\mu_1\int\Big(\varepsilon_{yy}^2\psi_B'+\varepsilon^2_{y}\varphi_{i,B}'+\varepsilon^2\varphi_{i,B}'\Big)+Cb^4,
\end{equation}
for some universal constant $\mu_1>0$. 

{The idea is to split the integral into three parts. We denote by $I_k^{<,\sim,>}$ the integration on $y<-B/2$, $|y|\leq B/2$, $y>B/2$ respectively, for $k=1,2,3,4$.

On the region $y<-B/2$, using the following weighted Sobolev bound introduced in \cite[Lemma 6]{M1} and \cite[Proposition 3.1]{MMR1}:
\begin{align}
\label{C314}
\Big\|\varepsilon^2\sqrt{\varphi_{i,B}'}\Big\|^2_{L^{\infty}}&\lesssim\|\varepsilon\|^2_{L^2}\bigg(\int\varepsilon^2_{y}\varphi_{i,B}'+\int\varepsilon^2\frac{(\varphi_{i,B}'')^2}{\varphi_B'}\bigg)\nonumber\\
&\lesssim\delta(\kappa)\int(\varepsilon^2_{y}+\varepsilon^2)\varphi_{i,B}',
\end{align}
we have
\begin{align*}
I_2^<+I_3^<+I_4^<&\lesssim_B \int (\varepsilon^6+\varepsilon^5Q_{b,\omega}'+\varepsilon_y^2\varepsilon^4)\varphi_{i,B}'+\|Q_{b,\omega}\|_{L^{\infty}(y<-B/2)}\int(\varepsilon_y^2+\varepsilon^2)\varphi_{i,B}'\\
&\lesssim \delta(\kappa)\int(\varepsilon^2_{y}+\varepsilon^2)\varphi_{i,B}'.
\end{align*}
Hence we have
\begin{equation}
\label{31}
\sum_{k=1}^4I_k^<\leq-\mu_2\int_{y<-B/2}\Big(\varepsilon_{yy}^2\psi_B'+\varepsilon^2_{y}\varphi_{i,B}'+\varepsilon^2\varphi_{i,B}'\Big)
\end{equation}
for some $\mu_2>0$.

For the region $|y|\leq B/2$, we have
$$\sum_{k=1}^4I_k^{\sim}=-\frac{1}{B}\int_{|y|<B/2} \Big(3\varepsilon_y^2+\varepsilon^2-5Q^4\varepsilon^2+20yQ'Q^3\varepsilon^2\Big)+O\Bigg(\int_{|y|<B/2}(|b|+\omega)\varepsilon^2+\varepsilon^6\Bigg).$$

We then introduce the following coercivity lemma:
\begin{lemma}[{\cite[Lemma 3.4]{MMR1}}]
There exist $B_0>100$, $\mu_3>0$, such that for all $\varepsilon\in H^1$ and $B>B_0$, we have 
$$\int_{|y|<B/2} \Big(3\varepsilon_y^2+\varepsilon^2-5Q^4\varepsilon^2+20yQ'Q^3\varepsilon^2\Big)\geq \mu_3\int_{|y|<B/2}\varepsilon_y^2+\varepsilon^2-\frac{1}{B}\int\varepsilon^2e^{-\frac{|y|}{2}}.$$
\end{lemma}

The above lemma implies immediately that
\begin{equation}
\label{32}
\sum_{k=1}^4I_k^{\sim}\leq-\mu_2\int_{|y|<B/2}\Big(\varepsilon_{yy}^2\psi_B'+\varepsilon^2_{y}\varphi_{i,B}'+\varepsilon^2\varphi_{i,B}'\Big)
\end{equation}

While for the region $y>B/2$, we have $\psi_B'=\psi_B'''\equiv0$. We also have:
$$\|\varepsilon\|^2_{L^{\infty}(y>B/2)}\lesssim \|\varepsilon\|^2_{H^1(y>B/2)}\lesssim \mathcal{N}_2\leq \delta(\kappa).$$
Hence, we have:
$$\sum_{k=2}^4I_k^{>}\lesssim\big[\|Q_{b,\omega}\|_{L^{\infty}(y>B/2)}+\|\varepsilon\|_{L^{\infty}(y>B/2)}\big]\int(\varepsilon^2_y+\varepsilon^2)\varphi_{i,B}',$$
which implies that
\begin{equation}
\label{33}
\sum_{k=1}^4I_k^{>}\leq-\mu_2\int_{y>B/2}(+\varepsilon^2_{y}+\varepsilon^2)\varphi_{i,B}'.
\end{equation}

Combining \eqref{31}--\eqref{33}, we obtain \eqref{C313}.
}

Now we turn to the estimate of $I_5$, we have:
\begin{align}
|I_5|&\lesssim\int\psi_B'(|\varepsilon_{yy}|+|\varepsilon|+|\varepsilon|^5+\omega|\varepsilon|^q)(|\varepsilon|^5+\omega|\varepsilon|^q+|Q_{b,\omega}\varepsilon|)\nonumber\\
&\leq\frac{\mu_1}{100}\int(\varepsilon^2_{yy}+\varepsilon^2)\psi_B'+C(\mu_1)\bigg(\int Q_{b,\omega}^2\varepsilon^2\psi_B'+\int\varepsilon^{10}\psi_B'+\omega^2\int|\varepsilon|^{2q}\psi_B'\bigg).\label{C3001}
\end{align}

Combining \eqref{C314} and the hypothesis (H1), we have:
\begin{align}
&\int Q_{b,\omega}^2\varepsilon^2\psi_B'\lesssim\|Q_{b,\omega}\|^2_{L^{\infty}(y<-\frac{B}{2})}\int\varepsilon^2\varphi_{i,B}'\leq \frac{\mu_1}{500}\int(\varepsilon^2_{yy}+\varepsilon^2)\psi_B',\label{C3002}\\
&\int\varepsilon^{10}\psi_B'\lesssim\|\varepsilon^2(\psi_{B}')^{\frac{1}{4}}\|_{L^{\infty}}^4\int\varepsilon^2\lesssim\delta(\kappa)\int(\varepsilon^2_{y}+\varepsilon^2)\varphi_{i,B}'+\bigg(\int\varepsilon_{y}^2(\psi_B')^{\frac{1}{2}}\bigg)^2,\label{C3003}
\end{align}
and
\begin{align}
\omega^2\int|\varepsilon|^{2q}\psi_B'&\lesssim\omega^2\|\varepsilon^2(\psi_{B}')^{\frac{1}{4}}\|_{L^{\infty}}^4\int|\varepsilon|^{2q-8}\lesssim\omega^2\|\varepsilon^2(\psi_{B}')^{\frac{1}{4}}\|_{L^{\infty}}^4\|\varepsilon\|_{L^2}^{q-3}\|\varepsilon_{y}\|_{L^2}^{q-5}\nonumber\\
&\lesssim \delta(\kappa)\int(\varepsilon^2_{y}+\varepsilon^2)\varphi_{i,B}'+\bigg(\int\varepsilon_{y}^2(\psi_B')^{\frac{1}{2}}\bigg)^2,\label{C3004}
\end{align}
where we use the fact that $\omega\|\varepsilon_y\|_{L^2}^m\leq \kappa$ for the last inequality. 

From $\big((\psi')^{\frac{1}{2}}\big)''\lesssim\varphi_{i}'$ and (H1), we have
\begin{align}
\bigg(\int\varepsilon_{y}^2(\psi_{B}')^{\frac{1}{2}}\bigg)^2&=\bigg(-\int\varepsilon\varepsilon_{yy}(\psi_{B}')^{\frac{1}{2}}+\frac{1}{2}\int\varepsilon^2\big((\psi_{B}')^{\frac{1}{2}}\big)''\bigg)^2\nonumber\\
&\lesssim\int\varepsilon^2\int\varepsilon^2_{yy}\psi_B'+\bigg(\int\varepsilon^2\varphi_{i,B}'\bigg)^2\nonumber\\
&\lesssim\delta(\kappa)\int\big(\varepsilon_{yy}^2\psi_B'+\varepsilon^2\varphi_{i,B}'\big).\label{C316}
\end{align}
Injecting \eqref{C3002}--\eqref{C316} into \eqref{C3001}, we have:
\begin{equation}
\label{C315}
|I_5|\lesssim \frac{\mu_1}{50}\bigg(\int\varepsilon^2_{yy}\psi_B'+\int(\varepsilon^2+\varepsilon_{y}^2)\varphi_{i,B}'\bigg).
\end{equation}

Now, we turn to the estimate of $II$. We write $II$ in the following form:
$$II=II^{<}+II^>,$$
where $II^{<,>}$ correspond to the integration on $y<-\frac{B}{2}$ and $y>-\frac{B}{2}$ respectively.

For $II^<$, using the fact that $\psi_B'\sim(\varphi_{i,B}')^2$ for $y<-\frac{B}{2}$, we have:
\begin{align*}
|II^<|&\lesssim \omega\bigg(\int_{y<-\frac{B}{2}}(|\varepsilon|^{q+1}+|Q_{b,\omega}|^{q-1}\varepsilon^2)\varphi_{i,B}'+\int_{y<-\frac{B}{2}}|Q_{b,\omega}'|(|\varepsilon|^q+\varepsilon^2)\varphi_{i,B}\bigg)\\
&\quad+\omega\int_{y<-\frac{B}{2}}\psi_B'|\varepsilon_{y}|\Big(|\varepsilon|^{q-1}+|Q_{b,\omega}|^{q-2}|\varepsilon|+|\varepsilon_{y}||\varepsilon|^{q-1}+|\varepsilon_{y}||Q_{b,\omega}|^{q-1}\Big)\\
&\leq C(\mu_1)\omega\bigg(\int\varphi_{i,B}'\big(|\varepsilon|^{q+1}+|\varepsilon|^q\big)+\int\psi_B'\big(\varepsilon_{y}^2|\varepsilon|^{q-1}+|\varepsilon_{y}||\varepsilon|^{q-1}\big)\bigg)\\
&\quad+\frac{\mu_1}{500}\int_{y<-\frac{B}{2}}(\varepsilon^2+\varepsilon_{y}^2)\varphi_{i,B}'.
\end{align*}
We use (H1)--(H3) and Gagliardo-Nirenberg's inequality to estimate these terms separately. First, we have
\begin{align*}
\omega\int|\varepsilon|^{q+1}\varphi_{i,B}'&\lesssim\omega \|\varepsilon^2(\varphi_{i,B}')^{1/2}\|_{L^{\infty}}^2\int|\varepsilon|^{q-3}\\
&\lesssim\omega\bigg(\int\varepsilon^2\int(\varepsilon^2+\varepsilon_{y}^2)\varphi_{i,B}'\bigg)\Big(\|\varepsilon_{y}\|_{L^2}^{\frac{q-5}{2}}\|\varepsilon\|_{L^2}^{\frac{q-1}{2}}\Big)\\
&\lesssim(\omega\|\varepsilon_{y}\|_{L^2}^m)\bigg(\int\varepsilon^2\bigg)^{\frac{q+3}{4}}\int(\varepsilon^2+\varepsilon_{y}^2)\varphi_{i,B}'\\
&\lesssim\delta(\kappa)\int(\varepsilon^2+\varepsilon_{y}^2)\varphi_{i,B}',
\end{align*}
and 
\begin{align*}
\omega\int|\varepsilon|^{q}\varphi_{i,B}'&\lesssim\omega\big\|\varepsilon^2(\varphi_{i,B}')^{1/2}\big\|_{L^{\infty}}^{\frac{3}{2}}\big\||\varepsilon|^{\frac{1}{2}}(\varphi_{i,B}')^{\frac{1}{4}}\big\|_{L^4}\big\||\varepsilon|^{\frac{2q-7}{2}}\big\|_{L^{4/3}}\\
&\lesssim \omega\bigg(\int\varepsilon^2\bigg)^{\frac{3}{4}}\bigg(\int(\varepsilon^2+\varepsilon_{y}^2)\varphi_{i,B}'\bigg)\|\varepsilon\|_{L^2}^{\frac{q-2}{2}}\|\varepsilon_{y}\|_{L^2}^{\frac{q-5}{2}}\\
&\lesssim\delta(\kappa)\int(\varepsilon^2+\varepsilon_{y}^2)\varphi_{i,B}'.
\end{align*}
From $\psi'\lesssim (\varphi_i')^2$ and \eqref{C316}, we also have:
\begin{align*}
\omega\int\psi_B'\varepsilon^2_{y}|\varepsilon|^{q-1}&\lesssim \omega\big\|\varepsilon^2(\psi_{B}')^{1/4}\big\|_{L^{\infty}}^2\|\varepsilon\|_{L^{\infty}}^{q-5}\int\varepsilon_{y}^2(\psi_{B}')^{1/2}\\
&\lesssim (\omega\|\varepsilon_{y}\|_{L^2}^m)\|\varepsilon\|_{L^2}^{m+2}\bigg(\int(\varepsilon^2+\varepsilon_{y}^2)(\psi_{B}')^{1/2}\bigg)\int\varepsilon_{y}^2(\psi_{B}')^{1/2}\\
&\lesssim \delta(\kappa)\int(\varepsilon^2+\varepsilon_{y}^2)\varphi_{i,B}'+\bigg(\int\varepsilon_{y}^2(\psi_{B}')^{1/2}\bigg)^2\\
&\lesssim \delta(\kappa)\bigg(\int\varepsilon^2_{yy}\psi_B'+\int(\varepsilon^2+\varepsilon_{y}^2)\varphi_{i,B}'\bigg),
\end{align*}
and 
\begin{align*}
\omega\int\psi_B'|\varepsilon_{y}||\varepsilon|^{q-1}&\lesssim\omega\big\|\varepsilon^2(\psi_{B}')^{1/4}\big\|_{L^{\infty}}^{\frac{3}{2}}\big\|\varepsilon_{y}(\psi_B')^{\frac{1}{4}}\big\|_{L^2}\big\||\varepsilon|^{q-4}\big\|_{L^2}\\
&\lesssim (\omega\|\varepsilon_{y}\|_{L^2}^m)\|\varepsilon\|_{L^2}^{m+2}\bigg(\int(\varepsilon^2+\varepsilon_{y}^2)(\psi_{B}')^{1/2}\bigg)\bigg(\int\varepsilon_{y}^2(\psi_{B}')^{1/2}\bigg)^{\frac{1}{2}}\\
&\lesssim\delta(\kappa)\int(\varepsilon^2+\varepsilon_{y}^2)\varphi_{i,B}'+\bigg(\int\varepsilon_{y}^2(\psi_{B}')^{1/2}\bigg)^{\frac{3}{2}}\\
&\leq\frac{\mu_1}{1000}\int(\varepsilon^2+\varepsilon_{y}^2)\varphi_{i,B}'+C(\mu_1)\bigg(\int\varepsilon_{y}^2(\psi_{B}')^{1/2}\bigg)^{2}\\
&\leq\frac{\mu_1}{500}\bigg(\int\varepsilon^2_{yy}\psi_B'+\int(\varepsilon^2+\varepsilon_{y}^2)\varphi_{i,B}'\bigg).
\end{align*}
In conclusion, we have
\begin{equation}
\label{C317}
|II^<|\leq\frac{\mu_1}{50}\bigg(\int\varepsilon^2_{yy}\psi_B'+\int(\varepsilon^2+\varepsilon_{y}^2)\varphi_{i,B}'\bigg).
\end{equation}

For $II^{>}$, we know that $\psi_{B}'\equiv0$ for $y>-\frac{B}{2}$. Using Sobolev Embedding, we have
$$\|\varepsilon\|^2_{L^{\infty}(y>-\frac{B}{2})}\lesssim\|\varepsilon\|_{L^{2}(y>-\frac{B}{2})}\|\varepsilon_{y}\|_{L^{2}(y>-\frac{B}{2})}\leq\mathcal{N}_2\leq 1.$$
Thus, we have
\begin{equation}\label{C318}
|II^{>}|\lesssim\omega\bigg(\int\varepsilon^2\varphi_{i,B}'+\int(Q_{b,\omega})_{y}\varepsilon^2\varphi_{i,B}\bigg)\lesssim_B \delta(\kappa)\int(\varepsilon^2+\varepsilon_{y}^2)\varphi_{i,B}'.
\end{equation}
Combining \eqref{C313}, \eqref{C315}, \eqref{C317} and \eqref{C318}, we have
\begin{equation}\label{C319}
f_{1,1}\leq -\mu_0\int(\varepsilon^2+\varepsilon_{y}^2)\varphi_{i,B}'+Cb^4,
\end{equation}
for some universal constant $\mu_0>0$.

Now, let us deal with $f_{1,2}$, it is easy to see that
$$f_{1,2}=\widetilde{I}+\widetilde{II},$$
where 
\begin{align*}
&\widetilde{I}=2\bigg(\frac{\lambda_s}{\lambda}+b\bigg)\int\Lambda Q_{b,\omega}\big(-(\psi_B\varepsilon_{y})_{y}+\varepsilon\varphi_{i,B}-\psi_B[(Q_{b,\omega}+\varepsilon)^5-Q_{b,\omega}^5]\big),\\
&\widetilde{II}=2\omega\bigg(\frac{\lambda_s}{\lambda}+b\bigg)\int\Lambda Q_{b,\omega}\psi_B\big((Q_{b,\omega}+\varepsilon)|Q_{b,\omega}+\varepsilon|^{q-1}-Q_{b,\omega}|Q_{b,\omega}^{q-1}|\big).
\end{align*}

The term $\widetilde{I}$ can be estimated by the same argument as in \cite[Propostion 3.1]{MMR1}. Thus, we have
$$|\widetilde{I}|\leq \frac{\mu_0}{500}\int(\varepsilon^2+\varepsilon_{y}^2)\varphi_{i,B}'+Cb^2(\omega^2+b^2).$$
{We mention here that the modulation estimate \eqref{CMS1} in this paper is slightly different from \cite[(2.29)]{MMR1}, $i.e.$ there is an additional term ``$\omega|b|$" on the right hand side of \eqref{CMS1}. This additional term results in the appearance of the term ``$\omega^2b^2$" on the right hand side of the above inequality.}

While for $\widetilde{II}$, we have
$$|\widetilde{II}|\lesssim\omega\bigg|\frac{\lambda_s}{\lambda}+b\bigg|\big(B^{\frac{1}{2}}\mathcal{N}_{i,\rm loc}^{\frac{1}{2}}+\int|\varepsilon|^q\psi_B\big).$$
Using \eqref{CMS1} and the strategy for $f_{1,1}$, we have
$$|\widetilde{II}|\leq \frac{\mu_0}{500}\int(\varepsilon^2+\varepsilon_{y}^2)\varphi_{i,B}'+Cb^2(\omega^2+b^2).$$
Similar argument can be applied to $f_{1,k}$, $k=3,4,5$. Together with \eqref{C319}, we conclude the proof of \eqref{C311}.

{\bf Control of $f_2$.} For $f_2$, we integrate by parts, using \eqref{C320} to get
\begin{align*}
f_2=2\mathcal{J}_{i,j}\int&\varepsilon\varphi_{i,B}\bigg[\big(-\varepsilon_{yy}+\varepsilon-\Delta_{b,\omega}(\varepsilon)\big)_{y}+\bigg(\frac{\lambda_s}{\lambda}+b\bigg)\Lambda Q_{b,\omega}\\
&+\bigg(\frac{x_s}{\lambda}-1\bigg)(Q_{b,\omega}+\varepsilon)_{y}-b_s\frac{\partial Q_{b,\omega}}{\partial b}-\omega_s\frac{\partial Q_{b,\omega}}{\partial \omega}+\Psi_{b,\omega}\bigg].
\end{align*}
We integrate by parts, estimating all terms like we did for $f_{1}$. Together with
$$|\mathcal{J}_{i,j}|\lesssim|J_1|\lesssim \mathcal{N}_2^{\frac{1}{2}}\lesssim \delta(\kappa),$$
we have
\begin{equation}
\label{C321}
|f_2|\lesssim \delta(\kappa)\int(\varepsilon^2+\varepsilon_{y}^2)\varphi_{i,B}'+b^2(\omega^2+b^2).
\end{equation}

{\bf Control of $f_{3}$}. Recall that 
\begin{align*}
f_3=&2\frac{\lambda_s}{\lambda}\int\Lambda\varepsilon\big(-(\psi_B\varepsilon_{y})_{y}+(1+\mathcal{J}_{i,j})\varepsilon\varphi_{i,B}-\psi_B\Delta_{b,\omega}(\varepsilon)\big)\\
&+(\mathcal{J}_{i,j})_s\int\varphi_{i,B}\varepsilon^2-2(j-1)\frac{\lambda_s}{\lambda}\mathcal{F}_{i,j}.
\end{align*}
Integrating by parts%
\footnote{See \cite[Proposition 3.1, Step 5]{MMR1} and \cite[(5.22)]{L1} for more details.}%
, we have
$$f_3=\hat{I}+\hat{II},$$
where
\begin{align*}
\hat{I}=&\frac{\lambda_s}{\lambda}\int[(2-2(j-1))\psi_B-y\psi_B']\varepsilon^2_{y}\\
&-\frac{1}{3}\frac{\lambda_s}{\lambda}\int[(2-2(j-1))\psi_B-y\psi_B']\big[(Q_{b,\omega}+\varepsilon)^6-Q_{b,\omega}^6-6\varepsilon Q_{b,\omega}^5\big]\\
&+2\frac{\lambda_s}{\lambda}\int\psi_B\Lambda Q_b\big[(Q_{b,\omega}+\varepsilon)^5-Q_{b,\omega}^5-5\varepsilon Q_{b,\omega}^4\big]\\
&+(\mathcal{J}_{i,j})_s\int\varepsilon^2\varphi_{i,B}-\frac{\lambda_s}{\lambda}(1+\mathcal{J}_{i,j})\int y\varphi_{i,B}'\varepsilon^2-2(j-1)\frac{\lambda_s}{\lambda}(1+\mathcal{J}_{i,j})\int\varepsilon^2 \varphi_{i,B},
\end{align*}
and
\begin{align*}
\hat{II}=&\frac{2\omega}{q+1}\frac{\lambda_s}{\lambda}\int\bigg[\bigg(\frac{q+3}{q-1}-2(j-1)\bigg)\psi_B-y\psi_B'\bigg]\times\Big[|Q_{b,\omega}+\varepsilon|^{q+1}\\
&\qquad\quad-|Q_{b,\omega}|^{q+1}-(q+1)\varepsilon Q_{b,\omega}|Q_{b,\omega}|^{q-1}\Big]\\
&-2\omega\frac{\lambda_s}{\lambda}\int\psi_B\Lambda Q_b\big[(Q_{b,\omega}+\varepsilon)|Q_{b,\omega}+\varepsilon|^{q-1}-Q_{b,\omega}|Q_{b,\omega}|^{q-1}-q\varepsilon |Q_{b,\omega}|^{q-1}\big].
\end{align*}

Similarly, we can use the same strategy as in \cite[Proposition 3.1]{MMR1} to estimate $I$, which leads to:
\begin{equation}
\label{310}
|\hat{I}|\lesssim \delta(\kappa)\int(\varepsilon^2+\varepsilon_{y}^2)\varphi_{i,B}'+b^2(\omega^2+b^2).
\end{equation}

{More precisely, we can rewrite $\hat{I}$ as following
\begin{align*}
\hat{I}=&\frac{\lambda_s}{\lambda}\int[2(2-j)\psi_B-y\psi_B']\varepsilon_y^2\\
&-\frac{1}{3}\frac{\lambda_s}{\lambda}\int[2(2-j)\psi_B-y\psi_B']\times\big[(Q_{b,\omega}+\varepsilon)^6-Q_{b,\omega}^6-6\varepsilon Q_{b,\omega}^5\big]\\
&+2\frac{\lambda_s}{\lambda}\int\psi_B\Lambda Q_b\big[(Q_{b,\omega}+\varepsilon)^5-Q_{b,\omega}^5-5\varepsilon Q_{b,\omega}^4\big]\\
&+\frac{1}{i}\bigg[(\mathcal{J}_{i,j})_s-2(j-1)(1+\mathcal{J}_{i,j})\frac{\lambda_s}{\lambda}\bigg]\int(i\varphi_{i,B}-y\varphi_{i,B}')\varepsilon^2\\
&+\frac{1}{i}\bigg[(\mathcal{J}_{i,j})_s-(2(j-1)+i)(1+\mathcal{J}_{i,j})\frac{\lambda_s}{\lambda}\bigg]\int y\varphi_{i,B}'\varepsilon^2\\
=&\hat{I}_1+\hat{I}_2,
\end{align*}
where
$$\hat{I}_2=\frac{1}{i}\bigg[(\mathcal{J}_{i,j})_s-(2(j-1)+i)(1+\mathcal{J}_{i,j})\frac{\lambda_s}{\lambda}\bigg]\int y\varphi_{i,B}'\varepsilon^2.$$

We also denote by $\hat{I}_k^{<,\sim,>}$, $k=1,2$, the integration over $y<-B/2$, $|y|<B/2$, $y>B/2$ respectively.

For integration over $|y|<B/2$, the estimate is straightforward, we have
\begin{equation*}
|\hat{I}_1^{\sim}|+|\hat{I}_2^{\sim}|\lesssim\delta(\kappa)\int_{|y|<B/2}(\varepsilon_y^2+\varepsilon^2)\varphi_{i,B}'.
\end{equation*}
 While for $y<-B/2$, using \eqref{CMS1}, we have
\begin{align*}
&|\hat{I}_1^<|+|\hat{I}_2^<|\\
&\lesssim (|b|+\mathcal{N}_{i,\rm loc}^{\frac{1}{2}})\int _{y<-B/2}(\psi_B+|y|\phi'_{i,B}+\varphi_{i,B})(\varepsilon_y^2+\varepsilon^2)+|y|\varphi_{i,B}'\varepsilon^2\\
&\lesssim (|b|+\mathcal{N}_{i,\rm loc}^{\frac{1}{2}})\Bigg[\int_{y<-B/2}\varepsilon_y^2\varphi_{i,B}'+\int_{y<-B/2}|y|\varphi_{i,B}'\varepsilon^2\Bigg]\\
&\lesssim (|b|+\mathcal{N}_{i,\rm loc}^{\frac{1}{2}})\Bigg[\int\varepsilon_y^2\varphi_{i,B}'+\bigg(\int_{y<-B/2}y^{100}e^{\frac{y}{B}}\varepsilon^2\bigg)^{\frac{1}{100}}\bigg(\int_{y<-B/2}e^{\frac{y}{B}}\varepsilon^2\bigg)^{\frac{99}{100}}\Bigg]\\
&\lesssim (|b|+\mathcal{N}_{i,\rm loc}^{\frac{1}{2}})\times\bigg(\int\varepsilon^2_y\varphi'_{i,B}+\mathcal{N}_{i,\rm loc}^{\frac{99}{100}}\bigg)\leq \delta(\kappa)\int(\varepsilon_y^2+\varepsilon^2)\varphi_{i,B}'+Cb^4.
\end{align*}

Now, for $y>B/2$, we first have 
$$i\varphi_{i,B}-y\varphi_{i,B}'=0,$$
for all $y>B$. Hence
$$|\hat{I}_1^>|\lesssim  (|b|+\mathcal{N}_{i,\rm loc}^{\frac{1}{2}})\int(\varepsilon^2_y+\varepsilon^2)\varphi_{i,B}'\lesssim \delta(\kappa)\int(\varepsilon_y^2+\varepsilon^2)\varphi_{i,B}'.$$

Next, for $\hat{I}_2^>$, we know from \eqref{CLOL} that
\begin{align*}
\bigg|(\mathcal{J}_{i,j})_s-(2(j-1)+i)(1+\mathcal{J}_{i,j})\frac{\lambda_s}{\lambda}\bigg|&=\frac{4(j-1)+2i}{(1-J_1)^{4(j-1)+2i+1}}\bigg|(J_1)_s-\frac{\lambda_s}{2\lambda}(1-J_1)\bigg|\\
&\lesssim |b|+\mathcal{N}_{i,\rm loc}.
\end{align*}
Together with \eqref{C32}, \eqref{C33} and \eqref{C39}, we have:
\begin{align*}
|\hat{I}_2^>|&\lesssim (|b|+\mathcal{N}_{i, \rm loc}^{\frac{1}{2}})\bigg(1+\frac{1}{\lambda^{\frac{10}{9}}}\bigg)\mathcal{N}_{i,\rm loc}^{\frac{8}{9}}\\
&\lesssim|b|\big(1+\delta(\kappa)|b|^{-\frac{5}{9}}\big)\mathcal{N}_{i,\rm loc}^{\frac{8}{9}}+\mathcal{N}_{i,\rm loc}\big(1+\delta(\kappa)\mathcal{N}_{i,\rm loc}^{-\frac{5}{9}}\big)\mathcal{N}_{i,\rm loc}^{\frac{8}{9}}\\
&\lesssim\delta(\kappa)\int(\varepsilon_y^2+\varepsilon^2)\varphi_{i,B}'+b^2(\omega^2+b^2).
\end{align*}

Combining the above estimates, we obtain \eqref{310}.
}

Finally, for $\hat{II}$, from the following fact
$$\psi_B+\bigg|\bigg(\frac{q+3}{q-1}-2(j-1)\bigg)\psi_B-y\psi_B'\bigg|\lesssim_B \varphi_{i,B}',$$
we have
\begin{align*}
|\hat{II}|\lesssim\omega\bigg|\frac{\lambda_s}{\lambda}\bigg|\int(|\varepsilon|^{q+1}+|\varepsilon|^q+\varepsilon^2)\varphi_{i,B}'.
\end{align*}
Using $|\lambda_s/\lambda|\lesssim\delta(\kappa)$ and the strategy for $f_{1,1}$, we have:
$$|\hat{II}|\lesssim\delta(\kappa)\int(\varepsilon^2+\varepsilon_{y}^2)\varphi_{i,B}'.$$

In conclusion, we have
\begin{equation}
\label{C322}
|f_3|\lesssim \delta(\kappa)\int(\varepsilon^2+\varepsilon_{y}^2)\varphi_{i,B}'+b^2(\omega^2+b^2).
\end{equation}

{\bf Control of $f_4$.} From \eqref{C22} and \eqref{CMF2}, we have
$$|(Q_{b,\omega})_s|\lesssim |b_s|\bigg|\frac{\partial Q_{b,\omega}}{\partial b}\bigg|+|\omega_s|\bigg|\frac{\partial Q_{b,\omega}}{\partial \omega}\bigg|\leq (\omega+|b|)(|b|+\mathcal{N}_{i,\rm loc}^{\frac{1}{2}})\lesssim\delta(\kappa).$$
Using the Sobolev bounds \eqref{C314} and the strategy for $f_{1,1}$, we have
\begin{align}\label{C323}
|f_4|&\lesssim\delta(\kappa)\bigg(\int(\omega|\varepsilon|^q+|\varepsilon|^5+\varepsilon^2)\varphi_{i,B}'\bigg)\nonumber\\
&\lesssim\delta(\kappa)\bigg(\int\varepsilon^2_{yy}\psi_B'+\int(\varepsilon^2+\varepsilon_{y}^2)\varphi_{i,B}'\bigg).
\end{align}

{\bf Control of $f_5$.} From \eqref{CMS1} we know that 
$$|\omega_s|=m\omega\bigg|\frac{\lambda_s}{\lambda}\bigg|\lesssim \delta(\kappa).$$
Thus, the Sobolev bounds \eqref{C314} and the strategy for $f_{1,1}$, we have 
\begin{align}\label{C324}
|f_5|&\lesssim\delta(\kappa)\int(\omega|\varepsilon|^{q+1}+\varepsilon^2)\varphi_{i,B}'\nonumber\\
&\lesssim\delta(\kappa)\bigg(\int\varepsilon^2_{yy}\psi_B'+\int(\varepsilon^2+\varepsilon_{y}^2)\varphi_{i,B}'\bigg).
\end{align}

Combining \eqref{C321}, \eqref{C322}, \eqref{C323} and \eqref{C324}, we conclude the proof of \eqref{C312}, hence the proof of Proposition \ref{CP4}.
\end{proof}

\subsection{Dynamical control of the tail on the right} In order to close the bootstrap bound (H3), we need the dynamical control of the $L^2$ tail on the right introduced in \cite{MMR1}. More precisely, we choose a smooth function
$$\varphi_{10}(y)=
\begin{cases}
0,\text{ for }y<0,\\
y^{10},\text{ for }y>1,
\end{cases}
\quad\varphi_{10}'\geq0.$$
Then we have
\begin{lemma}[Dynamical control of the tail on the right, \cite{MMR1}]\label{CL4}Under the assumption of Proposition \ref{CP4}, there holds:
\begin{equation}\label{C37}
\frac{1}{\lambda^{10}}\frac{d}{ds}\bigg(\lambda^{10}\int\varphi_{10}\varepsilon^2\bigg)\lesssim_B\mathcal{N}_{1,\rm loc}+b^2.
\end{equation}
\end{lemma}
\begin{proof}[Proof]
The proof of Lemma \ref{CL4} is exactly the same as \cite[Lemma 3.7]{MMR1}. 

{More precisely, from \eqref{C320}, we have:
\begin{multline*}
\frac{1}{2}\frac{d}{ds}\int\varphi_{10}\varepsilon^2=\int\varphi_{10}\varepsilon\bigg[
\frac{\lambda_s}{\lambda}\Lambda\varepsilon+\big(-\varepsilon_{yy}+\varepsilon-\Delta_{b,\omega}(\varepsilon)\big)_{y}+\bigg(\frac{\lambda_s}{\lambda}+b\bigg)\Lambda Q_{b,\omega}\\
+\bigg(\frac{x_s}{\lambda}-1\bigg)(Q_{b,\omega}+\varepsilon)_{y}-b_s\frac{\partial Q_{b,\omega}}{\partial b}-\omega_s\frac{\partial Q_{b,\omega}}{\partial \omega}+\Psi_{b,\omega}\bigg],
\end{multline*}
where 
$$\Delta_{b,\omega}(\varepsilon)=(Q_{b,\omega}+\varepsilon)^5-Q_{b,\omega}^5-\omega(Q_{b,\omega}+\varepsilon)|Q_{b,\omega}+\varepsilon|^{q-1}+\omega Q_{b,\omega}|Q_{b,\omega}|^{q-1}.$$

We integrate the linear term by parts using the fact that $y\varphi_{10}'=10\varphi_{10}$ for $y\geq1$, and $\varphi_{10}'''\ll\varphi'_{10}$ for $y$ large, to obtain
\begin{align*}
&\int\varphi_{10}\varepsilon\bigg[
\frac{\lambda_s}{\lambda}\Lambda\varepsilon+\big(-\varepsilon_{yy}+\varepsilon)_y\bigg]\\
&=-\frac{1}{2}\frac{\lambda_s}{\lambda}\int y\varphi'_{10}\varepsilon^2-\frac{3}{2}\int\varepsilon^2_y\varphi'_{10}-\frac{1}{2}\int\varepsilon^2\varphi'_{10}+\frac{1}{2}\int\varepsilon^2\varphi'''_{10}\\
&\leq -5\frac{\lambda_s}{\lambda}\int \varphi_{10}\varepsilon^2-\frac{3}{2}\int\varepsilon^2_y\varphi'_{10}-\frac{1}{2}\int\varepsilon^2\varphi'_{10}+C\mathcal{N}_{1,\rm loc}.
\end{align*}

Next, from \eqref{C23}, \eqref{CMS1}, and \eqref{CMS2}, it is easy to obtain:
\begin{align*}
&\Bigg|\int\varphi_{10}\varepsilon\bigg[\bigg(\frac{\lambda_s}{\lambda}+b\bigg)\Lambda Q_{b,\omega}
+\bigg(\frac{x_s}{\lambda}-1\bigg)(Q_{b,\omega}+\varepsilon)_{y}-b_s\frac{\partial Q_{b,\omega}}{\partial b}-\omega_s\frac{\partial Q_{b,\omega}}{\partial \omega}+\Psi_{b,\omega}\bigg]\Bigg|\\
&\lesssim b^2+\mathcal{N}_{1,\rm loc}.
\end{align*}

While for the nonlinear term, we integrate by parts to remove all derivatives on $\varepsilon$ to obtain:
\begin{align*}
\Bigg|\int\varphi_{10}\varepsilon\big[\Delta_{b,\omega}(\varepsilon)\big]_y\Bigg|&\lesssim \int \varphi_{10}\varepsilon^2 e^{-\frac{|y|}{2}}\big(\|\varepsilon\|_{L^{\infty}(y>0)}\big)+\int\varepsilon^6\varphi_{10}'+\omega\int|\varepsilon|^{q+1}\varphi_{10}'\\
&\lesssim \delta(\kappa)\int  \varepsilon^2\varphi_{10}',
\end{align*}
where we use the fact that $|Q_{b,\omega}|+|Q_{b,\omega}'|\lesssim e^{-|y|/2}$ for $y>0$ and
$$\|\varepsilon\|_{L^{\infty}(y>0)}\lesssim \mathcal{N}_{1}\ll1.$$

Hence, we have
$$\frac{d}{ds}\int\varphi_{10}\varepsilon^2+10\frac{\lambda_s}{\lambda}\int\varphi_{10}\varepsilon^2\lesssim b^2+\mathcal{N}_{1,\rm loc},$$
which together with Gronwell's inequality, implies \eqref{C37} immediately.
}
\end{proof}

\section{Rigidity of the dynamics in $\mathcal{A}_{\alpha_0}$ and proof of Theorem \ref{CMT}}
In this section we will classify the behavior of any solution with initial data in $\mathcal{A}_{\alpha_0}$, which directly implies Theorem \ref{CMT}. To begin, we define
\begin{equation}\label{C41}
t^*=\sup\{0<t<+\infty|\text{for all $t'\in[0,t]$, $u(t')\in\mathcal{T}_{\alpha^*,\gamma}$}\}.
\end{equation} 
Assume $0<\gamma\ll\alpha_0\ll\alpha^*\ll1$, then the condition on the initial data, i.e. $u_0\in\mathcal{A}_{\alpha_0}$ implies $t^*>0$.

Next, by Lemma \ref{CL3}, $u(t)$ admits the following geometrical decomposition on $[0,t^*]$:
$$u(t,x)=\frac{1}{\lambda^{\frac{1}{2}}(t)}\big[Q_{b(t),\omega(t)}+\varepsilon(t)\big]\bigg(\frac{x-x(t)}{\lambda(t)}\bigg).$$

The condition $u_0\in\mathcal{A}_{\alpha_0}$ implies:
\begin{gather}
\omega(0)+\|\varepsilon(0)\|_{H^1}+\omega(0)\|\varepsilon_{y}(0)\|^m_{L^2}+|b(0)|+|1-\lambda(0)|\lesssim\delta(\alpha_0),\label{C42}\\
\int_{y>0}y^{10}\varepsilon^2(0)\,dy\leq2.\label{C4013}
\end{gather}
Using H\"{o}lder's inequality, we have:
\begin{equation}\label{C43}
\mathcal{N}_2(0)\lesssim\delta(\alpha_0).
\end{equation}
Then let us fix a $0<\kappa\ll1$ as in Proposition \ref{CP3} and \ref{CP4}, and define
\begin{equation}
t^{**}=\sup\{0<t<t^*|\text{(H1), (H2) and (H3) hold for all }t'\in[0,t]\}.
\end{equation}
Note that from \eqref{C42}--\eqref{C43}, we have $t^{**}>0$. Let $s^*=s(t^*)$, $s^{**}=s(t^{**})$. 

\subsection{Consequence of the monotonicity formula} We derive some crucial estimates from the monotonicity formula introduced in Section 3.
\begin{lemma}\label{CL5}
We have the following:
\begin{enumerate}
	\item Almost monotonicity of the localized Sobolev norm: there exists a universal constant $K_0>1$, such that for $i=1,2$ and $0\leq s_1<s_2\leq s^{**}$, there holds:
	\begin{align}
	&\mathcal{N}_{i}(s_2)+\int_{s_1}^{s_2}\int\big(\varepsilon_{y}^2(s,y)+\varepsilon^2(s,y)\big)\varphi'_{i,B}(y)\,dyds\nonumber\\
	&\leq K_0\bigg[\mathcal{N}_i(s_1)+\sup_{s\in[s_1,s_2]}|b(s)|^3+\sup_{s\in[s_1,s_2]}\omega^3(s)\bigg],\label{C44}\\
	&\frac{\mathcal{N}_i(s_2)}{\lambda^2(s_2)}+\int_{s_1}^{s_2}\frac{1}{\lambda^2(s)}\bigg[\bigg(\int\big(\varepsilon_{y}^2+\varepsilon^2\big)(s)\varphi'_{i,B}\bigg)+b^2(s)\big(|b(s)|+\omega(s)\big)\bigg]\,ds\nonumber\\
	&\leq K_0\Bigg(\frac{\mathcal{N}_i(s_1)}{\lambda^2(s_1)}+\bigg[\frac{b^2(s_1)+\omega^2(s_1)}{\lambda^2(s_1)}+\frac{b^2(s_2)+\omega^2(s_2)}{\lambda^2(s_2)}\bigg]\Bigg).\label{C45}
	\end{align}
	\item Control of $b$ and $\omega$: for all $0\leq s_1<s_2\leq s^{**}$, there holds:
	\begin{equation}
	\label{C46}
	\omega(s_2)+\int_{s_1}^{s_2}b^2(s)\,ds\lesssim\mathcal{N}_1(s_1)+\omega(s_1)+\sup_{s\in[s_1,s_2]}|b(s)|,
	\end{equation}
	\item Control of $\frac{b}{\lambda^2}$: let $c_1=\frac{m}{m+2}G'(0)>0$, where $G$ is the $C^2$ function introduced in \eqref{CLOB}. Then	there exists a universal constant $K_1>1$ such that for all $0\leq s_1<s_2\leq s^{**}$, there holds:
	\begin{align}
	&\bigg|\frac{b(s_2)+c_1\omega(s_2)}{\lambda^2(s_2)}-\frac{b(s_1)+c_1\omega(s_1)}{\lambda^2(s_1)}\bigg|\nonumber\\
	&\leq K_1\bigg(\frac{\mathcal{N}_1(s_1)}{\lambda^2(s_1)}+\frac{b^2(s_1)+\omega^2(s_1)}{\lambda^2(s_1)}+\frac{b^2(s_2)+\omega^2(s_2)}{\lambda^2(s_2)}\bigg).\label{C47}
	\end{align}
	\item Refined control of $\lambda$: let $\lambda_0(s)=\lambda(s)(1-J_1(s))^2$. Then there exists a universal constant $K_2>1$ such that for all $s\in[0,s^{**}]$,
	\begin{equation}\label{C48}
	\bigg|\frac{(\lambda_0)_s}{\lambda_0}+b\bigg|\leq K_2\Big[\mathcal{N}_1+(|b|+\omega)(\mathcal{N}_2^{\frac{1}{2}}+|b|)\Big].
	\end{equation}
\end{enumerate}
\end{lemma}
\begin{proof}[Proof]
{\it Proof of \eqref{C44} and \eqref{C46}}. From \eqref{C212}, we have:
$$\frac{d}{ds}G(\omega)+b^2\leq-b_s+C\mathcal{N}_{1,\rm loc}.$$
Integrating from $s_1$ to $s_2$, we have
\begin{align*}
G(\omega(s_2))+\int_{s_1}^{s_2}b^2&\lesssim \int_{s_1}^{s_2} \mathcal{N}_{1,\rm loc}+G(\omega(s_1))+|b(s_2)-b(s_1)|\\
&\lesssim \int_{s_1}^{s_2} \mathcal{N}_{1,\rm loc}+G(\omega(s_1))+\sup_{s\in[s_1,s_2]}|b(s)|
\end{align*}
Since $G(\omega)\sim \omega$, we then obtain \eqref{C46}.

Next, from the monotonicity formula \eqref{CMF1} and \eqref{CCOER} we obtain:
\begin{align}
&\mathcal{N}_{i}(s_2)+\int_{s_1}^{s_2}\int\big(\varepsilon_{y}^2(s,y)+\varepsilon^2(s,y)\big)\varphi'_{i,B}(y)\,dyds\nonumber\\
&\lesssim\mathcal{N}_i(s_1)+\Big[\sup_{s\in[s_1,s_2]}b^2(s)+\sup_{s\in[s_1,s_2]}\omega^2(s)\Big]\int_{s_1}^{s_2}b^2,\label{C410}
\end{align}
Combining \eqref{C46} and \eqref{C410}, we obtain \eqref{C44}. 

{\it Proof of \eqref{C45}}. First, from \eqref{C212} and \eqref{CMS2}, we have
\begin{align}
2\int_{s_1}^{s_2}\frac{|b|^3}{\lambda^2}&\leq\int_{s_1}^{s_2}\bigg[-\frac{|b|b_s-\omega_sG'(\omega)|b|+C\mathcal{N}_{1,\rm loc}+\delta(\kappa)|b|^3}{\lambda^2}\bigg]\nonumber\\
&\leq-\frac{1}{2}\frac{b|b|}{\lambda^2}\bigg|_{s_1}^{s_2}+O\bigg(\int_{s_1}^{s_2}\frac{\mathcal{N}_{1,\rm loc}+\omega b^2}{\lambda^2}\bigg)+\delta(\kappa)\int_{s_1}^{s_2}\frac{|b|^3}{\lambda^2}\label{C411}.
\end{align}
Recall that $\omega=\frac{\gamma}{\lambda^m}$. Then from \eqref{CMS1} we have:
\begin{align}
\int_{s_1}^{s_2}\frac{\omega b^2}{\lambda^2}&=-\int_{s_1}^{s_2}\frac{\lambda_s}{\lambda}\frac{\omega b}{\lambda^2}+\int_{s_1}^{s_2}\frac{\omega b}{\lambda^2}\bigg(\frac{\lambda_s}{\lambda}+b\bigg)\nonumber\\
&\leq\frac{1}{m+2}\int_{s_1}^{s_2}\bigg(\frac{\omega}{\lambda^2}\bigg)_sb+\delta(\kappa)\int_{s_1}^{s_2}\frac{\omega b^2}{\lambda^2}+O\bigg(\int_{s_1}^{s_2}\frac{\mathcal{N}_{1,\rm loc}}{\lambda^2}\bigg)\nonumber\\
&=\frac{1}{m+2}\int_{s_1}^{s_2}\frac{\omega}{\lambda^2}(-b_s)+\delta(\kappa)\int_{s_1}^{s_2}\frac{\omega b^2}{\lambda^2}\nonumber\\
&\quad+O\bigg(\int_{s_1}^{s_2}\frac{\mathcal{N}_{1,\rm loc}}{\lambda^2}+\frac{b^2(s_1)+\omega^2(s_1)}{\lambda^2(s_1)}+\frac{b^2(s_2)+\omega^2(s_2)}{\lambda^2(s_2)}\bigg).\label{C412}
\end{align}
From \eqref{C212} and \eqref{CMS2}, we have:
\begin{align}
\int_{s_1}^{s_2}\frac{\omega}{\lambda^2}(-b_s)&\leq\int_{s_1}^{s_2}\frac{\omega}{\lambda^2}\bigg[\bigg(2+\frac{m}{10}\bigg)b^2+\omega_sG'(\omega)+C(m)\mathcal{N}_{1,\rm loc}\bigg]\nonumber\\
&\leq \bigg(2+\frac{m}{10}\bigg)\int_{s_1}^{s_2}\frac{\omega b^2}{\lambda^2}+\int_{s_1}^{s_2}\frac{\omega_s\omega G'(\omega)}{(\gamma/\omega)^{2/m}}\nonumber\\
&\quad+O\bigg(\int_{s_1}^{s_2}\frac{\mathcal{N}_{1,\rm loc}}{\lambda^2}+\frac{b^2(s_1)+\omega^2(s_1)}{\lambda^2(s_1)}+\frac{b^2(s_2)+\omega^2(s_2)}{\lambda^2(s_2)}\bigg).\label{C413}
\end{align}
From \eqref{CMS1} and \eqref{CMS2} again, we have:
\begin{align}
\bigg|\int_{s_1}^{s_2}\frac{\omega_s\omega G'(\omega)}{(\gamma/\omega)^{2/m}}\bigg|=\frac{|M(\omega(s_2))-M(\omega(s_1))|}{\gamma^{2/m}}\lesssim\frac{\omega^2(s_1)}{\lambda^2(s_1)}+\frac{\omega^2(s_2)}{\lambda^2(s_2)},\label{C415}
\end{align}
where
$$M(\omega)=\int_{0}^\omega x^{1+2/m}G'(x)\,dx\sim\omega^{2+2/m}.$$
Therefore, combining \eqref{C412}--\eqref{C415}, we have
\begin{align}\label{C416}
\int_{s_1}^{s_2}\frac{\omega b^2}{\lambda^2}\leq&\bigg(\frac{2+m/10}{m+2}+\delta(\kappa)\bigg)\int_{s_1}^{s_2}\frac{\omega b^2}{\lambda^2}\nonumber\\
&+O\Bigg(\int_{s_1}^{s_2}\frac{\mathcal{N}_{1,\rm loc}}{\lambda^2}+\bigg[\frac{b^2(s_1)+\omega^2(s_1)}{\lambda^2(s_1)}+\frac{b^2(s_2)+\omega^2(s_2)}{\lambda^2(s_2)}\bigg]\Bigg).
\end{align}
Taking $\kappa>0$ small enough, from \eqref{C411} and \eqref{C416} we have
\begin{align}\label{C417}
\int_{s_1}^{s_2}\frac{b^2(\omega+|b|)}{\lambda^2}\lesssim\int_{s_1}^{s_2}\frac{\mathcal{N}_{1,\rm loc}}{\lambda^2}+\bigg[\frac{b^2(s_1)+\omega^2(s_1)}{\lambda^2(s_1)}+\frac{b^2(s_2)+\omega^2(s_2)}{\lambda^2(s_2)}\bigg],
\end{align}
Now, integrating the monotonicity formula \eqref{CMF2}, we have:
\begin{align*}
&\frac{\mathcal{N}_i(s_2)}{\lambda^2(s_2)}+\int_{s_1}^{s_2}\frac{1}{\lambda^2(s)}\bigg[\bigg(\int\big(\varepsilon_{y}^2+\varepsilon^2\big)(s)\varphi'_{i,B}\bigg)\bigg]\,ds\nonumber\\
&\lesssim \frac{\mathcal{N}_i(s_1)}{\lambda^2(s_1)}+\delta(\kappa)\int_{s_1}^{s_2}\frac{b^2(s)\big(\omega(s)+|b(s)|\big)}{\lambda^2(s)}\,ds,
\end{align*}
which implies \eqref{C45} immediately.

{\it Proof of \eqref{C47}}. The proof of \eqref{C47} based on integrating the equation of $\frac{b}{\lambda^2}$, i.e. \eqref{CLOBOLS}. More precisely, from \eqref{CMS1}, \eqref{CLOBOLS} and the fact that $|J|\lesssim \mathcal{N}_{1,\rm loc}^{\frac{1}{2}}$ (recall that $J$ given by \eqref{C204} is a well localized $L^2$ scalar product), we have:
\begin{align*}
&\bigg|\bigg(\frac{b}{\lambda^2}e^{J}\bigg)_s+\frac{\omega_sG'(\omega)}{\lambda^2}e^J\bigg|=\bigg|\bigg(\frac{b}{\lambda^2}\bigg)_s+\frac{b}{\lambda^2}J_s+\frac{\omega_sG'(\omega)}{\lambda^2}\bigg|e^J\nonumber\\
&\lesssim \bigg|\frac{\lambda_s}{\lambda}\frac{b}{\lambda^2}J\bigg|+O\Bigg(\frac{1}{\lambda^2}\bigg(\int \varepsilon^2e^{-\frac{|y|}{10}}+(\omega+|b|)b^2\bigg)\Bigg)\nonumber\\
&\lesssim\frac{b^2}{\lambda^2}|J|+O\Bigg(\frac{1}{\lambda^2}\bigg(\int \varepsilon^2e^{-\frac{|y|}{10}}+(\omega+|b|)b^2\bigg)\Bigg)\nonumber\\
&\lesssim O\bigg(\frac{1}{\lambda^2}\big(\mathcal{N}_{1,\rm loc}+(\omega+|b|)b^2\big)\bigg).
\end{align*}
We integrate this estimate in time using \eqref{C45} to get
\begin{align}
\Bigg|\bigg[\frac{b}{\lambda^2}e^J\bigg]_{s_1}^{s_2}+\int_{s_1}^{s_2}\frac{\omega_sG'(\omega)}{\lambda^2}e^J\Bigg|\lesssim\frac{\mathcal{N}_1(s_1)}{\lambda^2(s_1)}+\bigg[\frac{b^2(s_1)+\omega^2(s_1)}{\lambda^2(s_1)}+\frac{b^2(s_2)+\omega^2(s_2)}{\lambda^2(s_2)}\bigg].\label{C418}
\end{align}
Note that $|e^J-1|\leq2|J|\lesssim \mathcal{N}_{1,\rm loc}^{\frac{1}{2}}$. Together with \eqref{C45}, we have
\begin{align}
\bigg[\frac{b}{\lambda^2}e^J\bigg]_{s_1}^{s_2}&=\frac{b}{\lambda^2}\bigg|_{s_1}^{s_2}+\bigg|\bigg[\frac{b}{\lambda^2}\mathcal{N}^{\frac{1}{2}}_{1,\rm loc}\bigg]_{s_1}^{s_2}\bigg|\nonumber\\
&=\frac{b}{\lambda^2}\bigg|_{s_1}^{s_2}+O\bigg(\frac{\mathcal{N}_1(s_1)}{\lambda^2(s_1)}+\frac{b^2(s_1)+\omega^2(s_1)}{\lambda^2(s_1)}+\frac{b^2(s_2)+\omega^2(s_2)}{\lambda^2(s_2)}\bigg).\label{C419}
\end{align}
Next, from \eqref{CMS2}, \eqref{C45} and $|J|\lesssim \mathcal{N}_{1,\rm loc}^{\frac{1}{2}}$, we have
\begin{align}
&\bigg|\int_{s_1}^{s_2}\frac{\omega_sG'(\omega)}{\lambda^2}(e^J-1)\bigg|\lesssim\int_{s_1}^{s^2}\frac{(b^2+\omega^2)b^2+\mathcal{N}_{1,\rm loc}}{\lambda^2}\nonumber\\
&\lesssim\frac{\mathcal{N}_1(s_1)}{\lambda^2(s_1)}+\frac{b^2(s_1)+\omega^2(s_1)}{\lambda^2(s_1)}+\frac{b^2(s_2)+\omega^2(s_2)}{\lambda^2(s_2)}.\label{C420}
\end{align}
Finally, recall $\omega=\gamma/\lambda^m$, so we have:
$$\int_{s_1}^{s_2}\frac{\omega_sG'(\omega)}{\lambda^2}=\int_{s_1}^{s_2}\frac{\omega_sG'(\omega)}{(\gamma/\omega)^{2/m}}=\frac{\Sigma(\omega)}{\lambda^2}\bigg|_{s_1}^{s_2},$$
where
\begin{gather*}
\Sigma(\omega):=\frac{1}{\omega^{2/m}}\int_{0}^{\omega}x^{2/m}G'(x)\,dx.
\end{gather*}
Recall that $G$ is the $C^2$ function introduced in \eqref{CLOB}. We then have $\Sigma\in C^2$ and $c_1=\Sigma'(0)=\frac{m}{m+2}G'(0)>0$. Hence, we have
\begin{equation}
\label{C421}
\int_{s_1}^{s_2}\frac{\omega_sG'(\omega)}{\lambda^2}=\frac{c_1\omega}{\lambda^2}\bigg|_{s_1}^{s_2}+O\bigg(\frac{\omega^2(s_1)}{\lambda^2(s_1)}+\frac{\omega^2(s_2)}{\lambda^2(s_2)}\bigg),
\end{equation}

Combining \eqref{C418}--\eqref{C421}, we conclude the proof of \eqref{C47}.

{\it Proof of \eqref{C48}}. From \eqref{C36}, we have
$$\bigg|\frac{\lambda}{\lambda_0}-1\bigg|\lesssim|J_1|\lesssim\mathcal{N}_{2}^{\frac{1}{2}}\lesssim \delta(\kappa),$$
thus we obtain from \eqref{CLOL}:
\begin{align*}
\bigg|\frac{(\lambda_0)_s}{\lambda_0}+b\bigg|&=\bigg|\frac{1}{1-J_1}\bigg[(1-J_1)\frac{\lambda_s}{\lambda}+b-2(J_1)_s\bigg]-\frac{J_1}{1-J_1}b\bigg|\\
&\lesssim\int\varepsilon^2e^{-\frac{|y|}{10}}+(|b|+\omega)(\mathcal{N}^{\frac{1}{2}}_2+|b|).
\end{align*}
This concludes the proof of \eqref{C48}, hence the proof of Lemma \ref{CL5}.
\end{proof}

\subsection{Rigidity dynamics in $\mathcal{A}_{\alpha_0}$} In this part, we will give a specific classification for the asymptotic behavior of solution with initial data in $\mathcal{A}_{\alpha_0}$.

We first introduce the separation time $t_1^*$: 
\begin{align}
&t_1^*=0, \text{ if }|b(0)+c_1\omega(0)|\geq C^*\big(\mathcal{N}_1(0)+b^2(0)+\omega^2(0)\big),\nonumber\\
&t_1^*=\sup\Big\{0<t<t^*\big|\text{for all }t'\in[0,t],\nonumber\\
&\qquad\qquad\quad|b(t')+c_1\omega(t')|\leq C^*\big(\mathcal{N}_1(t')+b^2(t')+\omega^2(t')\big)\Big\},\;\text{otherwise},
\end{align}
where%
\footnote{Recall that $K_0$, $K_1$, $K_2$ and $c_1$ were introduced in Lemma \ref{CL5}.} 
\begin{equation}\label{C4010}
C^*=100(K_1+K_0K_2)>0.
\end{equation}
Then we have:
\begin{proposition}[Rigidity Dynamics]\label{CP5}
There exist universal constants $0<\gamma\ll\alpha_0\ll\alpha^*\ll1$ and $C^*>1$ such that the following holds. Let $u_0\in\mathcal{A}_{\alpha_0}$, and $u(t)$ be the corresponding solution to \eqref{CCPG}. Then we have:

\noindent (1) The following trichotomy holds:
\begin{itemize}
\item{\bf(Blow down):} If $t_1^*=t^*$, then $t_1^*=t^*=T=+\infty$ with, 
\begin{align}
&|b(t)|+\mathcal{N}_2(t)\rightarrow0,\quad\text{as }t\rightarrow+\infty,\label{C403}\\
&\lambda(t)\sim t^{\frac{2}{q+1}},\; x(t)\sim t^{\frac{q-3}{q+1}},\quad\text{as }t\rightarrow+\infty\label{C404}.
\end{align}

\item{\bf (Exit):} If $t_1^*<t^*$ with 
$$b(t^*_1)+c_1\omega(t^*_1)\leq -C^*\big(\mathcal{N}_1(t_1^*)+b^2(t^*_1)+\omega^2(t^*_1)\big),$$
then $t^*<T=+\infty$. In particular,
\begin{equation}\label{C405}
\inf_{\lambda_0>0,\lambda_0^{-m}\gamma<\omega^*,x_0\in\mathbb{R}}\bigg\|u(t^*)-\frac{1}{\lambda_0^{\frac{1}{2}}}\mathcal{Q}_{\lambda_0^{-m}\gamma}\bigg(\frac{x-x_0}{\lambda_0}\bigg)\bigg\|_{L^2}=\alpha^*.
\end{equation}
Moreover, we have
\begin{equation}\label{C406}
b(t^*)\leq-C(\alpha^*)<0,\quad\lambda(t^*)\geq\frac{C(\alpha^*)}{\delta(\alpha_0)}\gg1.
\end{equation}

\item{\bf (Soliton):} If $t_1^*<t^*$ with 
$$b(t^*_1)+c_1\omega(t^*_1)\geq C^*\big(\mathcal{N}_1(t_1^*)+b^2(t^*_1)+\omega^2(t^*_1)\big),$$
then $t^*=T=+\infty$. Moreover, we have:
\begin{align}
&\mathcal{N}_2(t)+|b(t)|\rightarrow0,\quad\text{as }t\rightarrow+\infty,\label{C407}\\
&\lambda(t)=\lambda_{\infty}\big(1+o(1)\big),\quad x(t)=\frac{t}{\lambda_{\infty}^2}\big(1+o(1)\big),\quad\text{as }t\rightarrow+\infty,\label{C408}
\end{align}
for some $\lambda_{\infty}\in(0,+\infty)$.
\end{itemize}

\noindent (2) All of the three scenarios introduced in (1) are known to occur. Moreover, the initial data sets which lead to the (Soliton) and (Exit) case are open in $\mathcal{A}_{\alpha_0}$ (under the topology of $H^1\cap L^2(y_+^{10}dy)$).
\end{proposition}
\begin{remark}
It is easy to see Proposition \ref{CP5} implies Theorem \ref{CMT} immediately.
\end{remark}
\begin{remark}
The constant $C^*$ chosen here is not sharp. We can replace it by some slightly different ones.
\end{remark}

\begin{proof}[Proof of Proposition \ref{CP5}]

{The basic idea of the proof is to show that the assumptions (H1)--(H3) introduced in Proposition \ref{CP4} hold for all $t\in[0,t^*)$ ($i.e.$ as long as the solution is close to the soliton manifold). And then together with the estimates obtained in Lemma \ref{CL5}, we can show that the error term $\varepsilon$ does not perturb the ODE system, and hence the parameters $(b,\lambda,x)$ has the same asymptotic behavior as the formal system \eqref{CFDS}, which concludes the proof of Theorem \ref{CMT}. 

In the (Blow down) and (Exit) cases, this is done by improving the estimates in (H1)--(H3) on $[0,t^{**}]$ (recall that $t^{**}$ is the largest time $t$ that (H1)--(H3) hold on $[0,t]$), and then a standard bootstrap argument shows that $t^{**}=t^{*}$, $i.e.$ (H1)--(H3) hold on $[0,t^*)$. While in the (Soliton) case, it seems hard to improve all the estimates on $[0,t^{**}]$. But, fortunately, we can use a similar bootstrap argument to show that some modified assumptions $\rm(H1)'$, $\rm (H2)'$, $\rm(H3)'$ hold on $[0,t^*]$, which is also enough to obtain the asymptotic behavior of the parameters.
} 

\subsubsection*{1. The blow down case.} Assume that $t^*_1=t^*$, i.e. for all $t\in[0,t^*]$,
\begin{equation}
\label{C422}
|b(t)+c_1\omega(t)|\leq C^*\big(\mathcal{N}_1(t)+b^2(t)+\omega^2(t)\big).
\end{equation}

{\noindent {\bf Step 1: Closing the bootstrap.}} 

We claim that $t^{**}=t^*$, i.e. the bootstrap assumptions (H1), (H2) and (H3) hold on $[0,t^*]$. 

Indeed, we claim that for all $s\in[0,s^{**})$, 
\begin{gather}
\omega(s)+|b(s)|+\|\varepsilon(s)\|_{L^2}+\mathcal{N}_2(s)\lesssim\delta(\alpha_0),\label{C423}\\
\lambda(s)\geq \frac{4}{5},\label{C424}\\
\int_{y>0}y^{10}\varepsilon^2(s)\,dy\leq 5.\label{C425}
\end{gather}
Then choosing $\alpha^*$, $\alpha_0$, $\gamma$ such that $0<\gamma\ll\alpha_0\ll\alpha^*\ll\kappa$, we can see that \eqref{C423}--\eqref{C425} imply $t^{**}=t^*$ immediately.

First, from \eqref{C422} we have: for all $s\in[0,s^{**})$,
\begin{gather}
b(s)\leq 4C^*\mathcal{N}_1(s)-|b(s)|,\label{C426}\\
|b(s)|\lesssim \mathcal{N}_1(s)+\omega(s),\label{C427}\\
\omega(s)\lesssim\mathcal{N}_1(s)+|b(s)|.\label{C428}
\end{gather}
Then we apply \eqref{C426} and \eqref{C428} to \eqref{C48} to obtain:
\begin{align*}
\frac{(\lambda_0)_s}{\lambda_0}&\geq-b-\mathcal{N}_1-C(\omega+|b|)(\mathcal{N}_2^{\frac{1}{2}}+|b|)\\
&\geq -5C^*\mathcal{N}_1+|b|-\delta(\kappa)|b|\\
&\gtrsim -\mathcal{N}_{1}.
\end{align*}
Integrating this from $s_1$ to $s_2$ for some $0\leq s_1<s_2\leq s^{**}$, and using \eqref{C44}, we have:
\begin{equation}
\label{C429}
\lambda(s_2)\geq\frac{9}{10}\lambda(s_1).
\end{equation}
In particular, we know from \eqref{C42} that for all $s\in[0,s^{**})$
\begin{equation}\label{C402}
\lambda(s)\geq\frac{9}{10}\lambda(0)\geq\frac{4}{5}.
\end{equation}
By our choice of $\gamma$, we have
\begin{equation}
\label{C430}
\omega(s)=\frac{\gamma}{\lambda^m(s)}\leq 2^m \gamma\lesssim\delta(\alpha_0).
\end{equation}
Next, from \eqref{C43}, \eqref{C44} and \eqref{C427}, we have for all $s\in[0,s^{**})$
$$\mathcal{N}_2(s)\lesssim \mathcal{N}_2(0)+\sup_{s'\in[0,s]}\mathcal{N}^3_2(s')+\sup_{s'\in[0,s]}\omega^3(s'),$$
which together with \eqref{C427} implies that 
$$|b(s)|+\mathcal{N}_2(s)\lesssim\delta(\alpha_0),$$
for all $s\in[0,s^{**})$.
Then from \eqref{CMC} and the condition on the initial data, we obtain
\begin{equation}
\|\varepsilon(s)\|_{L^2}\lesssim\delta(\alpha_0).
\end{equation}
From \eqref{CEC} and the condition on the initial data, we have
$$\frac{\|\varepsilon_{y}(s)\|^2_{L^2}}{\lambda^2(s)}\lesssim \delta(\alpha_0)+\frac{\|\varepsilon_{y}(s)\|^{m+2}_{L^2}}{\lambda^{m+2}(s)}.$$
Since $\|\varepsilon_{y}(0)\|_{L^2}\lesssim \delta(\alpha_0)$, $\lambda(0)\sim 1$, from a standard bootstrap argument, we have:
$$\frac{\|\varepsilon_{y}(s)\|^2_{L^2}}{\lambda^2(s)}\lesssim\delta(\alpha_0).$$
Thus, we have
\begin{equation}\label{C4014}
\omega(s)\|\varepsilon_{y}(s)\|^m_{L^2}\lesssim\gamma\frac{\|\varepsilon_{y}(s)\|^m_{L^2}}{\lambda^m(s)}\lesssim \delta(\alpha_0).
\end{equation}

Finally, let us integrate \eqref{C37} from $0$ to $s$, using \eqref{C4013}, \eqref{C44}, \eqref{C46}, \eqref{C429} and \eqref{C402} to obtain
\begin{align*}
\int\varphi_{10}\varepsilon^2(s)\,dy&\leq\frac{\lambda^{10}(0)}{\lambda^{10}(s)}\int\varphi_{10}\varepsilon^2(0)\,dy+C\int_0^s\frac{\lambda^{10}(s')}{\lambda^{10}(s)}\big(\mathcal{N}_{1,\rm loc}(s')+b^2(s')\big)\,ds'\\
&\leq 3+C\int_{0}^s\big(\mathcal{N}_{1,\rm loc}(s')+b^2(s')\big)\,ds'\leq3+\delta(\kappa)<5.
\end{align*}

We therefore conclude the proof of \eqref{C423}--\eqref{C425}, and obtain $t^{**}=t^*$. Since $0<\alpha_0\ll\alpha^*$, the estimate \eqref{C423} implies that $t^{**}=t^*=T=+\infty$.\\

{\noindent{\bf Step 2: Proof of \eqref{C403} and \eqref{C404}.}} 

We first claim that $\lambda(t)\rightarrow+\infty$ as $t\rightarrow+\infty$. Let
$$S=\int_0^{+\infty}\frac{1}{\lambda^3(\tau)}\,d\tau\in(0,+\infty].$$
From \eqref{CMS2}, \eqref{C44}, \eqref{C46} and \eqref{C428} we have:
\begin{gather*}
\int_0^{+\infty}|\omega_t|\,dt=\int_0^{S}|\omega_s|\,ds\lesssim\int_0^S(\mathcal{N}_{2,\rm loc}(s)+b^2(s))\,ds<+\infty,\\
\int_0^{+\infty}\frac{\gamma^2}{\lambda^{3+2m}(t)}\,dt=\int_0^S\omega^2(s)\,ds\lesssim\int_0^S(\mathcal{N}_{2,\rm loc}(s)+b^2(s))\,ds<+\infty.
\end{gather*}
This leads to $\lambda(t)\rightarrow+\infty$ as $t\rightarrow+\infty$, or equivalently $\lim_{t\rightarrow+\infty}\omega(t)=0$.

Next, we claim that $S=+\infty$. Otherwise, $b(s),\omega(s)\in L^1([0,S))$. Applying this to \eqref{C48}, we obtain:
$$\frac{(\lambda_0)_s}{\lambda_0}\in L^1([0,S)).$$
But since $\lambda_0(s)\rightarrow+\infty$ as $s\rightarrow S$, we have:
$$\bigg|\int_0^{S-\delta_0}\frac{(\lambda_0)_s}{\lambda_0}(s')\,ds'\bigg|=\bigg|\log\bigg(\frac{\lambda_0(S-\delta_0)}{\lambda_0(0)}\bigg)\bigg|\rightarrow+\infty,$$
as $\delta_0\rightarrow0$, which leads to a contradiction. 

Now we can prove \eqref{C403} and \eqref{C404}. To do this, we claim that for all $s\in[0,+\infty)$,
\begin{equation}\label{C4011}
\lambda^{m}(s)\mathcal{N}_2(s)+\int_{0}^s\lambda^{m}(s')(\varepsilon^2(s')+\varepsilon_{y}^2(s'))\varphi_{2,B}'\,ds'\lesssim 1.
\end{equation}
From \eqref{CMF1} we have:
\begin{align}\label{C4012}
\frac{1}{\lambda^{m}}\frac{d}{ds}\bigg(\lambda^{m}\mathcal{F}_{2,1}\bigg)\leq-\mu\int(\varepsilon^2+\varepsilon^2_{y})\varphi_{2,B}'+O(b^4+\omega^2b^2)-m\frac{\lambda_s}{\lambda}\mathcal{F}_{2,1}.
\end{align}
From \eqref{CMS1}, \eqref{CCOER}, \eqref{C39} and \eqref{C402}, we have
\begin{align*}
\bigg|\frac{\lambda_s}{\lambda}\mathcal{F}_{2,1}\bigg|&\lesssim(|b|+\mathcal{N}_{1,\rm loc}^{\frac{1}{2}})\Bigg[\bigg(1+\frac{1}{\lambda^{\frac{10}{9}}(s)}\bigg)\mathcal{N}_{2,\rm loc}^{\frac{8}{9}}+\int\varepsilon^2_{y}\psi_B\Bigg]\\
&\lesssim b^2+\delta(\kappa)\int(\varepsilon^2+\varepsilon^2_{y})\varphi_{2,B}'.
\end{align*}
Injecting this into \eqref{C4012} and integrating from $0$ to $s$, using \eqref{C427} and \eqref{C428}, we have,
\begin{align*}
&\lambda^{m}(s)\mathcal{N}_2(s)+\int_{0}^s\lambda^{m}(s')(\varepsilon^2(s')+\varepsilon_{y}^2(s'))\varphi_{2,B}'\,ds'\\
&\lesssim \int_0^s \lambda^{m}(s')\omega^4(s')\,ds'+\delta(\kappa)\int_0^s\lambda^{m}(s')\mathcal{N}_1(s')\,ds'\\
&\lesssim\gamma\int_0^s \omega^3(s')\,ds'+\delta(\kappa)\int_0^s\lambda^{m}(s')\mathcal{N}_1(s')\,ds'\\
&\lesssim\gamma\int_0^s b^2(s')\,ds'+\delta(\kappa)\int_0^s\lambda^{m}(s')\mathcal{N}_1(s')\,ds'.
\end{align*}
Together with \eqref{C46}, we obtain \eqref{C4011}.

Since $\lambda(s)\rightarrow+\infty$ as $s\rightarrow+\infty$, we have
\begin{equation}\label{C4001}
\mathcal{N}_2(s)\lesssim \lambda^{-m}(s)\rightarrow0\quad\text{as }s\rightarrow+\infty.
\end{equation}

Now, using \eqref{C48}, \eqref{C422} and \eqref{C427}, we have
\begin{align*}
\bigg|-\frac{(\lambda_0)_s}{\lambda_0}+c_1\omega\bigg|&\lesssim\bigg|\frac{(\lambda_0)_s}{\lambda_0}+b\bigg|+|b+c_1\omega|\\ 
&\lesssim\mathcal{N}_{1}+b^2+\omega^2+(|b|+\omega)(\mathcal{N}_2^{\frac{1}{2}}+|b|)\\
&\lesssim \mathcal{N}_1+\delta(\kappa)\omega.
\end{align*}
Multiplying the above inequality by $\lambda_0^m$ and integrating from $0$ to $s$, we obtain
$$-C\int_0^s \lambda_0^m\mathcal{N}_{1}+\frac{1}{2}c_1\gamma s\leq \int_0^s (\lambda_0)_s\lambda_0^{m-1}\leq C\int_0^s \lambda_0^m\mathcal{N}_{1}+2c_1\gamma s.$$
From \eqref{C4011} and $|1-\lambda/\lambda_0|\lesssim \delta(\kappa)$, we obtain
$$\lambda^m(s)\sim s,\quad\text{as }s\rightarrow+\infty.$$
We then have,
$$t(s)=\int_0^s\lambda^3(s')\,ds'\sim s^{\frac{m+3}{m}}=s^{\frac{q+1}{q-5}},\quad\text{as }s\rightarrow+\infty,$$
which implies 
$$\lambda(t)\sim t^{\frac{2}{q+1}},\quad \text{as }t\rightarrow+\infty.$$
Next, from \eqref{C422} and \eqref{C427}, we have
$$b(t)\rightarrow0,\quad\text{as }t\rightarrow+\infty.$$ 
Finally, integrating \eqref{CMS1}, we obtain:
$$x(t)\sim t^{\frac{q-3}{q+1}},\quad\text{as }t\rightarrow+\infty,$$
which concludes the proof of \eqref{C403} and \eqref{C404}.

\subsubsection*{2. The Exit case.} Assume $t_1^*<t^*$ with
\begin{equation}
\label{C432}
b(t^*_1)+c_1\omega(t^*_1)\leq -C^*\big(\mathcal{N}_1(t_1^*)+b^2(t^*_1)+\omega^2(t^*_1)\big).
\end{equation}

{\noindent {\bf Step 1: Closing the bootstrap.}} 

First of all, following the same procedure as in the (Blow down) case, we have for all $s\in[0,s_1^*]$,
\begin{gather}
\omega(s)+|b(s)|+\|\varepsilon(s)\|_{L^2}+\omega(s)\|\varepsilon_{y}(s)\|_{L^2}^m+\mathcal{N}_2(s)\lesssim\delta(\alpha_0),\\
\lambda(s)\geq \frac{4}{5},\\
\int_{y>0}y^{10}\varepsilon^2(s)\,dy\leq 5.
\end{gather}
In particular, we have $t_1^*<t^{**}\leq t^*$. Now, we claim $t^{**}=t^*<T=+\infty$.

To prove this, we use a standard bootstrap argument by improving (H1), (H2) and (H3) on $[t^*_1,t^{**}]$.
Let 
$$\ell^*=\frac{b(t^*_1)+c_1\omega(t^*_1)}{\lambda^2(t^*_1)}<0.$$
It is easy to see that $|\ell^*|\lesssim \delta(\alpha_0)$. Now we observe from \eqref{C47} that for all $s\in[s_1^*,s^{**})$, 
\begin{equation*}
2\ell^*-C^*\frac{b^2(s)+\omega^2(s)}{\lambda^2(s)}\leq\frac{b(s)+c_1\omega(s)}{\lambda^2(s)}\leq\frac{\ell^*}{2}+C^*\frac{b^2(s)+\omega^2(s)}{\lambda^2(s)}.
\end{equation*}
which implies 
\begin{gather}
-b(s)\gtrsim\omega(s)>0,\label{C434}\\
3\ell^*-C\frac{\omega(s)}{\lambda^2(s)}\leq\frac{b(s)}{\lambda^2(s)}\leq \frac{\ell^*}{3}<0.\label{C435}
\end{gather}
We then observe from \eqref{C48} and \eqref{C434} that, 
$$\frac{(\lambda_0)_s}{\lambda_0}\gtrsim-\mathcal{N}_{1,\rm loc},$$
which after integration, yields the almost monotonicity:
\begin{equation}
\label{C436}
\forall s_1^*\leq s_1<s_2\leq s^{**},\quad \lambda(s_2)\geq \frac{9}{10}\lambda(s_1)\geq\frac{1}{2}.
\end{equation}
So we obtain for all $s\in[s^*_1,s^{**})$,
$$\omega(s)+\frac{\omega(s)}{\lambda^2(s)}\lesssim\gamma\lesssim\delta(\alpha_0).$$
Together with \eqref{C45} and \eqref{C435}, we have for all $s\in[s_1^*,s^{**})$,
$$\frac{|b(s)|+\mathcal{N}_2(s)}{\lambda^2(s)}\lesssim \delta(\alpha_0),$$
which improves (H2) if we choose $\alpha_0\ll\kappa$. Next, 
Using the same strategy as in the (Blow down) case, we have for all $s\in[s_1^*,s^{**})$,
$$\int\varphi_{10}\varepsilon^2(s)\,dy\leq 7.$$
Then, (H3) is improved. We now only remains to improve (H1). Since for all $t\in[t^*_1,t^{*})$, $u(t)\in\mathcal{T}_{\alpha^*,\gamma}$. Following the argument in Lemma \ref{CL3}, we have for all $t\in[0,t^*)$, $|b(t)|\lesssim\delta(\alpha^*)$. By \eqref{CMC}, \eqref{C44}, and \eqref{C434}, we have for all $s\in[s^*_1,s^{**})$, 
$$\omega(s)+\|\varepsilon(s)\|_{L^2}+\mathcal{N}_2(s)\lesssim\delta(\alpha^*).$$
Now, following from the same argument as we did for \eqref{C4014}, we have:
$$\omega(s)\|\varepsilon_{y}(s)\|_{L^2}^m\lesssim\delta(\alpha_0).$$ 
Then (H1) is improved, due to our choice of the universal constant, i.e. $\alpha^*\ll\kappa$.

In conclusion, we have proved $t^{**}=t^*.$\\

{\noindent {\bf Step 2: Proof of \eqref{C405} and \eqref{C406}.}} 

We first claim that (Exit) occurs in finite time $t^*<+\infty$. Dividing \eqref{C48} by $\lambda^2_0$, and use \eqref{C434} to estimate on $[t^*_1,t^*)$
$$-\frac{\ell^*}{3}-C\frac{\mathcal{N}_{1,\rm loc}}{\lambda^2}\leq(\lambda_0)_t\leq -3\ell^*+C\frac{\mathcal{N}_{1,\rm loc}}{\lambda^2}.$$
Integrating from $t^*_1$ to $t$, we get
$$\frac{|\ell^*|(t-t^*_1)}{3}-C_1\int_{t^*_1}^t\frac{\mathcal{N}_{1,\rm loc}}{\lambda^2}\leq\lambda_0(t)-\lambda_0(t_1^*)\leq3|\ell^*|(t-t^*_1)+C_2\int_{t_1^*}^t\frac{\mathcal{N}_{1,\rm loc}}{\lambda^2}.$$
From \eqref{C436}, we have
$$\int_{t_1^*}^t\frac{\mathcal{N}_{1,\rm loc}}{\lambda^2}=\int_{s_1^*}^s\lambda\mathcal{N}_{1,\rm loc}\lesssim\lambda(s)\int_{s_1^*}^s\mathcal{N}_{1,\rm loc}\lesssim\delta(\kappa)\lambda(t).$$
Thus, for all $t\in[t^*_1,t^*)$,
$$\frac{1}{4}(|\ell^*|(t-t^*_1)+\lambda_0(t_1^*))\leq\lambda(t)\leq4(|\ell^*|(t-t^*_1)+\lambda_0(t_1^*)).$$
Next, from \eqref{C434}, we have for all $t\in[t^*_1,t^*)$,
$$-100|\ell^*|(|\ell^*|(t-t^*_1)+\lambda_0(t_1^*))^2\leq b(t)\leq-\frac{|\ell^*|}{100}(|\ell^*|(t-t^*_1)+\lambda_0(t_1^*))^2.$$
If $t^*=T=+\infty$, then the above estimate leads to $b(t)\rightarrow-\infty$ as $t\rightarrow+\infty$, which contradicts with the fact that $|b(t)|\lesssim \delta(\alpha^*)$ for all $t\in[t^*_1,t^*)$. Thus, we have $t^*<T=+\infty.$

Finally, since $0<t^*<+\infty$, by the definition of $t^*$, we must have $-b(t^*)\geq C(\alpha^*)>0$. While from \eqref{C434}, we have
$$\lambda^2(t^*)\geq\frac{1}{2}\frac{|b(t^*)|}{|\ell^*|}\gtrsim\frac{C(\alpha^*)}{\delta(\alpha_0)}\gg1,$$
which concludes the proof of \eqref{C405} and \eqref{C406}.

\subsubsection*{3. The Soliton case.} Assume $t_1^*<t^*$ with
\begin{equation}
\label{C437}
b(t^*_1)+c_1\omega(t^*_1)\geq C^*\big(\mathcal{N}_1(t_1^*)+b^2(t^*_1)+\omega^2(t^*_1)\big)>0.
\end{equation}

{\noindent {\bf Step 1: Estimates on the rescaled solution.}} 

Similar to the (Exit) case, we have for all $s\in[0,s_1^*]$,
\begin{gather}
\omega(s)+|b(s)|+\|\varepsilon(s)\|_{L^2}+\omega(s)\|\varepsilon_{y}(s)\|_{L^2}^m+\mathcal{N}_2(s)\lesssim\delta(\alpha_0),\label{C438}\\
\lambda(s)\geq \frac{4}{5},\label{C439}\\
\int_{y>0}y^{10}\varepsilon^2(s)\,dy\leq 5.\label{C440}
\end{gather} 

But here, we can't directly prove that $t^{**}=t^*$ as we did in the (Exit) case. The main difficulty is that we lack some control on the upper bound of $\lambda(t^*_1)$, which makes it hard to improve the bootstrap assumption (H2) and (H3). However, we will see that the bootstrap assumption (H2) and (H3) is related to the scaling symmetry of the problem. If we use the pseudo-scaling rule \eqref{C1} on $[t^*_1,t^*)$ to rescale $\lambda(t_1^*)$ to $1$, then we can get the desired result. Roughly speaking, on $[t^*_1,t^*]$, the bootstrap assumption (H2) and (H3) should be replaced by some other suitable assumptions $\rm(H2)'$ and $\rm(H3)'$.

More precisely, we introduce the following change of coordinates. For all $t\in[t^*_1,t^*)$, let
\begin{gather}
\bar{t}=\frac{t-t^*_1}{\lambda^3(t^*_1)},\;\bar{x}=\frac{x-x(t^*_1)}{\lambda(t^*_1)},\;\bar{\gamma}=\frac{\gamma}{\lambda^m(t^*_1)},\;\bar{t}^*=\frac{t^*-t^*_1}{\lambda^3(t^*_1)}\\
\bar{u}(\bar{t},\bar{x})=\lambda^{\frac{1}{2}}u\big(\lambda^3(t^*_1)\bar{t}+t^*_1,\lambda(t^*_1)\bar{x}+x(t^*_1)\big).
\end{gather} 
Then, from the pseudo-scaling rule \eqref{C1}, $\bar{u}(\bar{t},\bar{x})$ is a solution to the following Cauchy problem:
\begin{equation}\label{C441}\begin{cases}
\partial_{\bar{t}} \bar{u} +(\bar{u}_{\bar{x}\bar{x}}+\bar{u}^5-\bar{\gamma} \bar{u}|\bar{u}|^{q-1})_{\bar{x}}=0, \quad (\bar{t},\bar{x})\in[0,\bar{t}^*)\times\mathbb{R},\\
\bar{u}(0,\bar{x})=Q_{b(t^*_1),\omega(t^*_1)}(\bar{x})+\varepsilon(t^*_1,\bar{x})\in H^1(\mathbb{R}).
\end{cases}
\end{equation}
Moreover, for all $\bar{t}\in[0,\bar{t}^*)$ we define:
\begin{gather}
\bar{\varepsilon}(\bar{t},y)=\varepsilon(\lambda^3(t^*_1)\bar{t}+t^*_1,y),\;\bar{\lambda}(\bar{t})=\frac{\lambda(\lambda^3(t^*_1)\bar{t}+t^*_1)}{\lambda(t^*_1)},\;\bar{\omega}(\bar{t})=\frac{\bar{\gamma}}{\bar{\lambda}^m(\bar{t})},\\
\bar{b}(\bar{t})=b(\lambda^3(t^*_1)\bar{t}+t^*_1),\;\bar{x}(\bar{t})=\frac{x(\lambda^3(t^*_1)\bar{t}+t^*_1)-x(t^*_1)}{\lambda(t^*_1)}.
\end{gather}
Then, from \eqref{CGD}, it is easy to check that 
$$\bar{u}(\bar{t},\bar{x})=\frac{1}{\bar{\lambda}^{\frac{1}{2}}(\bar{t})}\Big[Q_{\bar{b}(\bar{t}),\bar{\omega}(\bar{t})}+\bar{\varepsilon}(\bar{t})\Big]\bigg(\frac{\bar{x}-\bar{x}(\bar{t})}{\bar{\lambda}(\bar{t})}\bigg),$$
with:
$$(\bar{\varepsilon}(\bar{s}),Q_{\bar{\omega}(\bar{s})})=(\bar{\varepsilon}(\bar{s}),\Lambda Q_{\bar{\omega}(\bar{s})})=(\bar{\varepsilon}(\bar{s}),\bar{y}\Lambda Q_{\bar{\omega}(\bar{s})})=0.$$
where $(\bar{s},\bar{y})$ are the scaling invariant variables:
$$\bar{s}=\int_0^{\bar{t}}\frac{1}{\bar{\lambda}^3(\tau)}\,d\tau,\quad \bar{y}=\frac{\bar{x}-\bar{x}(\bar{t})}{\bar{\lambda}(\bar{t})},$$
We then introduce the weighted Sobolev norms:
\begin{gather}
\overline{\mathcal{N}}_i(\bar{s})=\int\bigg(\bar{\varepsilon}_{\bar{y}}^2(\bar{s},\bar{y})\psi_B(\bar{y})+\bar{\varepsilon}^2(\bar{s},\bar{y})\varphi_{i,B}(\bar{y})\bigg)\,d\bar{y},\\
\overline{\mathcal{N}}_{i,\rm loc}(\bar{s})=\int\bar{\varepsilon}^2(\bar{s},\bar{y})\varphi_{i,B}'(\bar{y})\,d\bar{y},
\end{gather}
where $\varphi_{i,B}$ and $\psi_B$ are the weight functions introduced in Section 3. 

From now on, for all $\bar{t}\in[0,\bar{t}^*)$, we let $t=\lambda^3(t^*_1)\bar{t}+t^*_1$. In this setting, we have $\bar{s}(\bar{t})=s(t)-s^*_1$.
Since the pseudo-scaling rule \eqref{C1} is $L^2$ invariant, we have
$$\bar{u}(\bar{t})\in\mathcal{T}_{\alpha^*,\bar{\gamma}}\iff u(t)\in\mathcal{T}_{\alpha^*,\gamma},$$
which yields
$$\bar{t}^*=\sup\{0<\bar{t}<+\infty|\text{for all $t'\in[0,\bar{t}]$, $\bar{u}(t')\in\mathcal{T}_{\alpha^*,\bar{\gamma}}$}\}.$$

Next, let $\kappa>0$ be the universal constant introduced in Proposition \ref{CP3}, Proposition \ref{CP4} and Lemma \ref{CL5}. We then define the following bootstrap assumptions for the rescaled solution $\bar{u}(\bar{t},\bar{x})$. For all $\bar{s}\in[0,\bar{s}(\bar{t}))$:\\
$\rm(H1)'$ Scaling invariant bounds:
\begin{equation}
\bar{\omega}(\bar{s})+|\bar{b}(\bar{s})|+\overline{\mathcal{N}}_2(\bar{s})+\|\bar{\varepsilon}(\bar{s})\|_{L^2}+\bar{\omega}(\bar{s})\|\bar{\varepsilon}_{\bar{y}}(\bar{s})\|_{L^2}^m\leq\kappa;
\end{equation}
$\rm(H2)'$ Bounds related to $H^1$ scaling:
\begin{equation}
\frac{\bar{\omega}(\bar{s})+|\bar{b}(\bar{s})|+\overline{\mathcal{N}}_2(\bar{s})}{\bar{\lambda}^2(\bar{s})}\leq\kappa;
\end{equation}
$\rm(H3)'$ $L^2$ weighted bound on the right:
\begin{equation}
\int_{\bar{y}>0}\bar{y}^{10}\bar{\varepsilon}^2(\bar{s},\bar{y})\,d\bar{y}\leq 50\bigg(1+\frac{1}{\bar{\lambda}^{10}(\bar{s})}\bigg).
\end{equation}

We define $\bar{t}^{**}$ as following:
\begin{equation}
\bar{t}^{**}=\sup\{0<\bar{t}<\bar{t}^*|\text{$\rm(H1)'$, $\rm(H2)'$ and $\rm(H3)'$ hold for all }t'\in[0,\bar{t}]\}.
\end{equation}
Our goal here is to prove that $\bar{t}^{**}=\bar{t}^{*}=+\infty$, which gives us the desired asymptotic behaviors%
\footnote{Since $\lambda(t^*_1)\gtrsim 1$, we know that (H1) is equivalent to $\rm(H1)'$, (H2) is weaker than $\rm (H2)'$, while (H3) is stronger than $\rm (H3)'$. It is hard to determine whether $t^{**}=\lambda^3(t^*_1)\bar{t}^{**}+t^*_1$ holds.}%
. Let $\bar{s}^*=\bar{s}(\bar{t}^*)$, $\bar{s}^{**}=\bar{s}(\bar{t}^{**})$. Since
\begin{equation}\label{C4016}
\bar{\lambda}(0)=1,\;\bar{x}(0)=0,\;\bar{b}(0)=b(t^*_1),\;\bar{\omega}(0)=\omega(t^*_1),\;\bar{\varepsilon}(0,\bar{y})=\varepsilon(t^*_1,\bar{y}),\; \bar{\gamma}\lesssim \gamma, 
\end{equation}
we know from \eqref{C438}--\eqref{C440}, that $\bar{s}^{**}>0$.

On the other hand, on $[0,\bar{s}^{**})$, all conditions of Proposition \ref{CP3}, Proposition \ref{CP4}, Lemma \ref{CL4} and Lemma \ref{CL5} are satisfied for $\bar{u}(\bar{t},\bar{x})$. Repeating the same procedure, we have
\begin{lemma}[Estimates for the rescaled solution] For all $\bar{s}\in[0,\bar{s}^{**})$ or equivalently $s\in[s^*_1,s^*_1+\bar{s}^{**})$, all estimates of Proposition \ref{CP3}, Proposition \ref{CP4}, Lemma \ref{CL4} and Lemma \ref{CL5} hold with $$(t,x,u,\gamma,\lambda(t),b(t),x(t),\omega(t),\varepsilon(t),s,y)$$
replaced by $$(\bar{t},\bar{x},\bar{u},\bar{\gamma},\bar{\lambda}(\bar{t}),\bar{b}(\bar{t}),\bar{x}(\bar{t}),\bar{\omega}(\bar{t}),\bar{\varepsilon}(\bar{t}),\bar{s},\bar{y}).$$
\end{lemma}
\begin{remark}
For simplicity, we skip the statement of these similar estimates for $\bar{u}$. We also refer to the equation number of the corresponding inequality for $u(t)$, when we need to use these estimates for $\bar{u}(\bar{t})$.
\end{remark}

{\noindent {\bf Step 2: Closing the bootstrap.}

In this part, we will close the bootstrap argument to show that $\bar{t}^{**}=\bar{t}^*=+\infty.$ This is done through the following steps:
\begin{enumerate}
\item We prove that for $\bar{t}$ large enough, we have $\bar{\omega}(\bar{t})\gg|\bar{b}(\bar{t})|$, which coincides with the formal ODE system \eqref{CFDS} in the (Soliton) region, where we have $\omega(t)$ converges to a positive constant while $b(t)$ converges to $0$, as $t\rightarrow+\infty$. Indeed, if $|\bar{b}(\bar{t})|\gtrsim\bar{\omega}(\bar{t})$ holds for all $\bar{t}\in[0,\bar{t}^{**}]$, we will obtain finite time blow-up if $\bar{b}(0)>0$ or (Exit) behavior if $\bar{b}(0)<0$. Both of them lead to a contradiction.
\item The hardest part of the analysis is to prove that the scaling parameter $\bar{\lambda}$ is bounded from both above and below for all $\bar{t}\in[0,\bar{t}^{**}]$. This is done by proving that%
\footnote{See \eqref{C461} and \eqref{C462} for details.}
$$\bigg(\frac{1}{\bar{\lambda}^2}\bigg)_s+C\bar{\gamma}\bigg(\frac{1}{\bar{\lambda}^2}\bigg)^{1+m/2}\sim \bar{\ell}^*>0.$$
\item The estimates of the rest terms can be done by similar arguments as in the (Blow down) and (Exit) regions.
\end{enumerate}
}

Now we turn to the proof of $\bar{t}^{**}=\bar{t}^*=+\infty$. We first define 
\begin{equation*}
\bar{t}_2^*=
\begin{cases}
0,\;\text{if }|\bar{b}(0)|\leq\frac{1}{100}c_1\bar{\omega}(0),\\
\sup\{0<\bar{t}<\bar{t}^{*}\mid\text{for all $t'\in[0,\bar{t}]$, }|\bar{b}(t')|\geq\frac{1}{100}c_1\bar{\omega}(t')\},\;\text{otherwise.}
\end{cases}
\end{equation*}

Our first observation is that $\bar{t}^*_2<\bar{t}^{*}$. Otherwise, since $\bar{t}^*_2=\bar{t}^*\geq\bar{t}^{**}>0$, we have for all $\bar{t}\in[0,\bar{t}^{**})$, $\bar{b}(\bar{t})\not=0$.

If $\bar{b}(0)>0$, we claim that $\bar{t}^{**}=\bar{t}^*_2=\bar{t}^*=+\infty$. To prove this, we need to improve $\rm(H1)'$, $\rm (H2)'$ and $\rm (H3)'$ on $[0,\bar{t}^{**}]$. Indeed, from the definition of $\bar{t}^*_2$, we have
\begin{equation}
\label{C442}
0<\bar{\omega}(\bar{t})\lesssim \bar{b}(\bar{t}).
\end{equation}
for all $\bar{t}\in[0,\bar{t}^{**})$. Applying this to \eqref{C48}, we have:
$$\frac{(\bar{\lambda}_0)_{\bar{s}}}{\bar{\lambda}_0}\leq-\bar{b}+O(\overline{\mathcal{N}}_{2,\rm loc})+\delta(\kappa)|\bar{b}|.$$
Integrating this from $0$ to $\bar{t}$ using \eqref{C44} and the fact that $\bar{\lambda}(0)=1$, we obtain the almost monotonicity:
\begin{equation}
\label{C443}
\forall0\leq\bar{s}_1<\bar{s}_2\leq \bar{s}^{**},\quad \bar{\lambda}(\bar{s}_2)\leq\frac{10}{9}\bar{\lambda}(\bar{s}_1)\leq \frac{5}{4}.
\end{equation}
On the other hand, we learn from \eqref{C47}, \eqref{C437} and \eqref{C4016}, that for all $\bar{s}\in[0,\bar{s}^{**})$,
\begin{equation}\label{C409}
\frac{99}{100}\bar{\ell}^*-K_1\frac{\bar{b}^2(\bar{s})+\bar{\omega}^2(\bar{s})}{\bar{\lambda}^2(\bar{s})}\leq\frac{\bar{b}(\bar{s})+c_1\bar{\omega}(\bar{s})}{\bar{\lambda}^2(\bar{s})}\leq\frac{101}{100}\bar{\ell}^*+K_1\frac{\bar{b}^2(\bar{s})+\bar{\omega}^2(\bar{s})}{\bar{\lambda}^2(\bar{s})},
\end{equation}
where
$$0<\bar{\ell}^*=\frac{\bar{b}(0)+c_1\bar{\omega}(0)}{\bar{\lambda}^2(0)}=b(t^*_1)+c_1\omega(t^*_1)\lesssim \delta(\alpha_0).$$
Together with \eqref{C442}, we have for all $\bar{s}\in[0,\bar{s}^{**})$,
\begin{equation}
\label{C444}
\frac{\bar{b}(\bar{s})}{\bar{\lambda}^2(\bar{s})}\sim \bar{\ell}^*\lesssim\delta(\alpha_0),\quad \frac{\bar{\omega}(\bar{s})}{\bar{\lambda}^2(\bar{s})}\lesssim \bar{\ell}^*\lesssim\delta(\alpha_0).
\end{equation}
Then from \eqref{C443}, \eqref{C44} and \eqref{C45}, we have for all $\bar{s}\in[0,\bar{s}^{**})$,
\begin{equation}
\label{C445}
\frac{\overline{\mathcal{N}}_2(\bar{s})}{\bar{\lambda}^2(\bar{s})}\lesssim\delta(\alpha_0),\quad\overline{\mathcal{N}}_2(\bar{s})+\bar{\omega}(\bar{s})+|\bar{b}(\bar{s})|\lesssim\bar{\lambda}^2(\bar{s})\bar{\ell}^*+\delta(\alpha_0)\leq\delta(\alpha_0),
\end{equation}
Then, from \eqref{CMC}, \eqref{C438} and following fact $$\bar{u}(0,\bar{x})=Q_{b(t^*_1),\omega(t^*_1)}(\bar{x})+\varepsilon(t^*_1,\bar{x}),$$
we know that
\begin{align}
\label{C446}
\|\bar{\varepsilon}(\bar{s})\|_{L^2}&\lesssim\delta(\alpha_0)+\bigg|\int \bar{u}^2(0)-\int Q^2\bigg|^{\frac{1}{2}}\nonumber\\
&\lesssim \delta(\alpha_0)+\|\varepsilon(t^*_1)\|_{L^2}+|b(t^*_1)|^{\frac{1}{2}}+\omega^{\frac{1}{2}}(t^*_1)\nonumber\\
&\lesssim \delta(\alpha_0).
\end{align} 
Now, from \eqref{CEC} and \eqref{C445}, we have:
$$\bar{\omega}(\bar{s})\|\bar{\varepsilon}_{\bar{y}}(\bar{s})\|_{L^2}^m=\bar{\gamma}\frac{\|\bar{\varepsilon}_{\bar{y}}(\bar{s})\|_{L^2}^m}{\bar{\lambda}^m(\bar{s})}\lesssim \delta(\alpha_0)+\bigg(\bar{\gamma}\frac{\|\bar{\varepsilon}_{\bar{y}}(\bar{s})\|_{L^2}^{m}}{\bar{\lambda}^{m}(\bar{s})}\bigg)^{\frac{m+2}{2}}+|\bar{\gamma}^{\frac{2}{m}}\bar{E}(\bar{u}(0))|^{\frac{m}{2}},$$
where $\bar{E}(\bar{u}(0))$ is the energy of the Cauchy problem \eqref{C441}, i.e.
$$\bar{E}(\bar{u}(0))=\frac{1}{2}\int \bar{u}_{\bar{x}}^2(0)-\frac{1}{6}\int\bar{u}^6(0)+\frac{\bar{\gamma}}{q+1}\int|\bar{u}(0)|^{q+1},$$
Since $$\bar{u}(0,\bar{x})=\lambda^{\frac{1}{2}}(t^*_1)u(t^*_1,\lambda(t^*_1)\bar{x}+x(t^*_1)),$$
from the energy conservation law of \eqref{CCPG} and the condition on the initial data, we have 
$$|\bar{\gamma}^{\frac{2}{m}}\bar{E}(\bar{u}(0))|=\bigg|\gamma^{\frac{2}{m}}\frac{\bar{E}(\bar{u}(0))}{\lambda^2(t^*_1)}\bigg|=|\gamma^{\frac{2}{m}}E(u(t^*_1))|=|\gamma^{\frac{2}{m}}E_0|\lesssim \delta(\alpha_0).$$
Thus, for all $\bar{s}\in[0,\bar{s}^{**})$, we have
$$\bar{\omega}(\bar{s})\|\bar{\varepsilon}_{\bar{y}}(\bar{s})\|_{L^2}^m\lesssim \delta(\alpha_0)+\big(\bar{\omega}(\bar{s})\|\bar{\varepsilon}_{\bar{y}}(\bar{s})\|_{L^2}^m\big)^{1+\frac{m}{2}}.$$
From \eqref{C438} and \eqref{C4016}, we have 
$$\bar{\omega}(0)\|\bar{\varepsilon}_{\bar{y}}(0)\|_{L^2}^m=\omega(s^*_1)\|\varepsilon_{y}(s^*_1)\|_{L^2}^m\lesssim\delta(\alpha_0),$$
then a standard bootstrap argument leads to:
\begin{equation}\label{C4015}
\bar{\omega}(\bar{s})\|\bar{\varepsilon}_{\bar{y}}(\bar{s})\|_{L^2}^m\lesssim \delta(\alpha_0),
\end{equation}
for all $\bar{s}\in[0,\bar{s}^{**})$.

Finally, integrating \eqref{C37}, using \eqref{C44} and \eqref{C443} we obtain:
\begin{align}
\int\varphi_{10}(\bar{y})\bar{\varepsilon}^2(\bar{s},\bar{y})\,d\bar{y}&\leq\frac{\bar{\lambda}^{10}(0)}{\bar{\lambda}^{10}(\bar{s})}\int\varphi_{10}(\bar{y})\bar{\varepsilon}^2(0,\bar{y})\,d\bar{y}+\frac{C}{\bar{\lambda}^{10}(\bar{s})}\int_0^{\bar{s}}\bar{\lambda}^{10}(\overline{\mathcal{N}}_{1,\rm loc}+\bar{b}^2)\nonumber\\
&\leq\frac{1}{\bar{\lambda}^{10}(\bar{s})}\bigg[5+C\bar{\lambda}^{10}(0)\int_0^{\bar{s}}(\overline{\mathcal{N}}_{1,\rm loc}+\bar{b}^2)\bigg]\leq \frac{5+\delta(\kappa)}{\bar{\lambda}^{10}(\bar{s})}.\label{C447}
\end{align}
Combining \eqref{C444}--\eqref{C447}, we conclude that $\bar{t}^{**}=\bar{t}^*$. Since all $H^1$ solution of \eqref{C441} is global in time, we must have $\bar{t}^{**}=\bar{t}^*=+\infty$, provided that $\alpha_0\ll\alpha^*$.
Now we inject \eqref{C444} into \eqref{C48} to obtain:
$$\frac{\bar{\ell}^*}{3}-C\frac{\overline{\mathcal{N}}_{1,\rm loc}}{\bar{\lambda}^2}\leq-(\bar{\lambda}_0)_{\bar{t}}\leq3\bar{\ell}^*+C\frac{\overline{\mathcal{N}}_{1,\rm loc}}{\bar{\lambda}^2}$$
Integrating in time, we have for all $\bar{t}\in[0,+\infty)$
$$0<\bar{\lambda}_0(\bar{t})\leq\bar{\lambda}(0)-\frac{\bar{\ell}^*\bar{t}}{3}+C\int_{0}^{\bar{t}}\frac{\overline{\mathcal{N}}_{1,\rm loc}}{\bar{\lambda}^2}.$$
From \eqref{C443} and \eqref{C44} we have
$$\int_{0}^{\bar{t}}\frac{\overline{\mathcal{N}}_{1,\rm loc}}{\bar{\lambda}^2}=\int_{0}^{\bar{s}}\bar{\lambda}(\tau)\overline{\mathcal{N}}_{1,\rm loc}(\tau)\,d\tau\lesssim \int_{0}^{\bar{s}}\overline{\mathcal{N}}_{1,\rm loc}(\tau)\,d\tau\lesssim \delta(\kappa),$$
which implies that the solution blows up in finite time. This is a contradiction. 

Now we consider the other case $\bar{b}(0)<0$. We claim again that $\bar{t}^*_2=\bar{t}^{**}=\bar{t}^*=+\infty$. It is also done by improving the three  bootstrap assumptions. First, we know from \eqref{C47}, \eqref{C437} and \eqref{C4016} that \eqref{C409} still holds in this case. And the definition of $\bar{t}^*_2$ implies that
\begin{equation}
\label{C448}
0<\bar{\ell}^*\lesssim-\frac{\bar{b}(\bar{s})}{\bar{\lambda}^2(\bar{s})}\sim\frac{\bar{\omega}(\bar{s})}{\bar{\lambda}^2(\bar{s})}.
\end{equation}
Then we apply the fact that $0<\bar{\omega}\lesssim-\bar{b}$ to \eqref{C48} to obtain:
$$\frac{(\bar{\lambda}_0)_{\bar{s}}}{\bar{\lambda}_0}\geq-\frac{1}{2}\bar{b}-O(\overline{\mathcal{N}}_{2,\rm loc}).$$
Integrating in time we have:
\begin{equation}
\label{C449}
\forall0\leq\bar{s}_1<\bar{s}_2\leq \bar{s}^{**},\quad \bar{\lambda}(\bar{s}_2)\geq\frac{9}{10}\bar{\lambda}(\bar{s}_1)\geq \frac{4}{5},
\end{equation}
which yields for all $\bar{s}\in[0,\bar{s}^{**})$,
\begin{equation}
\label{C450}
\bar{\omega}(\bar{s})+\frac{\bar{\omega}(\bar{s})}{\bar{\lambda}^2(\bar{s})}\lesssim \bar{\gamma}\lesssim\delta(\alpha_0).
\end{equation}
From \eqref{C448}, \eqref{C44} and \eqref{C45}, we get 
\begin{equation}
\label{C451}
\overline{\mathcal{N}}_2(\bar{s})+|\bar{b}(\bar{s})|+\frac{\overline{\mathcal{N}}_2(\bar{s})+|\bar{b}(\bar{s})|}{\bar{\lambda}^2(\bar{s})}\lesssim\delta(\alpha_0)
\end{equation}
Using the same argument as we did for \eqref{C446}--\eqref{C447}, we have
\begin{equation}
\label{C452}
\|\bar{\varepsilon}(\bar{s})\|_{L^2}\lesssim \delta(\alpha_0),\quad\bar{\omega}(\bar{s})\|\bar{\varepsilon}_{\bar{y}}(\bar{s})\|_{L^2}^m\lesssim \delta(\alpha_0),\quad\int\varphi_{10}\bar{\varepsilon}^2(\bar{s})\,d\bar{y}\leq 7.
\end{equation}
Combining \eqref{C450}--\eqref{C452}, we conclude that $\bar{t}^{**}=\bar{t}^{*}=+\infty$. But from \eqref{C448}, we have
\begin{equation}\label{C453}
-\bar{b}\sim\bar{\omega}(\bar{s})\gtrsim\bar{\gamma}^{\frac{2}{m+2}}(\bar{\ell}^*)^{\frac{m}{m+2}}>0.
\end{equation}
On the other hand, from \eqref{C46}, we have
$$\int_{0}^{\bar{s}^{**}}\bar{b}^2(s')\,ds'\lesssim 1.$$
The above 2 estimates imply that
$$\bar{s}^{**}=\int_0^{+\infty}\frac{1}{\bar{\lambda}^3(\tau)}\,d\tau<+\infty.$$
which leads to $\bar{\lambda}(\bar{t}_n)\rightarrow+\infty$ as $n\rightarrow+\infty$, for some sequence $\bar{t}_n\rightarrow+\infty$, or equivalently $\lim_{n\rightarrow+\infty}\bar{\omega}(\bar{t}_n)=0.$ This contradicts with \eqref{C453}.

In conclusion, we have proved that $\bar{t}^*_2<\bar{t}^{*}$ with
$$|\bar{b}(\bar{t}^*_2)|\leq \frac{1}{100}c_1\bar{\omega}(\bar{t}^*_2).$$

Let $\bar{s}_2^*=\bar{s}(\bar{t}^*_2)$. Repeating the same procedure as before, we have for all $\bar{s}\in[0,\bar{s}^*_2]$,
\begin{gather}
\bar{\omega}(\bar{s})+|\bar{b}(\bar{s})|+\|\bar{\varepsilon}(\bar{s})\|_{L^2}+\bar{\omega}(\bar{s})\|\bar{\varepsilon}_{\bar{y}}(\bar{s})\|_{L^2}^m+\overline{\mathcal{N}}_2(\bar{s})\lesssim\delta(\alpha_0),\label{C454}\\
\frac{\bar{\omega}(\bar{s})+|\bar{b}(\bar{s})|+\overline{\mathcal{N}}_2(\bar{s})}{\bar{\lambda}^2(\bar{s})}\lesssim \delta(\alpha_0),\\
\int_{\bar{y}>0}\bar{y}^{10}\bar{\varepsilon}^2(\bar{s})\,d\bar{y}\leq 7\bigg(1+\frac{1}{\bar{\lambda}^{10}(\bar{s})}\bigg).\label{C456}
\end{gather}

In particular, we have $\bar{t}^*_2<\bar{t}^{**}\leq \bar{t}^*$. Similarly, we need to improve the three bootstrap assumptions on $[\bar{t}^*_2,\bar{t}^{**})$ to obtain $\bar{t}^{**}=\bar{t}^{*}=+\infty$.

First, it is easy to see that \eqref{C409} holds on $[\bar{s}_2^*,\bar{s}^{**})$. So the definition of $\bar{s}^*_2$ yields%
\footnote{Recall that $c_1=G'(0)>0$, where $G$ is the $C^2$ function introduced in \eqref{CLOB}.}%
\begin{equation}
\label{C457}
\frac{19}{20}\bar{\ell}^*\leq\frac{c_1\bar{\omega}(\bar{s}^*_2)}{\bar{\lambda}^2(\bar{s}^*_2)}\leq\frac{21}{20}\bar{\ell}^*,
\end{equation}
which implies
\begin{equation}
\label{C458}
\frac{9}{10}\bigg(\frac{\bar{\ell}^*}{c_1\bar{\gamma}}\bigg)^{\frac{2}{m+2}}\leq\frac{1}{\bar{\lambda}^2(\bar{s}^*_2)}\leq\frac{11}{10}\bigg(\frac{\bar{\ell}^*}{c_1\bar{\gamma}}\bigg)^{\frac{2}{m+2}}.
\end{equation}

Next, we let 
$$C_1=\frac{99}{100}c_1<c_1,\quad C_2=\frac{101}{100}c_1>c_1,$$
Then, we learn from \eqref{C409} that for all $\bar{s}\in[\bar{s}^*_2,\bar{s}^{**})$,
\begin{align*}
\frac{99}{100}\bar{\ell}^*&\leq \frac{\bar{b}(\bar{s})+C_2\bar{\omega}(\bar{s})}{\bar{\lambda}^2(\bar{s})}-\frac{c_1}{100}\frac{\bar{\omega}(\bar{s})}{\bar{\lambda}^2(\bar{s})}+O\bigg(\frac{\bar{b}^2(\bar{s})+\bar{\omega}^2(\bar{s})}{\bar{\lambda}^2(\bar{s})}\bigg)\\
&\leq\frac{\bar{b}(\bar{s})+C_2\bar{\omega}(\bar{s})}{\bar{\lambda}^2(\bar{s})}-\frac{c_1}{100}\frac{\bar{\omega}(\bar{s})}{\bar{\lambda}^2(\bar{s})}+\delta(\kappa)\Bigg(\bigg|\frac{\bar{b}(\bar{s})+C_2\bar{\omega}(\bar{s})}{\bar{\lambda}^2(\bar{s})}\bigg|+\bigg|\frac{\bar{\omega}(\bar{s})}{\bar{\lambda}^2(\bar{s})}\bigg|\Bigg),
\end{align*}
which implies%
\footnote{Here we use the fact that $|1-(\bar{\lambda}/\bar{\lambda}_0)|\lesssim |\bar{J}_1|\lesssim\delta(\kappa)$.}
\begin{equation}
\label{C459}
\frac{49}{50}\bar{\ell}^*\leq \frac{\bar{b}(\bar{s})+C_2\bar{\omega}_0(\bar{s})}{\bar{\lambda}_0^2(\bar{s})}-\frac{c_1}{200}\frac{\bar{\omega}_0(\bar{s})}{\bar{\lambda}_0^2(\bar{s})},
\end{equation}
where $$\bar{\omega}_0(\bar{s})=\frac{\bar{\gamma}}{\bar{\lambda}^m_0(\bar{s})}.$$
Injecting \eqref{C48} into \eqref{C459}, using \eqref{C45} and the fact that%
\footnote{This is a direct consequence of \eqref{C459}.}
$$\frac{\bar{b}(\bar{s})+C_2\bar{\omega}_0(\bar{s})}{\bar{\lambda}_0^2(\bar{s})}>0,$$
we have
\begin{align}
\frac{49}{50}\bar{\ell}^*\leq&\frac{101}{100}\bigg(-\frac{(\bar{\lambda}_0)_{\bar{s}}}{\bar{\lambda}^3_0}+ \frac{C_2\bar{\omega}_0(\bar{s})}{\bar{\lambda}_0^2(\bar{s})}\bigg)-\frac{1}{100}\bigg( \frac{\bar{b}(\bar{s})+C_2\bar{\omega}_0(\bar{s})}{\bar{\lambda}_0^2(\bar{s})}\bigg)-\frac{c_1}{200}\frac{\bar{\omega}_0(\bar{s})}{\bar{\lambda}_0^2(\bar{s})}\nonumber\\
&+\frac{101K_2}{100}\frac{\overline{\mathcal{N}}_1(\bar{s})}{\bar{\lambda}_0^2(\bar{s})}+\delta(\kappa)\Bigg(\bigg|\frac{\bar{b}(\bar{s})+C_2\bar{\omega}_0(\bar{s})}{\bar{\lambda}_0^2(\bar{s})}\bigg|+\bigg|\frac{\bar{\omega}_0(\bar{s})}{\bar{\lambda}_0^2(\bar{s})}\bigg|\Bigg)\nonumber,\\
\leq&\frac{101}{100}\bigg(-\frac{(\bar{\lambda}_0)_{\bar{s}}}{\bar{\lambda}^3_0}+ \frac{C_2\bar{\omega}_0(\bar{s})}{\bar{\lambda}_0^2(\bar{s})}\bigg)-\frac{1}{100}\bigg( \frac{\bar{b}(\bar{s})+C_2\bar{\omega}_0(\bar{s})}{\bar{\lambda}_0^2(\bar{s})}\bigg)-\frac{c_1}{300}\frac{\bar{\omega}_0(\bar{s})}{\bar{\lambda}_0^2(\bar{s})}\nonumber\\
&+\frac{101K_0K_2}{100}\frac{\big(\overline{\mathcal{N}}_1(0)+\bar{b}^2(0)+\bar{\omega}^2(0)\big)}{\bar{\lambda}_0^2(0)}+\delta(\kappa)\Bigg(\bigg|\frac{\bar{b}(\bar{s})+C_2\bar{\omega}_0(\bar{s})}{\bar{\lambda}_0^2(\bar{s})}\bigg|+\bigg|\frac{\bar{\omega}_0(\bar{s})}{\bar{\lambda}_0^2(\bar{s})}\bigg|\Bigg)\nonumber,\\
\leq&\frac{101}{100}\bigg(-\frac{(\bar{\lambda}_0)_{\bar{s}}}{\bar{\lambda}^3_0}+ \frac{C_2\bar{\omega}_0(\bar{s})}{\bar{\lambda}_0^2(\bar{s})}\bigg)+\frac{51K_0K_2}{50}\frac{\big(\overline{\mathcal{N}}_1(0)+\bar{b}^2(0)+\bar{\omega}^2(0)\big)}{\bar{\lambda}^2(0)}.\label{C460}
\end{align} 
From \eqref{C437} and \eqref{C4016}, we have
$$\bar{\ell}^*=\frac{\bar{b}(0)+c_1\bar{\omega}(0)}{\bar{\lambda}^2(0)}\geq100(K_1+K_0K_2)\frac{\big(\overline{\mathcal{N}}_1(0)+\bar{b}^2(0)+\bar{\omega}^2(0)\big)}{\bar{\lambda}^2(0)}.$$
So \eqref{C460} implies that for all $\bar{s}\in[\bar{s}^*_2,\bar{s}^{**})$,
\begin{equation}
\label{C461}
\frac{1}{2}\bigg(\frac{1}{\bar{\lambda}_0^2}\bigg)_{\bar{s}}+C_2\bar{\gamma}\bigg(\frac{1}{\bar{\lambda}_0^2}\bigg)^{1+\frac{m}{2}}\geq\frac{9}{10}\bar{\ell}^*.
\end{equation}

Similar to \eqref{C459}, we have
\begin{equation}
\label{C4017}
\frac{51}{50}\bar{\ell}^*\geq \frac{\bar{b}(\bar{s})+C_1\bar{\omega}_0(\bar{s})}{\bar{\lambda}_0^2(\bar{s})}
+\frac{c_1}{200}\frac{\bar{\omega}_0(\bar{s})}{\bar{\lambda}_0^2(\bar{s})}-\delta(\kappa)\bigg|\frac{\bar{b}(\bar{s})+C_2\bar{\omega}_0(\bar{s})}{\bar{\lambda}_0^2(\bar{s})}\bigg|,
\end{equation}
which leads to 
\begin{align*}
\frac{51}{50}\bar{\ell}^*\geq&\frac{99}{100}\bigg(-\frac{(\bar{\lambda}_0)_{\bar{s}}}{\bar{\lambda}^3_0}+ \frac{C_1\bar{\omega}_0(\bar{s})}{\bar{\lambda}_0^2(\bar{s})}\bigg)+\frac{1}{100}\bigg( \frac{\bar{b}(\bar{s})+C_1\bar{\omega}_0(\bar{s})}{\bar{\lambda}_0^2(\bar{s})}\bigg)+\frac{c_1}{200}\frac{\bar{\omega}_0(\bar{s})}{\bar{\lambda}_0^2(\bar{s})}\nonumber\\
&+\frac{99K_2}{100}\frac{\overline{\mathcal{N}}_1(\bar{s})}{\bar{\lambda}_0^2(\bar{s})}-\delta(\kappa)\Bigg(\bigg|\frac{\bar{b}(\bar{s})+C_2\bar{\omega}_0(\bar{s})}{\bar{\lambda}_0^2(\bar{s})}\bigg|+\bigg|\frac{\bar{\omega}_0(\bar{s})}{\bar{\lambda}_0^2(\bar{s})}\bigg|\Bigg)\nonumber,\\
\end{align*}
and
\begin{align*}
\frac{51}{50}\bar{\ell}^*\geq&\frac{101}{100}\bigg(-\frac{(\bar{\lambda}_0)_{\bar{s}}}{\bar{\lambda}^3_0}+ \frac{C_1\bar{\omega}_0(\bar{s})}{\bar{\lambda}_0^2(\bar{s})}\bigg)-\frac{1}{100}\bigg( \frac{\bar{b}(\bar{s})+C_1\bar{\omega}_0(\bar{s})}{\bar{\lambda}_0^2(\bar{s})}\bigg)+\frac{c_1}{200}\frac{\bar{\omega}_0(\bar{s})}{\bar{\lambda}_0^2(\bar{s})}\nonumber\\
&+\frac{101K_2}{100}\frac{\overline{\mathcal{N}}_1(\bar{s})}{\bar{\lambda}_0^2(\bar{s})}-\delta(\kappa)\Bigg(\bigg|\frac{\bar{b}(\bar{s})+C_2\bar{\omega}_0(\bar{s})}{\bar{\lambda}_0^2(\bar{s})}\bigg|+\bigg|\frac{\bar{\omega}_0(\bar{s})}{\bar{\lambda}_0^2(\bar{s})}\bigg|\Bigg)\nonumber.\\
\end{align*}
Using the same strategy as \eqref{C460}, and discussing the sign of $\frac{\bar{b}(\bar{s})+C_1\bar{\omega}_0(\bar{s})}{\bar{\lambda}_0^2(\bar{s})}$, we have:
\begin{equation}
\label{C462}
\frac{1}{2}\bigg(\frac{1}{\bar{\lambda}_0^2}\bigg)_{\bar{s}}+C_1\bar{\gamma}\bigg(\frac{1}{\bar{\lambda}_0^2}\bigg)^{1+\frac{m}{2}}\leq\frac{11}{10}\bar{\ell}^*.
\end{equation}
Then we need following basic lemma:
\begin{lemma}\label{CL6}
Let $F$: $[0,x_0)\rightarrow(0,+\infty)$ be a $C^1$ function. Let $\nu>0$, $L>0$ be 2 positive constants. Then we have:
\begin{enumerate}
	\item If for all $x\in[0,x_0)$
	$$F_x+F^{1+\nu}\geq L,$$
	then for all $x\in[0,x_0)$,
	$$F(x)\geq\min(F(0),L^{\frac{1}{1+\nu}}).$$
	\item If for all $x\in[0,x_0)$
	$$F_x+F^{1+\nu}\leq L,$$
	then for all $x\in[0,x_0)$,
	$$F(x)\leq\max(F(0),L^{\frac{1}{1+\nu}}).$$
\end{enumerate}
\end{lemma}
It is easy to prove Lemma \ref{CL6} by standard ODE theory. Now we apply Lemma \ref{CL6} to \eqref{C461} and \eqref{C462} on $[\bar{s}^*_2,\bar{s}^{**})$, using \eqref{C458} to obtain
\begin{equation}
\label{C463}
\frac{90}{101}\bigg(\frac{\bar{\ell}^*}{c_1\bar{\gamma}}\bigg)^{\frac{2}{m+2}}\leq\frac{1}{\bar{\lambda}^2(\bar{s})}\leq\frac{10}{9}\bigg(\frac{\bar{\ell}^*}{c_1\bar{\gamma}}\bigg)^{\frac{2}{m+2}},
\end{equation}
for all $\bar{s}\in[\bar{s}^*_2,\bar{s}^{**})$. This also implies that for all $\bar{s}\in[\bar{s}^*_2,\bar{s}^{**})$,
\begin{equation}
\label{C464}
\bar{\omega}(\bar{s})\sim \bar{\gamma}^{\frac{2}{m+2}}(\bar{\ell}^*)^{\frac{m}{m+2}}\lesssim\delta(\alpha_0),\quad\frac{\bar{\omega}(\bar{s})}{\bar{\lambda}^2(\bar{s})}\sim \bar{\ell}^*\lesssim\delta(\alpha_0).
\end{equation}
From \eqref{C459} and \eqref{C4017}, we have
$$\frac{\bar{b}(\bar{s})+C_2\bar{\omega}_0(\bar{s})}{\bar{\lambda}_0^2(\bar{s})}\geq\frac{49}{50}\bar{\ell}^*,\quad\frac{\bar{b}(\bar{s})+C_1\bar{\omega}_0(\bar{s})}{\bar{\lambda}_0^2(\bar{s})}\leq2\bar{\ell}^*,$$
together with \eqref{C464}, we have
\begin{equation}
\label{C465}
\bigg|\frac{\bar{b}(\bar{s})}{\bar{\lambda}^2(\bar{s})}\bigg|\lesssim \bar{\ell}^*\lesssim\delta(\alpha_0),\quad|\bar{b}(\bar{s})|\lesssim\bar{\gamma}^{\frac{2}{m+2}}(\bar{\ell}^*)^{\frac{m}{m+2}}\lesssim\delta(\alpha_0).
\end{equation}
Again, from the mass conservation law \eqref{CMC}, energy conservation law \eqref{CEC} and the almost monotonicity \eqref{C44}, \eqref{C45}, we have for all $\bar{s}\in[\bar{s}^*_2,\bar{s}^{**})$:
\begin{equation}
\label{C466}
\|\bar{\varepsilon}(\bar{s})\|_{L^2}+\bar{\omega}(\bar{s})\|\bar{\varepsilon}_{\bar{y}}(\bar{s})\|_{L^2}^m+\overline{\mathcal{N}}_2(\bar{s})+\frac{\overline{\mathcal{N}}_2(\bar{s})}{\bar{\lambda}^2(\bar{s})}\lesssim\delta(\alpha_0).
\end{equation}

Finally, we learn from \eqref{C463}, that for all $\bar{s}^*_2\leq\bar{s}_1<\bar{s}_2\leq \bar{s}^{**}$,
$$\frac{1}{4}<\bigg(\frac{81}{101}\bigg)^5\leq \bigg(\frac{\bar{\lambda}(\bar{s}_1)}{\bar{\lambda}(\bar{s}_2)}\bigg)^{10}\leq\bigg(\frac{101}{81}\bigg)^5<4.$$
Then for all $\bar{s}\in[\bar{s}^*_2,\bar{s}^{**})$, we integrate \eqref{C37} from $\bar{s}^*_2$ to $\bar{s}$ to obtain:
\begin{align}
\int\varphi_{10}(\bar{y})\bar{\varepsilon}^2(\bar{s},\bar{y})\,d\bar{y}&\leq\frac{\bar{\lambda}^{10}(\bar{s}^*_2)}{\bar{\lambda}^{10}(\bar{s})}\int\varphi_{10}(\bar{y})\bar{\varepsilon}^2(\bar{s}^*_2,\bar{y})\,d\bar{y}+\frac{C}{\bar{\lambda}^{10}(\bar{s})}\int_{\bar{s_2}}^{\bar{s}}\bar{\lambda}^{10}(\overline{\mathcal{N}}_{1,\rm loc}+\bar{b}^2)\nonumber\\
&\leq\frac{\bar{\lambda}^{10}(\bar{s}^*_2)}{\bar{\lambda}^{10}(\bar{s})}\times7\bigg(1+\frac{1}{\bar{\lambda}^{10}(\bar{s}^*_2)}\bigg)+4C\int_{\bar{s_2}}^{\bar{s}}(\overline{\mathcal{N}}_{1,\rm loc}+\bar{b}^2)\nonumber\\
&\leq 28\bigg(1+\frac{1}{\bar{\lambda}^{10}(\bar{s})}\bigg)+\delta(\kappa)<30\bigg(1+\frac{1}{\bar{\lambda}^{10}(\bar{s})}\bigg).\label{C467}
\end{align}

Combining \eqref{C464}--\eqref{C467}, we have improved $\rm (H1)'$, $\rm (H2)'$ and $\rm (H3)'$, hence $\bar{t}^{**}=\bar{t}^*=+\infty$. This also implies that $t^*=+\infty$. \\

{\noindent {\bf Step 3: Proof of \eqref{C407} and \eqref{C408}.}} 

Now it is sufficient to prove the following
$$|\bar{b}(\bar{t})|+\overline{\mathcal{N}}_2(\bar{t})\rightarrow0,\quad \bar{\lambda}(\bar{t})\rightarrow\bar{\lambda}_{\infty}\in(0,+\infty),$$
as $\bar{t}\rightarrow+\infty$.
First of all, from \eqref{C463}, we know that
$$\bar{s}^{**}=\bar{s}^*=\int_0^{+\infty}\frac{1}{\bar{\lambda}^3(\tau)}\,d\tau=+\infty.$$
Then we claim that $\bar{b}_{\bar{s}}\bar{b}\in L^1((0,+\infty))$. Indeed, from \eqref{C212}, we have
$$\big|\bar{b}\bar{b}_{\bar{s}}+\bar{\omega}_{\bar{s}}G'(\bar{\omega})\bar{b}\big|\lesssim
\bar{b}^2+\int\bar{\varepsilon}^2e^{-\frac{|\bar{y}|}{10}}\in L^1((0,+\infty)).$$
From \eqref{CMS1}, we have:
\begin{align*}
\bar{\omega}_{\bar{s}}G'(\bar{\omega})\bar{b}=m\bar{\omega}G'(\bar{\omega})\bar{b}^2+O\Bigg(\bar{\omega}\bigg|\bar{b}\bigg(\frac{\bar{\lambda}_{\bar{s}}}{\bar{\lambda}}+\bar{b}\bigg)\bigg|\Bigg)=O\bigg(\bar{b}^2+\int\bar{\varepsilon}^2e^{-\frac{|\bar{y}|}{10}}\bigg).
\end{align*}
The above 2 estimates imply that
$$\int_0^{+\infty}|\bar{b}_{\bar{s}}\bar{b}(s')|\,ds'=\int_0^{+\infty}\frac{1}{2}|(\bar{b}^2)_{\bar{s}}|<+\infty.$$
Together with 
$$\int_{0}^{+\infty}\bar{b}^2(\bar{s})\,d\bar{s}<+\infty,$$
we conclude that $\bar{b}(\bar{t})\rightarrow0$, as $\bar{t}\rightarrow+\infty$.
Next, We use \eqref{C212} again to obtain:
$$\big|\bar{b}_{\bar{s}}+\bar{\omega}_{\bar{s}}G'(\bar{\omega})\big|\lesssim
\bar{b}^2+\int\bar{\varepsilon}^2e^{-\frac{|\bar{y}|}{10}}\in L^1((0,+\infty)).$$
Thus, we have
$$\int_0^{+\infty}\big|(\bar{b}+G(\bar{\omega}))_{\bar{s}}(s')\big|\,ds'<+\infty.$$
We then know that $b(\bar{t})+G(\bar{\omega}(\bar{t}))$ has a limit as $\bar{t}\rightarrow+\infty$. Since $\lim_{\bar{t}\rightarrow+\infty}\bar{b}(\bar{t})=0$, we obtain that $G(\bar{\omega}(\bar{t}))$ has a limit as $\bar{t}\rightarrow+\infty$. On the other hand, we have $G'(0)>0$, $\bar{\omega}(\bar{t})\ll1$, so there exists a constant $\bar{\omega}_{\infty}>0$, such that
$$\lim_{\bar{t}\rightarrow+\infty}\bar{\omega}(\bar{t})=\bar{\omega}_{\infty}\sim\bar{\gamma}^{\frac{2}{m+2}}(\bar{\ell}^*)^{\frac{m}{m+2}},$$
or equivalently
$$\lim_{\bar{t}\rightarrow+\infty}\bar{\lambda}(\bar{t})=\bar{\lambda}_{\infty}\sim\bigg(\frac{c_1\bar{\gamma}}{\bar{\ell}^*}\bigg)^{\frac{1}{m+2}}.$$
Let
$$\ell^*=\frac{b(t^*_1)+c_1\omega(t^*_1)}{\lambda^2(t^*_1)}>0.$$
Recall that
$$\bar{\gamma}=\frac{\gamma}{\lambda^m(t^*_1)},\quad \bar{\ell}^*=b(t^*_1)+c_1\omega(t^*_1),\quad \bar{\lambda}(\bar{t})=\frac{\lambda(\lambda^3(t^*_1)\bar{t}+t^*_1)}{\lambda(t^*_1)}.$$
We obtain
\begin{equation}
\label{C468}
\lim_{t\rightarrow+\infty}\lambda(t)=\lambda_{\infty}\sim\bigg(\frac{c_1\gamma}{\ell^*}\bigg)^{\frac{1}{m+2}}.
\end{equation}

Next, the inequality \eqref{C44} implies the existence of a sequence $\bar{s}_n$ such that
$$\overline{\mathcal{N}}_1(\bar{s}_n)\lesssim \int\big(\bar{\varepsilon}^2(\bar{s}_n)+\bar{\varepsilon}^2_{\bar{y}}(\bar{s}_n)\big)\varphi'_{2,B}\rightarrow0,\text{ as }n\rightarrow+\infty, $$
where $\lim_{n\rightarrow+\infty}\bar{s}_n=+\infty$.
Using the monotonicity \eqref{C410}, we have
$$\overline{\mathcal{N}}_1(\bar{s})\rightarrow0,\;\text{as }\bar{s}\rightarrow +\infty.$$
Together with \eqref{C39} and \eqref{C463}, we obtain
$$\overline{\mathcal{N}}_2(\bar{t})\rightarrow0,\;\text{as }\bar{t}\rightarrow +\infty,$$
which implies that
$$\mathcal{N}_2(t)\rightarrow0,\;\text{as }t\rightarrow +\infty,$$

Finally, from \eqref{CMS1}, we have:
$$\lambda^2(t)x_t(t)\sim1,\;\text{as }t\rightarrow +\infty,$$
which after integration implies
$$x(t)\sim\frac{t}{\lambda^2_{\infty}},\;\text{as }t\rightarrow +\infty.$$
We then conclude the proof of \eqref{C407} and \eqref{C408}, hence the proof of the first part of Proposition \ref{CP5}.

\subsubsection*{4. Nonemptiness and stability.}
Now we give the proof of the second part of Proposition \ref{CP5}. 

First, we show that the (Soliton) and (Exit) regimes are stable under small perturbation in $\mathcal{A}_{\alpha_0}$.
From \eqref{C206}, we know that the parameters depend continuously on the initial data, which implies that the cases (Exit) and (Soliton) are both open in $\mathcal{A}_{\alpha_0}$, since the separation condition is an open condition of initial data in $\mathcal{A}_{\alpha_0}$.

Indeed, for all $u_0\in\mathcal{A}_{\alpha_0}$, if the corresponding solution $u(t)$ to \eqref{CCPG} belongs to the (Soliton) regime, we let $t^*_1$ be the separation time introduced in Proposition \ref{CP5}. For all $\tilde{u}_0\in\mathcal{A}_{\alpha_0}$, close enough to $u_0$, we let $\tilde{u}(t)$ be the corresponding solution to \eqref{CCPG}, and $\tilde{b}(t)$, $\tilde{x}(t)$, $\tilde{\lambda}(t)$, $\tilde{\varepsilon}(t)$ be the corresponding geometrical parameters and error term. Then from local theory, we have $\sup_{t\in[0,t_1^*]}\|u(t)-\tilde{u}(t)\|_{H^1}\ll1$, which together with \eqref{C206}, leads to
$$\tilde{b}(t^*_1)+c_1\tilde{\omega}(t^*_1)\geq \frac{999}{1000}C^*\big(\widetilde{\mathcal{N}}_1(t_1^*)+\tilde{b}^2(t^*_1)+\tilde{\omega}^2(t^*_1)\big).$$
So $\tilde{u}(t)$ must belong to the (Soliton) regime. This implies the openness of (Soliton) regime. The openness of the (Exit) regime follows from the same argument.

Next, we claim that there exists initial data in $\mathcal{A}_{\alpha_0}$ such that the corresponding solution to \eqref{CCPG} belongs to the (Soliton) and (Exit) regimes respectively. First, it is easy to check that the traveling wave solution
$$u(t,x)=\mathcal{Q}_{\gamma}(x-t)$$
belongs to the (Soliton) regime. On the other hand, from \eqref{C2006}, we can see, in both the (Soliton) and (Blow down) cases, we have
$$\|u_0\|_{L^2}\geq \|Q\|_{L^2}.$$
Hence, for initial data $u_0\in\mathcal{A}_{\alpha_0}$ with%
\footnote{Since we assume that $\gamma\ll\alpha_0$, such $u_0$ exists in $\mathcal{A}_{\alpha_0}$.}
 $\|u_0\|_{L^2}<\|Q\|_{L^2}$, the corresponding solution must belong to the (Exit) regime. 

Finally, since the sets of initial data which leads to the (Soliton) and (Exit) regime are both open and nonempty in $\mathcal{A}_{\alpha_0}$. Together with the fact that $\mathcal{A}_{\alpha_0}$ is connected, we conclude that there exists $u_0\in\mathcal{A}_{\alpha_0}$, such that the corresponding solution to \eqref{CCPG} belongs to the (Blow down) regime.
\end{proof}

\section{Proof of Theorem \ref{CST}}
In this part we will use the local Cauchy theory of generalized KdV equations developed in \
{KPV} to prove Theorem \ref{CST}.

\subsection{$H^1$ perturbation theory}
First of all, let us introduce the following linear estimates proved by Kenig, Ponce and Vega in \cite{KPV}.
\begin{lemma}[Linear estimates, \cite{KPV}]\label{CL7}The following linear estimates hold:
\begin{enumerate}
	\item For all $u_0\in H^1$,
	\begin{gather}
	\bigg\|\frac{\partial}{\partial_x}W(t)u_0\bigg\|_{L^{\infty}_xL^2_t(\mathbb{R})}+\|W(t)u_0\|_{L_x^5L_t^{10}(\mathbb{R})}\lesssim \|u_0\|_{L^2},\\
	\Big\|D^{\alpha_q}_xD^{\beta_q}_tW(t)u_0\Big\|_{L^p_xL^r_t(I)}\lesssim \|D_x^{s_q}u_0\|_{L^2},\label{C510}
	\end{gather}
	where $q>5$ is the power of the defocusing nonlinear term of \eqref{CCPG}, and 
	\begin{gather*}
	W(t)f=e^{-t\partial_x^3}f,\quad s_q=\frac{1}{2}-\frac{2}{(q-1)},\\
	\alpha_q=\frac{1}{10}-\frac{2}{5(q-1)},\quad\beta_q=\frac{3}{10}-\frac{6}{5(q-1)},\\
	\frac{1}{p}=\frac{2}{5(q-1)}+\frac{1}{10},\quad\frac{1}{r}=\frac{3}{10}-\frac{4}{5(q-1)}.
	\end{gather*}
	\item For all well localized $g$, we have:
	\begin{align}
	&\sup_{t\in I}\bigg\|\frac{\partial}{\partial x}\int_0^tW(t-t')g(\cdot,t')\,dt'\bigg\|_{L^2_x}\lesssim \|g\|_{L^1_xL^2_t(I)},\\
	&\bigg\|\frac{\partial^2}{\partial x^2}\int_0^tW(t-t')g(\cdot,t')\,dt'\bigg\|_{L^{\infty}_x L^2_t(I)}\lesssim \|g\|_{L^1_xL^2_t(I)},\\
	&\bigg\|\int_0^tW(t-t')g(\cdot,t')\,dt'\bigg\|_{L^5_x L^{10}_t(I)}\lesssim \|g\|_{L^{5/4}_xL^{10/9}_t(I)},\\
	&\bigg\|D^{\alpha_q}_xD^{\beta_q}_t\int_0^tW(t-t')g(\cdot,t')\,dt'\bigg\|_{L^p_x L^r_t(I)}\lesssim \|g\|_{L^{p'}_xL^{r'}_t(I)},\label{C500}\\
	&\|g\|_{L^{5(q-1)/4}_xL^{5(q-1)/2}_t}\lesssim \|D^{\alpha_q}_xD^{\beta_q}_tg\|_{L^p_xL^r_t},\label{C504}
	\end{align}
	where
	$$1=\frac{1}{p}+\frac{1}{p'}=\frac{1}{r}+\frac{1}{r'}.$$
\end{enumerate}
\end{lemma}
\begin{proof}[Proof]See Theorem 3.5, Corollary 3.8, Lemma 3.14, Lemma 3.15 and Corollary 3.16 in \cite{KPV} for the proof of (1) and (2).
\end{proof}

Now we define the following norms:
\begin{align*}
\eta^1_{I}(w)=&\|w\|_{L^5_x L^{10}_t(I)},\;\eta^2_I(w)=\|w_x\|_{L^{\infty}_xL^2_t(I)},\;\eta^3_I(w)=\|D^{\alpha_q}_xD^{\beta_q}_tw\|_{L^p_x L^r_t(I)},\\
\Omega_I(w)=&\max_{j=1,2}\big[\eta^j_I(w)+\eta^j_I(w_x)\big]+\eta^3_I(w),\\
\Delta_I(h)=&\|h\|_{L_x^1L^2_t(I)}+\|h_x\|_{L^{5/4}_xL^{10/9}_t(I)}+\|h_x\|_{L_x^1L^2_t(I)}+\|h_{xx}\|_{L^{5/4}_xL^{10/9}_t(I)}\\
&+\|h_{x}\|_{L^{p'}_xL^{r'}_t(I)},
\end{align*}
for all interval $I\subset \mathbb{R}$.

Then we have the following modified $H^1$ perturbation theory:
\begin{proposition}[Modified long time $H^1$ perturbation theory]\label{CP6}
Let $I$ be an interval containing $0$, and $\tilde{u}$ be an $H^1$ solution to 
\begin{equation}
\begin{cases}
\partial_t\tilde{u}+(\partial_{xx}\tilde{u}+\tilde{u}^5-\gamma\tilde{u}|\tilde{u}|^{q-1})_x=e_x,\;(t,x)\in I\times\mathbb{R},\\
\tilde{u}(0,x)=\tilde{u}_0\in H^1.
\end{cases}
\end{equation}
Suppose we have
\begin{align*}
\sup_{t\in I}\|\tilde{u}(t)\|_{H^1}+\Omega_{I}(\tilde{u})\leq M,
\end{align*}
for some $M>0$ independent of $\gamma$. Let $u_0\in H^1$ be such that
{\begin{multline*}
\|u_0-\tilde{u}_0\|_{H^1}+\|e\|_{L_x^1L^2_t(I)}+\|e_x\|_{L^{5/4}_xL^{10/9}_t(I)}+\\
\|e_x\|_{L_x^1L^2_t(I)}+\|e_{xx}\|_{L^{5/4}_xL^{10/9}_t(I)}
+\|e_{x}\|_{L^{p'}_xL^{r'}_t(I)}\leq \epsilon,
\end{multline*}}
for some small $0<\epsilon<\epsilon_0(M)$. Then the solution of \eqref{CCPG} with initial data $u_0$ satisfies:
\begin{equation}\label{C501}
\sup_{t\in I}\|u-\tilde{u}\|_{H^1}+\Omega_I(u-\tilde{u})\leq C(M)\epsilon.
\end{equation}
\end{proposition}
\begin{remark}
The perturbation theory still holds true if we replace $H^1$ by $H^{s}$, with $s\geq\frac{1}{2}-\frac{2}{q-1}>0$.
\end{remark}
\begin{proof}[Proof of Proposition \ref{CP6}]Without loss of generality, we assume that $I=[0,T_0]$ for some $T_0>0$.
	
We first claim the following short time perturbation theory.
\begin{lemma}[Short time perturbation theory]\label{CL8}
Under the same notation of Proposition \ref{CP6}, if we assume in addition that $\Omega_I(\tilde{u})\leq \epsilon_0$, for some small $0<\epsilon_0=\epsilon_1(M)\ll1$. Then there exists a constant $C_0(M)$ which depends only on $M$ such that if $0<\epsilon<\epsilon_0=\epsilon_1(M)$, then
\begin{equation}\label{C502}
\sup_{t\in I}\|u-\tilde{u}\|_{H^1}+\Omega_I(u-\tilde{u})\leq C_0(M)\epsilon.
\end{equation}
\end{lemma}
We leave the proof of Lemma \ref{CL8} in Appendix \ref{APP2}.

Now we turn to the proof of Proposition \ref{CP6}. Let $\epsilon_0=\epsilon_1(2M)>0$ as in Lemma \ref{CL8}. We then choose $0=t_0<t_1<\ldots<t_N=T_0$ (recall that we assume $I=[0,T_0]$), such that for all $j=1,\ldots,N$,
$$\Omega_{[t_{j-1},t_{j}]}(\tilde{u})\leq \epsilon_0.$$
From a standard argument, we know that $N=N(M,\epsilon_0)=N(M)>0$.
We use Lemma \ref{CL8} on each interval $[t_{j-1},t_j]$ to obtain:
$$\sup_{t\in[t_{j-1},t_j]}\|u(t)-\tilde{u}(t)\|_{H^1}+\Omega_{[t_{j-1},t_j]}(\tilde{u})\leq C_0(M)\max(\epsilon,\|u(t_{j-1})-\tilde{u}(t_{j-1})\|_{H^1}).$$
Arguing by induction, using $\|u(0)-\tilde{u}(0)\|_{H^1}\leq \epsilon$, we have for all $j=1,\ldots,N$,
$$\sup_{t\in[t_{j-1},t_j]}\|u(t)-\tilde{u}(t)\|_{H^1}+\Omega_{[t_{j-1},t_j]}(\tilde{u})\leq C(j,M)\epsilon.$$
Summarizing these estimates, we have:
\begin{align*}
\sup_{t\in I}\|u-\tilde{u}\|_{H^1}+\Omega_{I}(\tilde{u})&\leq\sum_{j=1}^N\sup_{t\in[t_{j-1},t_j]}\|u(t)-\tilde{u}(t)\|_{H^1}+\Omega_{[t_{j-1},t_j]}(\tilde{u})\\
&\leq\sum_{j=1}^N C(j,M)\epsilon= C(M)\epsilon,
\end{align*}
which concludes the proof of Proposition \ref{CP6}.
\end{proof}

\subsection{End of the proof of Theorem \ref{CST}}
Now for $0<\gamma\ll\alpha_0\ll\alpha^*\ll1$, we choose a $u_0\in\mathcal{A}_{\alpha_0/2}\subset\mathcal{A}_{\alpha_0}$, such that the corresponding solution $u(t)$ to \eqref{CCP} belongs to the (Blow up) regime with blow up time $T<+\infty$. Let $u_{\gamma}(t)$ be the corresponding solution to \eqref{CCPG}. From \cite[Section 4.4]{MMR1}, we know that there exists a $0<T^*_1<T<+\infty$, geometrical parameters $(\lambda(t),b(t),x(t))$ and an error term $\varepsilon(t)$ such that the following geometrical decomposition holds on $[0,T^*_1]$:
\begin{equation}\label{C51}
u(t,x)=\frac{1}{\lambda(t)^{\frac{1}{2}}}\big[Q_{b(t)}+\varepsilon(t)\big]\bigg(\frac{x-x(t)}{\lambda(t)}\bigg),
\end{equation}
with
\begin{equation}\label{C52}
(\varepsilon,Q)=(\varepsilon,\Lambda Q)=(\varepsilon,y\Lambda Q)=0.
\end{equation}
Moreover, we have for all $t\in[0,T^*_1]$,
\begin{gather}
\mathcal{N}_2(t)+\|\varepsilon(t)\|_{L^2}+|b(t)|+|1-\lambda(t)|\lesssim\delta(\alpha_0),\label{C53}\\
\int_{y>0}y^{10}\varepsilon^2(t,y)\,dy\leq 5.\label{C54}
\end{gather}
and 
\begin{equation}\label{C55}
b(T^*_1)\geq 2C^*\mathcal{N}_1(T^*_1),
\end{equation}
where $C^*$ is the universal constant%
\footnote{The constant $C^*$ chosen here might be different from the one in \cite[(4.23)]{MMR1}. But we can always replace $C^*$ (both constants in this paper and in \cite{MMR1}) by some larger universal constant.}
 introduced in Section 4.2. One may easily check that $C^*$ defined by \eqref{C4010} is independent of $\gamma$.

Next, we claim that there exists a constant $C(u_0,q)>1$ which depends only on $u_0$ and $q$, such that
\begin{equation}\label{C56}
\sup_{t\in[0,T^*_1]}\|u(t)\|_{H^1}+\Omega_{[0,T^*_1]}(u)+\Delta_{[0,T^*_1]}(u|u|^{q-1})\leq C(u_0,q)<+\infty.
\end{equation}

Indeed, from \cite[Corollary 2.11]{KPV} (taking $s=1$), we have
\begin{align*}
\eta^1_{[0,T^*_1]}(u)+\eta^1_{[0,T^*_1]}(u_x)+\eta^2_{[0,T^*_1]}(u)+\eta^2_{[0,T^*_1]}(u_x)\leq C(u_0,q)<+\infty.
\end{align*}
Then, from Duhamel's principle, we have:
$$u(t)=W(t)u_0+\int_0^t\bigg(W(t-t')\partial_x(u^5)\bigg)\,dt'.$$
Together with \eqref{C510}, \eqref{C500} and the Gagliardo-Nirenberg's inequality introduced in \cite[Theorem 2.44]{BCD}, we have
\begin{align*}
\eta^3_{[0,T^*_1]}(u)&\lesssim\|u_xu^4\|_{L^{p'}_xL^{r'}_t}+\|u_0\|_{H^1}\lesssim\|u\|^4_{L^5_xL_t^{10}}\|u_x\|_{L^{p_0}_xL^{r_0}_t}+\|u_0\|_{H^1}\\
&\lesssim \|u\|^4_{L^5_xL_t^{10}}\|D_x^{s_q}u\|_{L^5_xL^{10}_t}^{1-s_q}\|D_x^{s_q}u_x\|_{L^{\infty}_xL^2_t}^{s_q}+\|u_0\|_{H^1}\\
&\lesssim\|u\|^4_{L^5_xL_t^{10}} \Big(\|u\|^{1-s_q}_{L^5_xL^{10}_t}\|u_x\|^{s_q}_{L^5_xL^{10}_t}\Big)^{1-s_q}\Big(\|u_x\|^{1-s_q}_{L^{\infty}_xL^{2}_t}\|u_{xx}\|^{s_q}_{L^{\infty}_xL^{2}_t}\Big)^{s_q}\\
&\quad+\|u_0\|_{H^1}\\
&\lesssim \Big(\eta^1_{[0,T^*_1]}(u)+\eta^1_{[0,T^*_1]}(u_x)+\eta^2_{[0,T^*_1]}(u)+\eta^2_{[0,T^*_1]}(u_x)\Big)^5+\|u_0\|_{H^1},
\end{align*}
where
$$\frac{1}{p_0}=\frac{1}{10}-\frac{2}{5(q-1)},\quad\frac{1}{r_0}=\frac{3}{10}+\frac{4}{5(q-1)}.$$
This implies $\Omega_{[0,T^*_1]}(u)\leq C(u_0,q)<+\infty$.

Next, using the arguments in \cite[Section 6]{KPV}, we obtain
$$
\Delta_{[0,T^*_1]}(u|u|^{q-1})\lesssim\big(\Omega_{[0,T^*_1]}(u)\big)^q\leq C(u_0,q) $$
which yields \eqref{C56}.

Then we apply Proposition \ref{CP6} to $u(t)$ and $u_{\gamma}(t)$, with $e=\gamma u|u|^{q-1}$. Note that from \eqref{C56}, we have 
$$\Delta_{[0,T^*_1]}(e)<\gamma C(u_0,q)\leq \gamma^{\frac{1}{2}}\ll\epsilon_0(C(u_0,q)),$$
provided that $0<\gamma<\gamma(u_0,\alpha_0,\alpha^*,q)\ll1$. Then Proposition \ref{CP6} implies that for all $t\in[0,T^*_1]$, we have
\begin{equation}
\label{C57}
\|u(t)-u_{\gamma}(t)\|_{H^1}\lesssim \gamma^{\frac{1}{2}}.
\end{equation}
Combining with \eqref{C51}--\eqref{C54}, we know that for all $t\in[0,T^*_1]$, $u_{\gamma}(t)\in\mathcal{T}_{\alpha_0,\gamma}$. This allows us to apply Lemma \ref{CL3} to $u_{\gamma}(t)$ on $[0,T^*_1]$, i.e. there exist geometrical parameters $(b_{\gamma}(t),\lambda_{\gamma}(t),x_{\gamma}(t))$ and an error term $\varepsilon_{\gamma}(t)$, such that
$$u_{\gamma}(t,x)=\frac{1}{\lambda_{\gamma}(t)^{\frac{1}{2}}}\big[Q_{b_{\gamma}(t),\omega_{\gamma}(t)}+\varepsilon_{\gamma}(t)\big]\bigg(\frac{x-x_{\gamma}(t)}{\lambda_{\gamma}(t)}\bigg),$$
with
$$\omega_{\gamma}(t)=\frac{\gamma}{\lambda^m_{\gamma}(t)}.$$
Moreover, the orthogonality conditions \eqref{COC} hold. 

Now, from Lemma \ref{CL3} and \eqref{C57} we obtain that for all $t\in[0,T^*_1]$,
\begin{equation}
\label{C58}
\bigg|1-\frac{\lambda(t)}{\lambda_{\gamma}(t)}\bigg|+|b_{\gamma}(t)-b(t)|+|x_{\gamma}(t)-x(t)|+\|\varepsilon_{\gamma}(t)-\varepsilon(t)\|_{H^1}\lesssim \delta(\gamma).
\end{equation} 
Together with \eqref{C53}--\eqref{C55}, we have the following:
\begin{enumerate}
	\item For all $t\in[0,T^*_1]$, \eqref{C438}--\eqref{C440} hold for $u_{\gamma}(t)$.
    \item At the time $t=T^*_1$, there holds:
    $$b_{\gamma}(T^*_1)+c_1\omega_{\gamma}(T^*_{1})\geq C^*(\mathcal{N}_{1,\gamma}(T^*_1)+b^2_{\gamma}(T^*_1)+\omega_{\gamma}^2(T^*_1)),$$
    where
    $$\mathcal{N}_{i,\gamma}(t)=\int(\varepsilon_{\gamma})^2_{y}\psi_B+\varepsilon_{\gamma}^2\varphi_{i,B}.$$
\end{enumerate}

By the argument in Section 4, we know that $u_{\gamma}(t)$ belongs to the (Soliton) regime introduced in Theorem \ref{CMT}. Moreover, we also obtain \eqref{C12} from \eqref{C468}. This concludes the proof of the first part of Theorem \ref{CST}.

The second part of Theorem \ref{CST} follows from exactly the same procedure. Thus, we complete the proof of Theorem \ref{CST}.  

\appendix
{
\section{Proof of the geometrical decomposition}\label{APP1}
In this section, we will give the proof of Lemma \ref{CL3}. We first introduce the following notations: for all suitable $(\tilde{\lambda},\tilde{x},\tilde{b},\tilde{\omega},v)$
\begin{align}\label{A1}
&F_1(\tilde{\lambda},\tilde{x},\tilde{b},\tilde{\omega},v)=(\mathcal{Q}_{\tilde{\omega}},\varepsilon_{\tilde{\lambda},\tilde{x},\tilde{b},\tilde{\omega},v}),\\
&F_2(\tilde{\lambda},\tilde{x},\tilde{b},\tilde{\omega},v)=(\Lambda\mathcal{Q}_{\tilde{\omega}},\varepsilon_{\tilde{\lambda},\tilde{x},\tilde{b},\tilde{\omega},v}),\\
&F_3(\tilde{\lambda},\tilde{x},\tilde{b},\tilde{\omega},v)=(y\Lambda\mathcal{Q}_{\tilde{\omega}},\varepsilon_{\tilde{\lambda},\tilde{x},\tilde{b},\tilde{\omega},v}),
\end{align} 
where
$$\varepsilon_{\tilde{\lambda},\tilde{x},\tilde{b},\tilde{\omega},v}(y)=\tilde{\lambda}^{1/2}v(\tilde{\lambda} y+\tilde{x})-Q_{\tilde{b},\tilde{\omega}}(y).$$

We mention here that we don't assume 
$$\tilde{\omega}=\frac{\gamma}{\tilde{\lambda}^m}.$$

At $(\tilde{\lambda},\tilde{x},\tilde{b},\tilde{\omega},v)=(1,0,0,0,Q)$, we have
\begin{align*}
&\bigg(\frac{\partial F_1}{\partial \tilde{\lambda}},\frac{\partial F_1}{\partial \tilde{x}},\frac{\partial F_1}{\partial \tilde{b}}\bigg)=\Big((\Lambda Q,Q),(Q',Q),(P,Q)\Big),\\
&\bigg(\frac{\partial F_2}{\partial \tilde{\lambda}},\frac{\partial F_2}{\partial \tilde{x}},\frac{\partial F_2}{\partial \tilde{b}}\bigg)=\Big((\Lambda Q,\Lambda Q),(Q',\Lambda Q),(P,\Lambda Q)\Big),\\
&\bigg(\frac{\partial F_3}{\partial \tilde{\lambda}},\frac{\partial F_3}{\partial \tilde{x}},\frac{\partial F_3}{\partial \tilde{b}}\bigg)=\Big((\Lambda Q,y\Lambda Q),(Q',y\Lambda Q),(P,y\Lambda Q)\Big).
\end{align*}
Since $(\Lambda Q,Q)=(Q',Q)=(Q',\Lambda Q)=(\Lambda Q,y\Lambda Q)=0$, $(P,Q)\not=0$, $(\Lambda Q,\Lambda Q)\not=0$, $(Q',y\Lambda Q)\not=0$, it is easy to see that the above Jacobian is not degenerate. Hence, from implicit function theory, we have: there exist unique continues maps 
\begin{equation}
\label{A2}
(\tilde{\lambda}_0,\tilde{x}_0,\tilde{b}_0):\; (\tilde{\omega},v)\mapsto (1-\delta,1+\delta)\times(-\delta,\delta)\times(-\delta,\delta),\quad \delta>0,
\end{equation}
such that for all $\tilde{\omega}\ll1$, $\|v-Q\|_{H^1}\ll1$, there holds
\begin{equation}
\label{A3}
F_{j}(\tilde{\lambda}_0(\tilde{\omega},v),\tilde{x}_0(\tilde{\omega},v),\tilde{b}_0(\tilde{\omega},v),\tilde{\omega},v)=0,\quad j=1,2,3.
\end{equation}
The uniqueness also implies that for all $\tilde{\omega}\ll1$, we have
\begin{equation}
\label{A7}
\tilde{\lambda}_0(\tilde{\omega},\mathcal{Q}_{\tilde{\omega}})\equiv1.
\end{equation}

Next we fix a time $t\in[0,t_0)$ as in Lemma \ref{CL3}. For a solution $u(t)$ to \eqref{CCPG} with 
$$u(t,x)=\frac{1}{\lambda_1^{\frac{1}{2}}(t)}\big[\mathcal{Q}_{\omega_1(t)}+\varepsilon_1(t)\big]\bigg(\frac{x-x_1(t)}{\lambda_1(t)}\bigg),$$
and
$$\omega_1(t)=\frac{\gamma}{\lambda_1^m(t)}\ll1,$$
we let 
$$v(t,\cdot)=\lambda_1^{1/2}(t)u(t,\lambda_1(t)\cdot+x_1(t))=\mathcal{Q}_{\omega_1(t)}(\cdot)+\varepsilon_1(t,\cdot).$$
Then we have $\|v(t,\cdot)-Q(\cdot)\|_{H^1}\ll1$. 

We claim that there exists a $\underline{\lambda}(t)>0$, such that
\begin{equation}
\label{A4}
\lambda_1(t)\tilde{\lambda}_0\bigg(\frac{\gamma}{\underline{\lambda}^m(t)},v(t)\bigg)=\underline{\lambda}(t), \quad \frac{\gamma}{\underline{\lambda}^m(t)}\ll1.
\end{equation}

This is easy to be verified by implicit function theory. We let
$$M(\underline{\lambda},v)=\underline{\lambda}-\lambda_1(t)\tilde{\lambda}_0\bigg(\frac{\gamma}{\underline{\lambda}^m},v\bigg).$$
Then we have
\begin{align*}
&M(\lambda_1(t),\mathcal{Q}_{\omega_1(t)})=0,\\
&\frac{\partial M}{\partial \underline{\lambda}}\bigg|_{(\underline{\lambda},v)=(\lambda_1(t),\mathcal{Q}_{\omega_1(t)})}=1+m\omega_1(t)\frac{\partial \tilde{\lambda}_0}{\partial\tilde{\omega}}(\omega_1(t),\mathcal{Q}_{\omega_1(t)})>0,
\end{align*}
which implies \eqref{A4} immediately.

Applying \eqref{A2}--\eqref{A4} to $v(t)$, we have
\begin{align}
&F_{j}\Big(\tilde{\lambda}_0\big(\omega(t),v(t)\big),\tilde{x}_0\big(\omega(t),v(t)\big),\tilde{b}_0\big(\omega(t),v(t)\big),\omega(t),v(t)\Big)=0,\; j=1,2,3,\label{A5}\\
&\lambda_1(t)\tilde{\lambda}_0\big(\omega(t),v(t)\big)=\underline{\lambda}(t),\label{A6}
\end{align}
where
$$\omega(t)=\frac{\gamma}{\underline{\lambda}^m(t)}.$$

Now, we let 
\begin{gather}
\lambda(t)=\underline{\lambda}(t),\; b(t)=\tilde{b}_0\big(\omega(t),v(t)\big),\; x(t)=x_1(t)+\lambda_1(t)\tilde{x}_0(\omega(t),v(t)),\\
\omega(t)=\frac{\gamma}{\lambda^m(t)},\;\varepsilon(t,y)=\lambda^{1/2}(t)u\big(t,\lambda(t)\cdot+x(t)\big)-Q_{b(t),\omega(t)}.
\end{gather}

We claim that this $(\lambda(t),x(t),b(t))$ satisfies the orthogonality conditions \eqref{COC}. Indeed, from \eqref{A4}--\eqref{A6}, we have
\begin{align*}
0=&F_1\Big(\tilde{\lambda}_0\big(\omega(t),v(t)\big),\tilde{x}_0\big(\omega(t),v(t)\big),\tilde{b}_0\big(\omega(t),v(t)\big),\omega(t),v(t)\Big)\\
=&\bigg(\mathcal{Q}_{\omega(t)}(\cdot),\;\tilde{\lambda}_0^{1/2}(\omega(t),v(t))v\Big(t,\tilde{\lambda}_0(\omega(t),v(t))\cdot+\tilde{x}_0(\omega(t),v(t))\Big)-Q_{b(t),\omega(t)}(\cdot)\bigg)\\
=&\bigg(\mathcal{Q}_{\omega(t)}(\cdot),\;\big[\lambda_1(t)\tilde{\lambda}_0(\omega(t),v(t))\big]^{1/2}\times\\
&\quad u\Big(t,\lambda_1(t)\big[\tilde{\lambda}_0(\omega(t),v(t))\cdot+\tilde{x}_0(\omega(t),v(t))\big]+x_1(t)\Big)-Q_{b(t),\omega(t)}(\cdot)\bigg)\\
=&\big(\mathcal{Q}_{\omega(t)}(\cdot),\;\lambda^{1/2}(t)u(t,\lambda(t)\cdot+x(t))-Q_{b(t),\omega(t)}(\cdot)\big)=(\mathcal{Q}_{\omega(t)},\varepsilon(t)).
\end{align*}
The other two orthogonality conditions can be verified similarly.

Finally, since the maps 
$$(\tilde{\lambda}_0,\tilde{x}_0,\tilde{b}_0):\; (\tilde{\omega},v)\mapsto (1-\delta,1+\delta)\times(-\delta,\delta)\times(-\delta,\delta)$$
are continuous, the remaining part of Lemma \ref{CL3} follows immediately.
}

\section{Proof of Lemma \ref{CL8}}\label{APP2}
In this section, we give the proof of the modified short time perturbation theory, i.e. Lemma \ref{CL8}.

First, we let $v(t,x)=u(t,x)-\tilde{u}(t,x)$, $S(t)=\Omega_{[0,t]}(v)$. We claim the following estimate holds true for all $t\in I$:
\begin{equation}
\label{C503}
S(t)\lesssim_{M} \epsilon+S(t)\big(S(t)^4+S(t)^{q-1}+\Omega_I(\tilde{u})^4+\Omega_I(\tilde{u})^{q-1}\big).
\end{equation}
Since $S(0)=0$ and $\Omega_I(\tilde{u})\leq \epsilon_0$, we know that Lemma \ref{CL8} follows from a standard bootstrap argument. Now it only remains to prove \eqref{C503}.

First, by Duhamel's principle, we have
\begin{align*}
v(t)&=W(t)(\tilde{u}_0-u_0)+\int_{0}^t\bigg(W(t-t')\partial_x\big[\tilde{u}^5-\gamma\tilde{u}|\tilde{u}|^{q-1}\\
&\qquad-(\tilde{u}+v)^5+\gamma(\tilde{u}+v)|\tilde{u}+v|^{q-1}-e\big]\bigg)\,dt'\\
&=v_L(t)+v_N(t).
\end{align*}

For the linear part $v_L$, from Lemma 5.1, we have:
\begin{equation}\label{C505}
\Omega_{[0,t]}(v_L)+\sup_{t'\in[0,t]}\|v_L\|_{H^1}\lesssim \|\tilde{u}_0-u_0\|_{H^1}\lesssim \epsilon.
\end{equation}

Now, for the nonlinear part $v_N$, we use Lemma \ref{CL7} to estimate:
\begin{align*}
\eta^1_{[0,t]}(v_N)&\lesssim\|e_x\|_{L^{5/4}_xL^{10/9}_t([0,t])}+\|(v+\tilde{u})^4(v+\tilde{u})_x-\tilde{u}^4\tilde{u}_x\|_{L^{5/4}_xL^{10/9}_t([0,t])}\\
&\quad+\||v+\tilde{u}|^{q-1}(v+\tilde{u})_x-|\tilde{u}|^{q-1}\tilde{u}_x\|_{L^{5/4}_xL^{10/9}_t([0,t])}.
\end{align*}
By H\"{o}lder's inequality, we have:
\begin{align*}
&\|(v+\tilde{u})^4(v+\tilde{u})_x-\tilde{u}^4\tilde{u}_x\|_{L^{5/4}_xL^{10/9}_t([0,t])}\\
&\lesssim\big\|\big((v+\tilde{u})^4-\tilde{u}^4\big)\tilde{u}_x\big\|_{L^{5/4}_xL^{10/9}_t}+\|(v+\tilde{u})^4v_x\|_{L^{5/4}_xL^{10/9}_t}\\
&\lesssim\big(\|\tilde{u}\|^3_{L^5_xL^{10}_t}+\|v\|^3_{L^5_xL^{10}_t}\big)\|v\|_{L^5_xL^{10}_t}\|\tilde{u}_x\|_{L^{\infty}_xL^{2}_t}\\
&\quad+\|v\|^4_{L^5_xL^{10}_t}\big(\|v_x\|_{L^{\infty}_xL^{2}_t}+\|\tilde{u}_x\|_{L^{\infty}_xL^{2}_t}\big)\\
&\lesssim S(t)\big(S(t)^4+S(t)^{q-1}+\Omega_I(\tilde{u})^4+\Omega_I(\tilde{u})^{q-1}\big),
\end{align*}
and
\begin{align*}
&\||v+\tilde{u}|^{q-1}(v+\tilde{u})_x-|\tilde{u}|^{q-1}\tilde{u}_x\|_{L^{5/4}_xL^{10/9}_t([0,t])}\\
&\lesssim\big\|\big(|v+\tilde{u}|^{q-1}-|\tilde{u}|^{q-1}\big)\tilde{u}_x\big\|_{L^{5/4}_xL^{10/9}_t}+\||v+\tilde{u}|^{q-1}v_x\|_{L^{5/4}_xL^{10/9}_t}\\
&\lesssim\big(\|\tilde{u}\|^{q-2}_{L^{5(q-1)/4}_xL^{5(q-1)/2}_t}+\|v\|^{q-2}_{L^{5(q-1)/4}_xL^{5(q-1)/2}_t}\big)\|v\|_{L^{5(q-1)/4}_xL^{5(q-1)/2}_t}\|\tilde{u}_x\|_{L^{\infty}_xL^{2}_t}\\
&\quad+\|v\|^{q-1}_{L^{5(q-1)/4}_xL^{5(q-1)/2}_t}\big(\|v_x\|_{L^{\infty}_xL^{2}_t}+\|\tilde{u}_x\|_{L^{\infty}_xL^{2}_t}\big)\\
&\lesssim\big(\|D^{\alpha_q}_xD^{\beta_q}_t\tilde{u}\|^{q-2}_{L^p_x L^r_t}+\|D^{\alpha_q}_xD^{\beta_q}_tv\|^{q-2}_{L^p_x L^r_t}\big)\|D^{\alpha_q}_xD^{\beta_q}_tv\|_{L^p_x L^r_t}\|\tilde{u}_x\|_{L^{\infty}_xL^{2}_t}\\
&\quad+\|D^{\alpha_q}_xD^{\beta_q}_tv\|^{q-1}_{L^p_x L^r_t}\big(\|v_x\|_{L^{\infty}_xL^{2}_t}+\|\tilde{u}_x\|_{L^{\infty}_xL^{2}_t}\big)\\
&\lesssim S(t)\big(S(t)^4+S(t)^{q-1}+\Omega_I(\tilde{u})^4+\Omega_I(\tilde{u})^{q-1}\big),
\end{align*}
where we use \eqref{C504} for the last but two inequality. The above two estimates imply that
\begin{equation}
\label{C506}
\eta^1_{[0,t]}(v_N)\lesssim S(t)\big(S(t)^4+S(t)^{q-1}+\Omega_I(\tilde{u})^4+\Omega_I(\tilde{u})^{q-1}\big)+\epsilon.
\end{equation}
Similarly, we have
\begin{align*}
\eta^1_{[0,t]}(\partial_xv_N)&\lesssim\|e_{xx}\|_{L^{5/4}_xL^{10/9}_t([0,t])}+\big\|\big((v+\tilde{u})^5-\tilde{u}^5\big)_{xx}\big\|_{L^{5/4}_xL^{10/9}_t([0,t])}\\
&\quad+\big\|\big((v+\tilde{u})|v+\tilde{u}|^{q-1}-\tilde{u}|\tilde{u}|^{q-1}\big)_{xx}\big\|_{L^{5/4}_xL^{10/9}_t([0,t])}.
\end{align*}
By H\"{o}lder's inequality again, we have:
\begin{align*}
&\big\|\big((v+\tilde{u})^5-\tilde{u}^5\big)_{xx}\big\|_{L^{5/4}_xL^{10/9}_t([0,t])}\\
&\lesssim \|(v+\tilde{u})^4v_{xx}\|_{L^{5/4}_xL^{10/9}_t}+\big\|(v+\tilde{u})^3(v_{x}+2\tilde{u}_x)v_x\big\|_{L^{5/4}_xL^{10/9}_t}\\
&\lesssim\big(\|\tilde{u}\|^{4}_{L^{5}_xL^{10}_t}+\|v\|^{4}_{L^{5}_xL^{10}_t}\big)\|v_{xx}\|_{L^{\infty}_xL^{2}_t}\\
&\quad+\|v\|^{3}_{L^{5}_xL^{10}_t}\|v_x\|_{L^{\infty}_xL^{2}_t}\big(\|v_x\|_{L^{5}_xL^{10}_t}+\|\tilde{u}_x\|_{L^{5}_xL^{10}_t}\big)\\
&\lesssim S(t)\big(S(t)^4+S(t)^{q-1}+\Omega_I(\tilde{u})^4+\Omega_I(\tilde{u})^{q-1}\big),
\end{align*}
and
\begin{align*}
&\big\|\big((v+\tilde{u})|v+\tilde{u}|^{q-1}-\tilde{u}|\tilde{u}|^{q-1}\big)_{xx}\big\|_{L^{5/4}_xL^{10/9}_t([0,t])}\\
&\lesssim\||v+\tilde{u}|^{q-1}v_{xx}\|_{L^{5/4}_xL^{10/9}_t}+\big\||v+\tilde{u}|^{q-2}(v_{x}+2\tilde{u}_x)v_x\big\|_{L^{5/4}_xL^{10/9}_t}\\
&\lesssim \big(\|\tilde{u}\|^{q-1}_{L^{5(q-1)/4}_xL^{5(q-1)/2}_t}+\|v\|^{q-1}_{L^{5(q-1)/4}_xL^{5(q-1)/2}_t}\big)\|v_{xx}\|_{L^{\infty}_xL^{2}_t}\\
&\quad+\|v\|^{q-2}_{L^{5(q-1)/4}_xL^{5(q-1)/2}_t}\|v_x\|_{L^{\infty}_xL^{2}_t}\big(\|v_x+2\tilde{u}_x\|_{L^{5(q-1)/4}_xL^{5(q-1)/2}_t}\big)\\
&\lesssim S(t)\big(S(t)^4+S(t)^{q-1}+\Omega_I(\tilde{u})^4+\Omega_I(\tilde{u})^{q-1}\big).
\end{align*}
Collecting these estimates, we have:
\begin{equation}
\label{C507}
\eta^1_{[0,t]}(\partial_xv_N)\lesssim S(t)\big(S(t)^4+S(t)^{q-1}+\Omega_I(\tilde{u})^4+\Omega_I(\tilde{u})^{q-1}\big)+\epsilon.
\end{equation}

Next, using similar strategy, we have:
\begin{align}
\eta^2_{[0,t]}(v_N)&\lesssim\|e\|_{L^1_xL^2_t([0,t])}+\|(v+\tilde{u})^5-\tilde{u}^5\|_{L^1_xL^2_t([0,t])}\nonumber\\
&\quad+\big\|(v+\tilde{u})|v+\tilde{u}|^{q-1}-\tilde{u}|\tilde{u}|^{q-1}\big\|_{L^1_xL^2_t([0,t])}\nonumber\\
&\lesssim\epsilon+\big(\|\tilde{u}\|^4_{L^5_xL^{10}_t}+\|v\|^4_{L^5_xL^{10}_t}\big)\|v\|_{L^5_xL^{10}_t}\nonumber\\
&\quad+\big(\|v\|^{q-1}_{L^{5(q-1)/4}_xL^{5(q-1)/2}_t}+\|\tilde{u}\|^{q-1}_{L^{5(q-1)/4}_xL^{5(q-1)/2}_t}\big)\|v\|_{L^5_xL^{10}_t}\nonumber\\
&\lesssim S(t)\big(S(t)^4+S(t)^{q-1}+\Omega_I(\tilde{u})^4+\Omega_I(\tilde{u})^{q-1}\big)+\epsilon,
\end{align}
and
\begin{align}
\eta^2_{[0,t]}(\partial_xv_N)&\lesssim\|e_x\|_{L^1_xL^2_t([0,t])}+
\big\|\big((v+\tilde{u})^5-\tilde{u}^5\big)_x\big\|_{L^1_xL^2_t([0,t])}\nonumber\\
&\quad+\big\|\big((v+\tilde{u})|v+\tilde{u}|^{q-1}-\tilde{u}|\tilde{u}|^{q-1}\big)_x\big\|_{L^1_xL^2_t([0,t])}\nonumber\\
&\lesssim\epsilon+\big\|\big((v+\tilde{u})^4-\tilde{u}^4\big)\tilde{u}_x\big\|_{L^1_xL^2_t([0,t])}+\big\|(v+\tilde{u})^4v_x\big\|_{L^1_xL^2_t([0,t])}\nonumber\\
&\quad+\big\|\big(|v+\tilde{u}|^{q-1}-|\tilde{u}|^{q-1}\big)\tilde{u}_x\big\|_{L^1_xL^2_t([0,t])}+\big\||v+\tilde{u}|^{q-1}v_x\big\|_{L^1_xL^2_t([0,t])}\nonumber\\
&\lesssim S(t)\big(S(t)^4+S(t)^{q-1}+\Omega_I(\tilde{u})^4+\Omega_I(\tilde{u})^{q-1}\big)+\epsilon.\label{C508}
\end{align}

Finally, we need to estimate $\eta^3_{[0,t]}(v_N)$. From Lemma 5.1, we have:
\begin{align*}
\eta^3_{[0,t]}(v_N)&\lesssim\|e_x\|_{L^{p'}_xL^{r'}_t([0,t])}+ \|\big((v+\tilde{u})^5-\tilde{u}^5\big)_x\big\|_{L^{p'}_xL^{r'}_t([0,t])}\\
&\quad+\big\|\big((v+\tilde{u})|v+\tilde{u}|^{q-1}-\tilde{u}|\tilde{u}|^{q-1}\big)_x\big\|_{L^{p'}_xL^{r'}_t([0,t])}\\
&\lesssim\epsilon+\big\|\big((v+\tilde{u})^4-\tilde{u}^4\big)\tilde{u}_x\big\|_{L^{p'}_xL^{r'}_t}+\big\|(v+\tilde{u})^4v_x\big\|_{L^{p'}_xL^{r'}_t}\\
&\quad+\big\|\big(|v+\tilde{u}|^{q-1}-|\tilde{u}|^{q-1}\big)\tilde{u}_x\big\|_{L^{p'}_xL^{r'}_t}+\big\||v+\tilde{u}|^{q-1}v_x\big\|_{L^{p'}_xL^{r'}_t}.
\end{align*}
By similar technique we use for \eqref{C508}, we have:
\begin{align*}
&\big\|\big((v+\tilde{u})^4-\tilde{u}^4\big)\tilde{u}_x\big\|_{L^{p'}_xL^{r'}_t}+\big\|(v+\tilde{u})^4v_x\big\|_{L^{p'}_xL^{r'}_t}\\
&\lesssim\big\|(v+\tilde{u})^4-\tilde{u}^4\big\|_{L^{5/4}_xL^{5/2}_t}\|\tilde{u}_x\|_{L^{p_0}_xL^{r_0}_t}+\|(v+\tilde{u})^4\|_{L^{5/4}_xL^{5/2}_t}\|v_x\|_{L^{p_0}_xL^{r_0}_t}\\
&\lesssim \|v_x\|_{L^{p_0}_xL^{r_0}_t}\big(S(t)^4+\Omega_{[0,t]}(\tilde{u})^4\big)+\|\tilde{u}_x\|_{L^{p_0}_xL^{r_0}_t}S(t)\big(S(t)^3+\Omega_{[0,t]}(\tilde{u})^3\big),
\end{align*}
and
\begin{align*}
&\big\|\big(|v+\tilde{u}|^{q-1}-|\tilde{u}|^{q-1}\big)\tilde{u}_x\big\|_{L^{p'}_xL^{r'}_t}+\big\||v+\tilde{u}|^{q-1}v_x\big\|_{L^{p'}_xL^{r'}_t}\\
&\lesssim\|v_x\|_{L^{p_0}_xL^{r_0}_t}\big(S(t)^{q-1}+\Omega_{[0,t]}(\tilde{u})^{q-1}\big)+\|\tilde{u}_x\|_{L^{p_0}_xL^{r_0}_t}S(t)\big(S(t)^{q-2}+\Omega_{[0,t]}(\tilde{u})^{q-2}\big),
\end{align*}
where
$$\frac{1}{p_0}=\frac{1}{10}-\frac{2}{5(q-1)},\quad\frac{1}{r_0}=\frac{3}{10}+\frac{4}{5(q-1)}.$$
By the Gagliardo-Nirenberg's inequality introduced in \cite[Theorem 2.44]{BCD}, we have:
\begin{align*}
\|v_x\|_{L^{p_0}_xL^{r_0}_t}&\lesssim \|D_x^{s_q}v\|_{L^5_xL^{10}_t}^{1-s_q}\|D_x^{s_q}v_x\|_{L^{\infty}_xL^2_t}^{s_q}\\
&\lesssim \Big(\|v\|^{1-s_q}_{L^5_xL^{10}_t}\|v_x\|^{s_q}_{L^5_xL^{10}_t}\Big)^{1-s_q}\Big(\|v_x\|^{1-s_q}_{L^{\infty}_xL^{2}_t}\|v_{xx}\|^{s_q}_{L^{\infty}_xL^{2}_t}\Big)^{s_q}\\
&\lesssim S(t),
\end{align*}
Similarly, we have:
$$\|\tilde{u}_x\|_{L^{p_0}_xL^{r_0}_t}\lesssim\Omega_{[0,t]}(\tilde{u}),$$
hence
\begin{equation}
\label{C509}
\eta^3_{[0,t]}(v_N)\lesssim S(t)\big(S(t)^4+S(t)^{q-1}+\Omega_I(\tilde{u})^4+\Omega_I(\tilde{u})^{q-1}\big)+\epsilon.
\end{equation}
Combining \eqref{C505}--\eqref{C509}, we conclude the proof of \eqref{C503}, hence the proof of Lemma \ref{CL8}.

\bibliographystyle{amsplain}
\bibliography{ref}

\providecommand{\bysame}{\leavevmode\hbox to3em{\hrulefill}\thinspace}
\providecommand{\MR}{\relax\ifhmode\unskip\space\fi MR }
\providecommand{\MRhref}[2]{%
  \href{http://www.ams.org/mathscinet-getitem?mr=#1}{#2}
}
\providecommand{\href}[2]{#2}
\begin{thebibliography}{10}

\bibitem{BCD}
H.~Bahouri, J.~Y. Chemin, and R.~Danchin, \emph{Fourier analysis and nonlinear
  partial differential equations}, vol. 343, Springer Science \& Business
  Media, 2011.

\bibitem{BL}
H.~Berestycki and P.~L. Lions, \emph{Nonlinear scalar field equations, {I}
  existence of a ground state}, Arch. Ration. Mech. Anal. \textbf{82} (1983),
  no.~4, 313--345.

\bibitem{CMR}
C.~Collot, F.~Merle, and P.~Rapha{\"e}l, \emph{Dynamics near the ground state
  for the energy critical nonlinear heat equation in large dimensions}, to
  appear in Comm. Math. Phys. (2016).

\bibitem{DK}
R.~Donninger and J.~Krieger, \emph{Nonscattering solutions and blowup at
  infinity for the critical wave equation}, Math. Ann. \textbf{357} (2013),
  no.~1, 89--163.

\bibitem{DM1}
T.~Duyckaerts and F.~Merle, \emph{Dynamics of threshold solutions for
  energy-critical wave equation}, Int. Math. Res. Pap. IMRP \textbf{2008}
  (2008), rpn002.

\bibitem{DM2}
\bysame, \emph{Dynamic of threshold solutions for energy-critical {N}{L}{S}},
  Geom. Funct. Anal. \textbf{18} (2009), no.~6, 1787--1840.

\bibitem{Fibich}
G.~Fibich, \emph{The nonlinear {S}chr{\"o}dinger equation}, vol. 192, Springer,
  2015.

\bibitem{FMR}
G.~Fibich, F.~Merle, and P.~Rapha{\"e}l, \emph{Proof of a spectral property
  related to the singularity formation for the ${L}^2$ critical nonlinear
  {S}chr{\"o}dinger equation}, Phys. D \textbf{220} (2006), no.~1, 1--13.

\bibitem{GA}
K.~Glasner and J.~Allen-Flowers, \emph{Nonlinearity saturation as a singular
  perturbation of the nonlinear {S}chro\"{o}dinger equation}, SIAM J. Appl.
  Math. \textbf{76} (2016), no.~2, 525--550.

\bibitem{JR}
C.~Josserand and S.~Rica, \emph{Coalescence and droplets in the subcritical
  nonlinear {S}chr{\"o}dinger equation}, Phys. Rev. Lett. \textbf{78} (1997),
  no.~7, 1215.

\bibitem{Ka}
T.~Kato, \emph{On the {C}auchy problem for the (generalized) {K}orteweg-de
  {V}ries equation}, Stud. Appl. Math. \textbf{8} (1983), 93--128.

\bibitem{KM2}
C.~E. Kenig and F.~Merle, \emph{Global well-posedness, scattering and blow-up
  for the energy-critical, focusing, non-linear {S}chr{\"o}dinger equation in
  the radial case}, Invent. Math. \textbf{166} (2006), no.~3, 645--675.

\bibitem{KM1}
\bysame, \emph{Global well-posedness, scattering and blow-up for the
  energy-critical focusing non-linear wave equation}, Acta Math. \textbf{201}
  (2008), no.~2, 147--212.

\bibitem{KPV}
C.~E. Kenig, G.~Ponce, and L.~Vega, \emph{Well-posedness and scattering results
  for the generalized {K}orteweg-de {V}ries equation via the contraction
  principle}, Comm. Pure Appl. Math. \textbf{46} (1993), no.~4, 527--620.

\bibitem{KKSV}
R.~Killip, S.~Kwon, S.~Shao, and M.~Visan, \emph{On the mass-critical
  generalized {K}d{V} equation}, Discrete Contin. Dyn. Syst. \textbf{32}
  (2012), no.~1, 191--221.

\bibitem{K}
H.~Koch, \emph{Self-similar solutions to super-critical g{K}d{V}}, Nonlinearity
  \textbf{28} (2015), no.~3, 545--575.

\bibitem{KNS}
J.~Krieger, K.~Nakanishi, and W.~Schlag, \emph{Global dynamics away from the
  ground state for the energy-critical nonlinear wave equation}, Amer. J. Math.
  \textbf{135} (2013), no.~4, 935--965.

\bibitem{KNS1}
\bysame, \emph{Threshold phenomenon for the quintic wave equation in three
  dimensions}, Comm. Math. Phys. \textbf{327} (2014), no.~1, 309--332.

\bibitem{L1}
Y.~Lan, \emph{Stable self-similar blow-up dynamics for slightly
  ${L}^2$-supercritical generalized {K}d{V} equations}, Comm. Math. Phys.
  \textbf{345} (2016), no.~1, 223--269.

\bibitem{Mal}
V.~M. Malkin, \emph{On the analytical theory for stationary self-focusing of
  radiation}, Phys. D \textbf{64} (1993), no.~1, 251--266.

\bibitem{MM2}
Y.~Martel and F.~Merle, \emph{Instability of solitons for the critical
  generalized {K}orteweg-de {V}ries equation}, Geom. Funct. Anal. \textbf{11}
  (2001), no.~1, 74--123.

\bibitem{MM4}
\bysame, \emph{Blow up in finite time and dynamics of blow up solutions for the
  ${L}^2$--critical generalized {K}d{V} equation}, J. Amer. Math. Soc.
  \textbf{15} (2002), no.~3, 617--664.

\bibitem{MM5}
\bysame, \emph{Nonexistence of blow-up solution with minimal ${L}^2$-mass for
  the critical g{K}d{V} equation}, Duke Math. J. \textbf{115} (2002), no.~2,
  385--408.

\bibitem{MM3}
\bysame, \emph{Stability of blow-up profile and lower bounds for blow-up rate
  for the critical generalized {K}d{V} equation}, Ann. of Math. \textbf{155}
  (2002), no.~1, 235--280.

\bibitem{MMNR}
Y.~Martel, F.~Merle, K.~Nakanishi, and P.~Rapha{\"e}l, \emph{Codimension one
  threshold manifold for the critical g{K}d{V} equation}, Comm. Math. Phys.
  \textbf{342} (2016), no.~3, 1075--1106.

\bibitem{MMR1}
Y.~Martel, F.~Merle, and P.~Rapha{\"e}l, \emph{Blow up for the critical
  generalized {K}orteweg-de {V}ries equation {I}: {D}ynamics near the soliton},
  Acta Math. \textbf{212} (2014), no.~1, 59--140.

\bibitem{MMR2}
\bysame, \emph{Blow up for the critical g{K}d{V} equation {I}{I}: minimal mass
  blow up}, J. Eur. Math. Soc. (JEMS) \textbf{17} (2015), no.~8, 1855--1925.

\bibitem{MMR3}
\bysame, \emph{Blow up for the critical g{K}d{V} equation {I}{I}{I}: exotic
  regimes}, Ann. Sc. Norm. Super. Pisa Cl. Sci. \textbf{14} (2015), no.~2,
  575--631.

\bibitem{0MRS}
J.~L. Marzuola, S.~Raynor, and G.~Simpson, \emph{A system of {O}{D}{E}s for a
  perturbation of a minimal mass soliton}, J. Nonlinear Sci. \textbf{20}
  (2010), no.~4, 425--461.

\bibitem{M3}
F.~Merle, \emph{Limit of the solution of a nonlinear {S}chr{\"o}dinger equation
  at blow-up time}, J. Funct. Anal. \textbf{84} (1989), no.~1, 201--214.

\bibitem{M2}
\bysame, \emph{Limit behavior of saturated approximations of nonlinear
  schr{\"o}dinger equation}, Comm. Math. Phys. \textbf{149} (1992), no.~2,
  377--414.

\bibitem{M6}
\bysame, \emph{On uniqueness and continuation properties after blow-up time of
  self-similar solutions of nonlinear {S}chr{\"o}dinger equation with critical
  exponent and critical mass}, Comm. Pure Appl. Math. \textbf{45} (1992),
  no.~2, 203--254.

\bibitem{M1}
\bysame, \emph{Existence of blow-up solutions in the energy space for the
  critical generalized {K}d{V} equation}, J. Amer. Math. Soc. \textbf{14}
  (2001), no.~3, 555--578.

\bibitem{MR3}
F.~Merle and P.~Rapha\"{e}l, \emph{Sharp upper bound on the blow-up rate for
  the critical nonlinear {S}chr{\"o}dinger equation}, Geom. Funct. Anal.
  \textbf{13} (2003), no.~3, 591--642.

\bibitem{MR2}
\bysame, \emph{On universality of blow-up profile for ${L}^2$ critical
  nonlinear {S}chr{\"o}dinger equation}, Invent. Math. \textbf{156} (2004),
  no.~3, 565--672.

\bibitem{MR1}
\bysame, \emph{The blow-up dynamic and upper bound on the blow-up rate for
  critical nonlinear {S}chr\"{o}dinger equation}, Ann. of Math. \textbf{161}
  (2005), no.~1, 157--220.

\bibitem{MR5}
\bysame, \emph{Profiles and quantization of the blow up mass for critical
  nonlinear {S}chr{\"o}dinger equation}, Comm. Math. Phys. \textbf{253} (2005),
  no.~3, 675--704.

\bibitem{MRS2}
F.~Merle, P.~Rapha{\"e}l, and J.~Szeftel, \emph{Stable self-similar blow-up
  dynamics for slightly ${L}^2$ super-critical {N}{L}{S} equations}, Geom.
  Funct. Anal. \textbf{20} (2010), no.~4, 1028--1071.

\bibitem{MRS3}
\bysame, \emph{The instability of {B}ourgain-{W}ang solutions for the ${L}^2$
  critical {N}{L}{S}}, Amer. J. Math. \textbf{135} (2013), no.~4, 967--1017.

\bibitem{NS1}
K.~Nakanishi and W.~Schlag, \emph{Global dynamics above the ground state energy
  for the focusing nonlinear {K}lein-{G}ordon equation}, J. Differential
  Equations \textbf{250} (2011), no.~5, 2299--2333.

\bibitem{NS3}
\bysame, \emph{Global dynamics above the ground state energy for the cubic
  {N}{L}{S} equation in 3{D}}, Calc. Var. Partial Differential Equations
  \textbf{44} (2012), no.~1-2, 1--45.

\bibitem{NS2}
\bysame, \emph{Global dynamics above the ground state for the nonlinear
  {K}lein-{G}ordon equation without a radial assumption}, Arch. Ration. Mech.
  Anal. \textbf{203} (2012), no.~3, 809--851.

\bibitem{MR4}
P.~Rapha{\"e}l, \emph{Stability of the $\log$-$\log$ bound for blow up
  solutions to the critical nonlinear {S}chr{\"o}dinger equation}, Math. Ann.
  \textbf{331} (2005), no.~3, 577--609.

\bibitem{Stru}
N.~Strunk, \emph{Well-posedness for the supercritical g{K}d{V} equation},
  Commun. Pure Appl. Anal. \textbf{13} (2012), no.~2, 527--542.

\bibitem{SS}
C.~Sulem and P.~L. Sulem, \emph{The nonlinear {S}chr{\"o}dinger equation:
  self-focusing and wave collapse}, vol. 139, Springer Science \& Business
  Media, 2007.

\bibitem{W1}
M.~I. Weinstein, \emph{Nonlinear {S}chr{\"o}dinger equations and sharp
  interpolation estimates}, Comm. Math. Phys. \textbf{87} (1983), no.~4,
  567--576.

\bibitem{W2}
\bysame, \emph{Modulational stability of ground states of nonlinear
  {S}chr{\"o}dinger equations}, SIAM J. Math. Anal. \textbf{16} (1985), no.~3,
  472--491.

\bibitem{W3}
\bysame, \emph{Lyapunov stability of ground states of nonlinear dispersive
  evolution equations}, Comm. Pure Appl. Math. \textbf{39} (1986), no.~1,
  51--67.

\end{thebibliography}
\end{document}